\documentclass[12pt,oneside,a4paper]{article}
\usepackage{amsmath,amsthm, amssymb}
\usepackage{hyperref}
\usepackage{cleveref}

\usepackage[left=3.0cm,right=3.0cm,top=2.0cm,bottom=2.0cm]{geometry}
\usepackage[utf8]{inputenc}
\usepackage{hyperref}
\usepackage{esint}
\usepackage{stmaryrd}
\newtheoremstyle{mytheorem}% name
  {3pt}%      Space above, empty = `usual value'
  {3pt}%      Space below
  {\itshape}% Body font
  {}%         Indent amount (empty = no indent, \parindent = para indent)
  {\bfseries}% Thm head font
  {.}%        Punctuation after thm head
  {1em}%     Space after thm head: " " = normal interword space;
  {}% Thm head spec1

\newtheorem{theoremA}{Theorem}

\newtheorem{theoremB}{Theorem}

\newtheorem{theoremC}{Theorem}

\newtheorem{definition}{Definition}[section]
\newtheorem{proposition}[definition]{Proposition}

\theoremstyle{mytheorem}
\newtheorem{theorem}[definition]{Theorem}

\newtheorem{lemma}[definition]{Lemma}

\newtheorem{corollary}[definition]{Corollary}

\newtheorem{question}[definition]{Question}
\newtheorem{conjecture}[definition]{Conjecture}

\theoremstyle{remark}
\newtheorem{remark}[definition]{Remark}

\theoremstyle{example}
\newtheorem{example}[definition]{Example}

\providecommand{\keywords}[1]
{
  \small	
  \textbf{\textit{Keywords---}} #1
}

\begin{document}
\title{ Scalar Curvature in Dimension 4 }
\author{Jialong Deng \thanks{Jialongdeng@gmail.com}}
\date{}
\maketitle
\begin{abstract}
We prove that every locally conformally flat metric on a closed, oriented hyperbolic $4$-manifold with scalar curvature bounded below by $-12$ satisfies Schoen’s conjecture. We also classify all closed, oriented Riemannian $4$-manifolds of positive scalar curvature that arise as total spaces of fiber bundles whose fiber and base are connected and have positive dimension and whose base is orientable. For a closed, locally conformally flat, scalar-flat manifold $(M^4,g)$ with $\pi_2(M^4) \neq 0$,  we show that the universal Riemannian cover $(\widetilde{M},\tilde{g})$ is homothetic to the standard product $\mathbb{H}^2 \times S^2$.  This affirmatively answers a question of M.~H.~Noronha. 
\end{abstract}

\begingroup
\par\smallskip
\noindent\keywords{Scalar curvature, Yamabe invariant, hyperbolic manifolds,
volume comparison, fiber bundles, locally conformally flat manifolds,
scalar-flat metrics, rigidity.}
\par
\endgroup

\tableofcontents
\section{Introduction}

In the small-radius expansion of geodesic-ball volume, a more
negative scalar curvature gives a larger first correction to the
Euclidean volume. Globally, however, the behavior is more subtle.
In the study of the Yamabe problem, Schoen \cite{MR994021}
conjectures that if $g$ is any Riemannian metric on a closed
hyperbolic $n$-manifold $(M^n,g_H)$ with $n \geq 3$ and
$\mathrm{Sc}_g \geq -n(n-1)$, then the total volume of $M$ with
respect to $g$ is at least its hyperbolic volume. In dimension~3,
Schoen's conjecture is studied by Anderson
\cite[page~132]{MR2213687}. In this paper, we give a partial
confirmation of Schoen's conjecture in dimension~4 in the
following cases.

\begin{theoremA}\label{A}
Let $(M^4,g_H)$ be a closed, oriented hyperbolic $4$-manifold.
Suppose $g_1$ is a Riemannian metric on $M^4$ with
$\mathrm{Sc}_{g_1} \geq -12$.
Assume in addition that one of the following holds:
\begin{enumerate}
    \item[(i)] $g_1$ is self-dual or anti-self-dual;
    \item[(ii)] the Yamabe metric $g \in [g_1]$ with
    $\mathrm{Sc}_g = -12$ satisfies one of the following conditions:
    \begin{enumerate}
        \item[(a)] $g$ is Einstein;
        \item[(b)] for every point $x \in M^4$, there exists
        $r(x)>0$ such that, for all $0<r\leq r(x)$,
        \[
        \mathrm{Vol}_g(B^g_r(x))
        \geq
        \mathrm{Vol}_{g_H}(B^{g_H}_r(x)).
        \]
    \end{enumerate}
\end{enumerate}
Then
$\mathrm{Vol}_{g_1}(M^4)\geq\mathrm{Vol}_{g_H}(M^4)$.
Moreover, if equality holds, then $g_1$ is isometric to $g_H$.
\end{theoremA}

The strategy of the proof is to reduce the problem to the Yamabe
metric. Every conformal class on $M^4$ has negative Yamabe constant
and hence contains a unique Yamabe metric with scalar curvature
$-12$.
First, we compare the volume of the given metric with that of this
normalized Yamabe metric in its conformal class.
Second, we compare the Yamabe metric with the hyperbolic metric.
In case~(i), the signature formula and $\sigma(M^4)=0$
imply local conformal flatness,
while in case~(ii)(b), the small-ball volume expansion yields an
inequality between the Weyl tensor and the traceless Ricci tensor.
These observations allow us to apply the Gauss--Bonnet--Chern
formula in these two cases. In the Einstein case, we use the
Besson--Courtois--Gallot theorem \cite{zbMATH00847576}. In dimension~$4$, case~(ii)(a) is contained in case~(i): by
Remark~\ref{Einstein hyperbolic}, an Einstein metric with scalar
curvature $-12$ on $M^4$ is isometric to $g_H$, so $g_1$ is locally
conformally flat. We keep case~(ii)(a) because
Proposition~\ref{Besson} treats it in every dimension $n\geq3$.

Since any finitely presented group $\pi$ can be realized as the
fundamental group of a closed Riemannian $4$-manifold with positive
scalar curvature \cite{MR936805}, additional assumptions are
necessary in order to classify closed $4$-manifolds admitting
metrics of positive scalar curvature. For example, when the
fundamental group is trivial, the combined results of Freedman
and Donaldson, together with scalar curvature obstructions,
determine the possible homeomorphism types of closed, simply
connected smooth $4$-manifolds admitting metrics of positive scalar
curvature. Independently, by imposing the assumption that the
manifold is the total space of a fiber bundle whose fiber and base
have positive dimension and whose base is orientable, one obtains
the following classification.

\begin{theoremB}\label{B}
Let $M^4$ be a closed, oriented, smooth $4$-manifold which is the
total space of a smooth fiber bundle whose fiber and base are connected and have
positive dimension and whose base is orientable. Then:
\begin{enumerate}
    \item[(1)] If either the fiber or the base is $S^2$, then
    $M^4$ admits a Riemannian metric of positive scalar curvature.
    \item[(2)] If either the fiber or the base is a closed,
    oriented, connected $3$-manifold $N^3$, then $M^4$ admits a
    Riemannian metric of positive scalar curvature if and only
    if $N^3$ does.
\end{enumerate}
\end{theoremB}

The remaining case is when both the fiber and the base are
oriented surfaces of positive genus. Then $M^4$ is aspherical
and does not admit a metric of positive scalar curvature.
When the fiber or the base is $S^2$, the proof relies on results
concerning diffeomorphism groups of closed surfaces.

When the fiber is $N^3$, so that $M^4$ is a mapping torus, the
`only if' direction uses stable minimal hypersurfaces representing
the fiber class, together with the classification of $3$-manifolds
admitting metrics of positive scalar curvature.
For the `if' direction, we use the fact that the space of positive
scalar curvature metrics on $N^3$, when nonempty, is path connected.
The main result of Bamler and Kleiner
\cite{2019arXiv190908710B} shows that this space is contractible.
Its path connectedness gives an isotopy between a positive scalar
curvature metric $h$ and its pullback $\varphi^*h$ by the monodromy
$\varphi$. The resulting positive scalar curvature concordance
has product collars and therefore glues to a smooth metric of
positive scalar curvature on the mapping torus.

When the base is $N^3$, a metric of positive scalar curvature on
$N^3$ gives one on $M^4$ by shrinking the circle fibers, and the
`only if' direction follows from Hanke's argument as presented by
Kumar and Sen \cite{zbMATH08137326}%
\footnote{I thank Thomas Schick for pointing out this paper to me.}.

On the other hand, closed locally conformally flat Riemannian
$4$-manifolds with positive scalar curvature are classified by
Hamilton, Chen, Tang, and Zhu \cite{zbMATH06081388} using the
Ricci flow. A natural next step is to classify $4$-manifolds
admitting locally conformally flat scalar-flat metrics.
One motivation for this study is that the Levi-Civita connection
of such a metric on a closed, oriented $4$-manifold attains an
absolute minimum of the Yang--Mills functional among metric
connections on its tangent bundle.

For a closed $4$-manifold $(M^4,g)$ with a locally conformally flat
scalar-flat metric, its Riemannian universal cover need not be
isometric, up to homothety, to either the Euclidean space
$\mathbb{E}^4$ or the product
$(\mathbb{H}^2 \times S^2,g_H \oplus g_{st})$.
In 1993, Noronha \cite{MR1235219} asks whether, under the
additional assumption that $\pi_2(M^4)\neq0$, the Riemannian
universal cover is isometric, up to homothety, to
$(\mathbb{H}^2 \times S^2,g_H \oplus g_{st})$.
She shows, using Bochner-type formulas, that if $b_2(M^4)>0$
and $g$ is not flat, then $M^4$ is isometrically covered, up to
homothety, by
$(\mathbb{H}^2 \times S^2,g_H \oplus g_{st})$.
We provide a positive answer to Noronha's question in the
following theorem.

\begin{theoremC}[Theorem~\ref{the universal cover}]\label{D}
Let $(M^4,g)$ be a closed, locally conformally flat, and scalar-flat
$4$-manifold with $\pi_2(M^4)\neq0$.
Then its Riemannian universal cover $(\widetilde M,\tilde g)$
is isometric to
$(\mathbb{H}^2 \times S^2,g_H \oplus g_{st})$
up to homothety.
\end{theoremC}

Using the Schoen--Yau Liouville theorem
(Theorem~\ref{Liouville theorem}; see \cite{MR4836036}) and the
compact-quotient argument of Schoen and Yau
\cite[proof of Theorem~4.1]{zbMATH04075988}, we identify the
universal cover conformally with
$\Omega(\Gamma)=S^4\setminus\Lambda(\Gamma)$, where $\Gamma$ is
the faithful Kleinian holonomy image of $\pi_1(M^4)$ and
$M^4\cong\Omega(\Gamma)/\Gamma$.
The condition $\pi_2(M^4)\neq0$ implies that $\Gamma$ is
non-elementary. The compactness of $\Omega(\Gamma)/\Gamma$
makes $\Gamma$ conformally finite, and the bound
$\dim_{\mathcal H}(\Lambda(\Gamma))\leq1$ of Schoen and Yau
\cite[Theorem~4.7]{zbMATH04075988} then allows us to apply the
theorem of Chang--Qing--Yang \cite{MR2070141} to conclude that
$\Gamma$ is geometrically finite.

The condition $\pi_2(M^4)\neq0$ also gives
$H_2(\widetilde M;\mathbb Z)\neq0$ by the Hurewicz theorem.
Through \v{C}ech--Alexander duality, this implies that the
topological dimension satisfies
$\dim(\Lambda(\Gamma))\geq1$. Consequently,
$\dim(\Lambda(\Gamma))=\dim_{\mathcal H}(\Lambda(\Gamma))=1.$
The Patterson--Sullivan theorem \cite{zbMATH03903608} also gives
$\delta(\Gamma)=1$, where $\delta(\Gamma)$ is the critical
exponent of $\Gamma$.
The equality of the topological and Hausdorff dimensions,
together with geometric finiteness, allows us to apply
Kapovich's theorem \cite{MR2491697}, which shows that
$\Lambda(\Gamma)$ is a round circle.
Lemma~\ref{Stereographic} then identifies the universal cover
conformally with
$(\mathbb{H}^2 \times S^2,g_H \oplus g_{st})$.
Finally, using the compactness of $M^4$, Lemma~\ref{iso}
gives the required isometry up to homothety.

The argument also applies to closed, locally conformally flat,
scalar-flat manifolds $M^{2n}$, $n\geq2$, if one replaces the
condition $\pi_2(M^4)\neq0$ by
$\widetilde H_n(\widetilde M;\mathbb Z)\neq0$.
In this case, the Riemannian universal cover is isometric,
up to homothety, to
$(\mathbb{H}^n \times S^n,g_H \oplus g_{st})$;
see Corollary~\ref{2n}.

\paragraph*{Organization of the Paper.}
In Section~\ref{2}, we partially verify Schoen's conjecture in
dimension~4. In Section~\ref{3}, we characterize closed, oriented
$4$-manifolds admitting metrics of positive scalar curvature
that arise as total spaces of fiber bundles under the assumptions
of Theorem~\ref{B}. In Section~\ref{4}, we discuss results and
questions related to the rigidity of positive and negative scalar
curvature. In Section~\ref{5}, we characterize the Riemannian
universal cover of closed, locally conformally flat, scalar-flat
$2n$-manifolds under a homological condition on the universal cover.

\paragraph*{Acknowledgment.}          
The author acknowledges support from the Oberwolfach Leibniz Fellows programme (MFO), the YMSC Overseas Shuimu Scholarship, the Simons Center for Geometry and Physics, and ICMS Edinburgh (workshops on Geometric Measure Theory on Metric Spaces with Applications to Physics and Geometry and Geometric Moduli Spaces, respectively). I  thank Gerhard Huisken for discussions on the Ricci flow. The results were presented in seminars in 2025.  This work originates from a broader project initiated during my postdoctoral stay at the Yau Center. During that time, this work was also supported by NSFC  12401063 and partially by  NSFC  12271284. I am deeply grateful to Shing-Tung Yau and Akito Futaki for their trust \footnote{In particular, their trust gave me the hope and strength to survive some of the hardest times of my life — during the strict COVID-19 prevention measures in Beijing (2022), when daily life itself was extremely difficult, and  during my recovery from ACL reconstruction surgery in 2024.} and support,  which allowed me to pursue independent research.

\paragraph*{Disclosure on LLM assistance.}
LLMs were used to correct typographical and grammatical errors and polish the presentation of the author's draft. That draft appeared on arXiv and in the Oberwolfach Preprints series in December 2025 and was not generated by an LLM. The mathematical content was developed by the author, who takes responsibility for the final manuscript.

\section{Negative scalar curvature and smooth invariants}\label{2}

The Bishop–Gromov volume comparison theorem for Ricci curvature bounded below has many applications.  
It is natural to seek analogous volume comparison results for scalar curvature bounded below.  
For negative scalar curvature, the following volume comparison holds within a fixed conformal class. It provides the first step in our approach to Schoen’s conjecture.

\begin{proposition} \label{ncsc}
Let $(M^n,g)$ ($n\geq 3$)  be a closed, connected Riemannian manifold with constant negative scalar curvature  $\mathrm{Sc}_g \equiv -R<0$. If $g'$ is a metric conformal to $g$ with $\mathrm{Sc}_{g'}\geq -R$,  then $\mathrm{Vol}_{g'}(M)\geq \mathrm{Vol}_g(M)$. If $\mathrm{Vol}_{g'}(M)=\mathrm{Vol}_g(M)$, then $g'=g$.
\end{proposition}

\begin{proof}
Let $g'=e^{2f}g$, where $f\in C^{\infty}(M)$. Then, the scalar curvature of $g'$ is 
\begin{equation*}
\mathrm{Sc}_{g'}=e^{-2f}(\mathrm{Sc}_g+2(n-1)\Delta_g f -(n-2)(n-1)|\nabla f|_g^2).
\end{equation*}
 Throughout this paper, we use the convention
$\Delta_g=-g^{ij}\nabla_i\nabla_j$
for the positive-spectrum Laplacian on $(M,g)$, where $\nabla$ denotes the Levi--Civita connection of $g$.  The conditions $\mathrm{Sc}_g=-R$ and  $\mathrm{Sc}_{g'}\geq -R$ imply that 
\begin{equation*}
2(n-1)\Delta_g f -(n-2)(n-1)|\nabla f|_g^2 \geq -R(e^{2f}-1).
\end{equation*}
Rearranging gives
\begin{equation} \label{1}
-\Delta_g f +\frac{n-2}{2}|\nabla f|_g^2 \leq \frac{R}{2(n-1)}(e^{2f}-1).
\end{equation}
Since $M^n$ is closed and $R>0$,  after integrating the inequality (\ref{1}), we have 
\begin{equation*}
\int_M e^{2f}dv_g\geq \mathrm{Vol}_{g}(M).
\end{equation*}
Since $n\geq 3$, we can use  H\"older's inequality to obtain the following inequality:
\begin{equation*}
\int_M e^{2f}\times 1 dv_g \leq (\int_M e^{2f\times \frac{n}{2}}dv_g)^{\frac{2}{n}}(\int_M 1^{\frac{n}{n-2}}dv_g)^{\frac{n-2}{n}}.
\end{equation*}
Thus, we have 
\begin{equation*}
[\mathrm{Vol}_{g'}(M):=\int_M e^{2f\times \frac{n}{2}}dv_g]^{\frac{2}{n}} \geq \frac{\int_M e^{2f}dv_g}{\mathrm{Vol}_{g}(M)^{\frac{n-2}{n}}}\geq \mathrm{Vol}_{g}(M)^{\frac{2}{n}}.
\end{equation*}
Therefore, we have $\mathrm{Vol}_{g'}(M)\geq \mathrm{Vol}_{g}(M)$.

 If $\mathrm{Vol}_{g'}(M)=\mathrm{Vol}_g(M)$, all inequalities above are equalities. Equality in Hölder’s inequality holds if and only if $e^{nf}=C$ almost everywhere on $M$ for some constant $C$. Thus, $f$ is constant almost everywhere on $M$. Since $f\in C^{\infty}(M)$ and $M$ is connected, $f$ is constant  on $M$. As $\mathrm{Vol}_{g'}(M)=\mathrm{Vol}_{g}(M)$, we have $f\equiv 0$ on $M$. Thus $g'=g$.
\end{proof}

Let $S(g)$ be the normalized Hilbert--Einstein functional, 
\begin{equation*}
S(g):=\frac{1}{\mathrm{Vol}_g(M^n)^\frac{n-2}{n}}\int_{M^n} \mathrm{Sc}_gdv_g.
\end{equation*}
Let $[g]$ be the conformal class of $g$. The Yamabe constant $ Y(M^n,[g])$ of $(M^n,g)$ ($n\geq 3$)  is given by
\begin{equation*}
Y(M^n,[g]):= \inf_{g'\in[g]} S(g').
\end{equation*}
  The Yamabe–Trudinger–Aubin–Schoen theorem guarantees that this infimum is attained. Computing the first variation of \( S(g) \) shows that \( g' \in [g] \) is a critical point of \( S(g) \) restricted to \([g]\) if and only if the scalar curvature of \( g' \) is constant. A metric that attains the infimum in the Yamabe problem within a conformal class is called a Yamabe metric. If $Y(M^n, [g])>0$, the Yamabe metric may not be unique up to scaling.  On the other hand, if  $ Y(M^n,[g])\leq 0$, then the Yamabe metric $g'$ is unique up to a scaling factor.

\begin{proposition}\label{Yamabe}
Let $g'$ be a metric conformal to $g$ with $\mathrm{Sc}_{g'}=\mathrm{Sc}_{g} \equiv -R < 0$ ($R>0$) on a closed $n$-manifold ($n\geq 3$). Then $g'=g$.
\end{proposition}

\begin{proof}
Let $g'=u^{\frac{4}{n-2}}g$ for some $0<u\in C^{\infty}(M)$.     Since $\mathrm{Sc}_{g'}=\mathrm{Sc}_{g} \equiv -R$,  we have 
\begin{align*}
\frac{4(n-1)}{(n-2)}\Delta_{g}u=\mathrm{Sc}_{g'}u^{\frac{n+2}{n-2}}-\mathrm{Sc}_{g}u=-R u(u^{\frac{4}{n-2}}-1).                                  
\end{align*}
 That is, 
\begin{align*}
\Delta_{g}u=\frac{(n-2)}{4(n-1)} R u(1-u^{\frac{4}{n-2}}).                                  
\end{align*}

Since \( M \) is closed, \( u \) attains a maximum at some point \( p \in M \). Under the positive-spectrum convention, at a maximum point, the Laplacian is non-negative, i.e.,
\begin{align*}
\Delta_{g}u(p)=\frac{(n-2)}{4(n-1)} Ru(p)[1-u^{\frac{4}{n-2}}(p)]\geq 0.                                  
\end{align*}
Furthermore, $R>0$ and $u>0$, so $\max_M u=u(p)\leq1$. At a minimum point $q$, $\Delta_g u(q)\leq0$, so $u(q)\geq1$. Hence $u\equiv1$, and therefore $g'=g$.
\end{proof}

 In addition, if $Y(M^n,[g])\leq 0$ ($n\geq 3$),  then for every $g'\in [g]$, we have 
\begin{equation*}
(\min \mathrm{Sc}_{g'})\mathrm{Vol}_{g'}(M)^{\frac{2}{n}}\leq Y(M,[g])\leq (\max  \mathrm{Sc}_{g'})\mathrm{Vol}_{g'}(M)^{\frac{2}{n}}.
\end{equation*}
Indeed, $Y(M,[g])\leq S(g')$ and $S(g')\leq(\max\mathrm{Sc}_{g'})\mathrm{Vol}_{g'}(M)^{\frac{2}{n}}$. For the left inequality, if $Y(M,[g])<0$, then $\min\mathrm{Sc}_{g'}<0$, and Proposition~\ref{ncsc}, applied to $g'$ and to the Yamabe metric of $[g]$ rescaled to scalar curvature $\min\mathrm{Sc}_{g'}$, gives the claim; if $Y(M,[g])=0$, then $\min\mathrm{Sc}_{g'}\leq0$.
 If $\mathrm{Sc}_{g'}$ is constant, both inequalities are equalities; in particular, if $\mathrm{Sc}_{g'}=-n(n-1)$, then 
\begin{equation*}
\mathrm{Vol}_{g'}(M)=(\frac{ Y(M,[g])}{-n(n-1)})^\frac{n}{2}.
\end{equation*}
Thus, we  have \[ \sup_{g'\in[g]}(\min \mathrm{Sc}_{g'})\mathrm{Vol}_{g'}(M)^{\frac{2}{n}}=Y(M,[g]).\] This means that any metric \( g' \) with constant non-positive scalar curvature is a Yamabe metric in its conformal class \( [g] \), provided that \( Y(M, [g]) \leq 0 \). However, if $n\geq 3$ and $Y(M^n, [g])>0$, then \[ \sup_{g'\in[g]}(\min \mathrm{Sc}_{g'})\mathrm{Vol}_{g'}(M)^{\frac{2}{n}}=+\infty. \]

For $Y(M,[g])\leq 0$, assume $[g]$  contains an Einstein metric $g_E$. If $g'\in[g]$ has constant scalar curvature and $\mathrm{Vol}_{g'}(M)=\mathrm{Vol}_{g_E}(M)$, then  $g'$ is isometric to the Einstein metric $g_E$.

The Yamabe invariant of a closed smooth $n$-manifold $M^n$ ($n\geq 3$) is defined as 
\[Y(M^n):= \sup_{[g]} Y(M^n,[g]),\] where the supremum is taken over all conformal classes of smooth metrics on $M^n$. Since $Y(S^n)=Y(S^n,[g_{st}])$ and  $Y(T^n)=0$ for an $n$-torus $T^n$,  it is natural to ask whether the hyperbolic metric achieves  $Y(M^n)$ for a closed manifold $M^n$ admitting a hyperbolic metric.  Because a closed hyperbolic manifold $M^n$ carries no metrics with positive scalar curvature (see \cite{zbMATH07375613} and the references therein), its Yamabe invariant satisfies $Y(M^n)\leq 0$. Let $X^n$ be a closed $n$-manifold ($n\geq 3$) with $Y(X^n)\leq 0$. Then we have \[ \sup_{g}(\min \mathrm{Sc}_{g})\mathrm{Vol}_{g}(X^n)^{\frac{2}{n}}=Y(X^n)\leq 0,\]
where $g$ runs over all smooth Riemannian metrics on $X^n$. Furthermore, if $Y(X^n)=Y(X^n,[g])$ for some $[g]$, then $[g]$ contains an Einstein metric $g_E$, and with this metric $Y(X^n)=Y(X^n,[g])=\mathrm{Sc}_{g_E}\mathrm{Vol}_{g_E}(X)^{2/n}.$

While studying the Yamabe invariant, Schoen conjectures in \cite{MR994021} that if a Riemannian metric $g$ on a  closed hyperbolic $n$-manifold  ($M^n,g_H$) ($n\geq 3$) satisfies $\mathrm{Sc}_g\geq -n(n-1)$,   then $\mathrm{Vol}_g(M)\geq \mathrm{Vol}_{g_H}(M)$.  If Schoen’s conjecture were true, then  the Yamabe invariant of $M^n$ would actually be attained by the hyperbolic metric.    In the proof of Theorem~\ref{attained}, we show that any scalar-flat metric on $M$ must be Ricci-flat, contradicting the topology of a closed hyperbolic manifold.   Since $M^n$ also admits no metric of positive scalar curvature, $Y(M^n,[g'])<0$ for every Riemannian metric $g'$ on $M^n$, and the Yamabe--Trudinger--Aubin--Schoen theorem guarantees that $g'$ can be conformally deformed to a unique metric $g$ with constant scalar curvature $-n(n-1)$. In particular,
$Y(M^n, [g']) = -n(n-1)\left( \mathrm{Vol}_{g}(M) \right)^{2/n}$. If Schoen’s conjecture holds, then for every conformal class \( [g'] \) on \( M^n \), we have $Y(M^n, [g']) \leq Y(M^n, [g_H])$.  Therefore, the Yamabe invariant \( Y(M^n) \) is indeed realized by the hyperbolic metric. 

However, to the author's knowledge, it is not known whether there exists a closed hyperbolic $n$-manifold $M^n$ with $Y(M^n)<0$ for $n\geq 4$. Furthermore, even if \( Y(M^n) < 0 \) can be verified, it is not clear to the author whether the inequality \( Y(M^n \# N^n) < 0 \) holds for an arbitrary closed smooth \( n \)-manifold \( N^n \), although it is known that \( Y(M^n \# N^n) \leq 0 \).

Studying  critical points  of  several functionals on the space of Riemannian metrics,  Besson--Courtois--Gallot \cite{zbMATH04192565} show that  if $g$ is $C^2$-close enough to $g_H$, then  the condition  $\mathrm{Sc}_g\geq -n(n-1)$ implies $\mathrm{Vol}_g(M^n)\geq \mathrm{Vol}_{g_H}(M^n)$, and  the equality holds if and only if $g$ is isometric to $g_H$.  Later, using the barycenter method, they \cite{zbMATH00847576} show that the hyperbolic metric minimizes $\mathrm{Ent}(M^n,g)\,\mathrm{Vol}_g(M^n)^{\frac{1}{n}}$ among all metrics $g$ on $M^n$, and that the minimizer is unique up to homothety. If $\mathrm{Ric}_g\geq -(n-1)$, the Bishop--Gromov inequality gives $\mathrm{Ent}(M^n,g)\leq n-1$. Consequently,  we have \( \mathrm{Vol}_g(M^n) \geq \mathrm{Vol}_{g_H}(M^n) \)  and  the equality holds if and only if $g$ is isometric to $g_H$.

\begin{proposition}\label{Besson}
Let $(M^n,g_H)$, $n \ge 3$, be a closed hyperbolic manifold, and let $g$ be another metric on $M^n$ satisfying
$\mathrm{Sc}_g \ge -n(n-1)$. If there exists an Einstein metric $g_E$ with negative scalar curvature in the conformal class $[g]$, then
$\mathrm{Vol}_g(M^n) \ge \mathrm{Vol}_{g_H}(M^n)$ and equality holds if and only if $g$ is isometric to $g_H$.
\end{proposition}

\begin{proof}
If such an Einstein metric $g_E$ exists, then it is also a Yamabe metric.
After rescaling, we may assume that $\mathrm{Ric}_{g_E} = -(n-1).$
By Proposition~\ref{ncsc}, this implies
$\mathrm{Vol}_g(M^n) \ge \mathrm{Vol}_{g_E}(M^n).$
Moreover, by the theorem of Besson--Courtois--Gallot \cite{zbMATH00847576},
$\mathrm{Vol}_{g_E}(M^n) \ge \mathrm{Vol}_{g_H}(M^n).$
Combining these inequalities yields
$\mathrm{Vol}_g(M^n) \ge \mathrm{Vol}_{g_H}(M^n)$.  If $\mathrm{Vol}_g(M^n)=\mathrm{Vol}_{g_H}(M^n)$, both inequalities are equalities, so $g=g_E$ by the rigidity part of Proposition~\ref{ncsc}. Since $\mathrm{Ric}_{g_E}=-(n-1)$ gives $\mathrm{Ent}(M^n,g_E)\le n-1$, the inequality $\mathrm{Ent}(M^n,g_E)^n\,\mathrm{Vol}_{g_E}(M^n)\geq (n-1)^n\,\mathrm{Vol}_{g_H}(M^n)$ of \cite{zbMATH00847576} is an equality; hence $g_E$ is homothetic, and by the volume normalization isometric, to $g_H$. 
\end{proof}

 However, $\mathrm{Sc}_g\ge -n(n-1)$ does not imply the entropy estimate
$\mathrm{Ent}(M^n,g)\le n-1$ used in the preceding argument.  Indeed, Kazaras–Song–Xu \cite{KazarasSongXu2025} show that there exists a closed hyperbolic $3$-manifold $M^3$ admitting a Riemannian metric $g$ with $\mathrm{Sc}_g \geq -6$ and volume entropy strictly larger than $2$. 

In  dimension 3, Anderson \cite[page~132]{MR2213687} considers Schoen's conjecture. 
He studies the metric invariant $S_{-}(g) := \mathrm{Sc}_{\min}(g)\, \mathrm{Vol}_g(M^3)^{2/3}$
on a closed $3$-manifold $(M^3,g)$, where $\mathrm{Sc}_{\min}(g) := \min_{x \in M^3} \mathrm{Sc}_g(x).$ 
Perelman's work implies that $S_{-}(g(t))$ is monotone non-decreasing in $t$ along the Ricci flow with surgery, as long as $\mathrm{Sc}_{\min}\le0$. Then, for a closed three-manifold $M^3$ admitting a hyperbolic metric $g_H$, Anderson's idea is to use Perelman’s monotonicity formula for the Ricci flow with
surgery  to show that $S_{-}(M^3) := \sup_g S_{-}(g) = -6\, \mathrm{Vol}_{g_H}(M^3)^{2/3},$
where the supremum is taken over all smooth metrics $g$ on $M^3$. Then, $\mathrm{Sc}_g\geq -6$ implies  $\mathrm{Vol}_{g}(M^3) \geq \mathrm{Vol}_{g_H}(M^3)$. 
  One application of the verified Schoen conjecture in three dimensions is Reiris’s new proof that a complete asymptotically flat Riemannian three-manifold with non-negative scalar curvature cannot have negative ADM mass \cite{zbMATH05607319}. It is not clear to the author whether the verification of Schoen's conjecture in higher dimensions can be used to establish the positivity of the ADM mass.

 \begin{remark}
For a closed $n$-dimensional manifold $M^n$ admitting a hyperbolic metric, the normalized Ricci flow
\[
\frac{\partial g(t)}{\partial t} = -2\, \mathrm{Ric}_{g(t)} + \frac{2}{n\, \mathrm{Vol}_{g(t)}(M^n)} \left( \int_{M^n} \mathrm{Sc}_{g(t)} \, dv_{g(t)} \right) g(t)
\]
preserves the volume of $M^n$, and the minimum of the scalar curvature, $\mathrm{Sc}_{\min}(g(t))$, is non-decreasing along the flow. We may rescale the initial metric to normalize the volume to one, allowing us to define
$S_{-}(g(t)) := \mathrm{Sc}_{\min}(g(t))$. Along the normalized Ricci flow, the quantity \( S_{-}(g(t)) \) is also non-decreasing. However, due to the lack of a general long-time existence result for the normalized Ricci flow in higher dimensions, extending Perelman's argument to this setting remains a significant challenge.

 \end{remark}

 Let $\mathcal{M}$ be the space of smooth Riemannian metrics on a closed, connected, and oriented $n$-manifold  $M^n$ ($n\geq 3$). Then $\mathcal{M}$ is a smooth infinite-dimensional manifold. The space  $\mathcal{M}_{-1}:=\{g \in \mathcal{M}| \mathrm{Sc}_g=-1\}$ is a non-empty closed smooth infinite-dimensional sub-manifold of $\mathcal{M}$.  If $M^n$ ($n\geq 3$) admits a hyperbolic metric, the space $\mathcal{M}_{-1}$ is contractible.

Thus,   one may also consider the \emph{conformal Ricci flow} equations on a closed $n$-dimensional manifold $M^n$ with $n \geq 3$ admitting a hyperbolic metric:
\begin{align*}
\frac{\partial g(t)}{\partial t} + 2\left( \mathrm{Ric}_{g(t)} + \frac{1}{n}g(t) \right) &= -p(t)g(t), \\
\mathrm{Sc}_{g(t)} &= -1,
\end{align*}
where $p(t)\in C^{\infty}(M^n)$ for $t \in [0, T)$, $0<T\leq +\infty$. The existence of a short-time solution is shown by Fischer and Moncrief \cite[Theorem 4.1]{zbMATH02063627}.

 A metric $g$ with $\mathrm{Sc}_g=-1$ is an equilibrium point of the conformal Ricci flow equations if and only if $g$ is an Einstein metric with the negative Einstein constant $-1/n$. The conformal Ricci flow preserves the scalar curvature condition $\mathrm{Sc}_{g(t)} = -1$, and the volume is non-increasing along the flow. In particular,  if  the initial metric is not Einstein and the flow exists for all \(t\geq0\) and converges smoothly to an Einstein metric \(g_\infty\), then its volume converges monotonically to \(\mathrm{Vol}_{g_\infty}(M)\).

  \begin{remark}
 We propose the following strategy: Based on the above discussion, we may consider a non-Einstein metric \( g \) on \( M^n \) with scalar curvature \( \mathrm{Sc}_g = -n(n-1) \), and then rescale it so that its scalar curvature is $-1$. Using this as the initial metric, we can evolve the conformal Ricci flow. If the flow converges to an Einstein metric with Einstein constant \(-1/n \), then the volume decreases monotonically along the flow and converges to that of the Einstein metric. After rescaling the limiting Einstein metric to scalar curvature $-n(n-1)$, we use Proposition~\ref{Besson} to bound its volume below by
$\mathrm{Vol}_{g_H}(M)$.  However, the long-time existence of the flow remains an open problem, and the norm of the Riemann curvature tensor may blow up either in finite time or as time approaches infinity.
 \end{remark}

 For a closed smooth manifold $M^n$ ($n\geq 3$), Perelman also defines a differential-topological invariant,
 \begin{equation*}
 \bar{\lambda}(M)=\sup_g \lambda_g(\mathrm{Vol}_g(M))^{2/n},
 \end{equation*}
  where the supremum is taken over all smooth metrics $g$ on $M$ and  $\lambda_g$ is  the least eigenvalue  of the elliptic operator $4\Delta_g + \mathrm{Sc}_g$. Here  $\Delta_g = d^*d = -\nabla \cdot \nabla$ is the
positive-spectrum Laplace--Beltrami operator associated with $g$. In fact, $\lambda_g$  can also be defined using Rayleigh quotients as follows:
 \begin{equation*}
 \lambda_g:= \inf_{u\in C^{\infty}(M)\setminus\{0\}} \frac{\int_M (\mathrm{Sc}_g u^2+ 4|\nabla u|^2)dv_g}{\int_Mu^2dv_g}.
 \end{equation*}

Perelman defines \( \bar{\lambda}(M^n) \) because he observes that whenever \( \lambda_{g_t} (\mathrm{Vol}_{g_t}(M))^{2/n} \leq 0 \), the scale-invariant quantity \(  \lambda_{g_t} (\mathrm{Vol}_{g_t}(M))^{2/n}  \) is non-decreasing under the Ricci flow \( g_t \).  For a $3$-manifold $M^3$  admitting no metric of positive scalar curvature,  Perelman shows that  $(-\bar{\lambda}(M^3))^{3/2}$ is proportional to the minimal volume of the manifold.  If $M^n$ ($n\geq 3$) admits a Riemannian metric with positive scalar curvature, then $\bar{\lambda}(M^n)=+\infty$.  If $Y(M^n)\leq 0$, Akutagawa, Ishida, and LeBrun \cite{zbMATH05130246} show that $\bar{\lambda}(M^n)=Y(M^n)$. 

On the other hand, for a connected closed  smooth $n$-manifold $M^n$,  we have the following smooth invariant 
\begin{equation*}
L(M^n)=\inf_g \int_{M^n} |\mathrm{Sc}_g|^{n/2}dv_g,
\end{equation*} 
 where the infimum is taken over all smooth Riemannian metrics $g$ on  $M^n$. The values of \(L(M)\) are unbounded as \(M\) ranges over closed, simply connected smooth \(4\)-manifolds, whereas \(L(M)=0\) for all closed, simply connected \( n \)-manifolds with \( n \geq 3 \) and \( n \neq 4 \).   If a manifold \( M^n \) ($n\geq 3$)  satisfies \( Y(M^n) \geq 0 \), then \( L(M^n)=0\).  However, if a manifold \( M^n \) ($n\geq 3$)  satisfies \( Y(M^n) \leq 0 \), then \( L(M^n) = |Y(M^n)|^{n/2} \).

For a closed manifold $M^n$ admitting a hyperbolic metric $g_H$,   confirming  Schoen's conjecture in higher dimensions,  namely 
\begin{equation*}
Y(M)=-n(n-1)(\mathrm{Vol}_{g_H}(M^n))^{2/n},
\end{equation*}
  implies that
 \begin{equation*}
 \bar{\lambda}(M^n)=-n(n-1)(\mathrm{Vol}_{g_H}(M^n))^{2/n}
 \end{equation*}
 and 
 \begin{equation*}
 L(M)=[n(n-1)]^{n/2}\mathrm{Vol}_{g_H}(M^n).
 \end{equation*}
Thus Schoen's conjecture would express hyperbolic volume in terms of the Yamabe invariant, Perelman's invariant, and $L(M)$. Here, we present a partial result in dimension 4.

The Weyl tensor $W$ in  an oriented Riemannian $4$-manifold can be  decomposed as $W=W^+ + W^-$, where $W^+$ is the self-dual  Weyl tensor and $W^-$ is the anti-self-dual Weyl tensor. Thus, we have $|W|^2=|W^+|^2+|W^-|^2$. Here we view $W$ as a section of $ (T^*M)^{\otimes4}$ and we use the tensor norm   $|W|^2:=W_{ijkl}W^{ijkl}$. Indeed, the squared norm \( | \cdot |^2 \) of a \((k,l)\)-tensor field \( T \) in this paper, expressed in coordinates, is given by  
 \[|T|^2:=g^{i_1m_1}\cdots g^{i_lm_l}g_{j_1n_1}\cdots g_{j_kn_k}T^{j_1 \cdots j_k}_{i_1 \cdots i_l} T^{n_1 \cdots n_k}_{m_1 \cdots m_l}.\]

 An oriented Riemannian $4$-manifold  $(M^4,g)$ is called self-dual (resp. anti-self-dual) if $W^- \equiv 0$ (resp. $W^+\equiv0$) on $M^4$.

\begin{theorem}\label{LCF}
Let \( g \) be a self-dual or anti-self-dual metric with scalar curvature \( \mathrm{Sc}_g \geq -12 \) on a closed oriented hyperbolic $4$-manifold \( (M^4, g_H) \). Then $\mathrm{Vol}_g(M^4)\geq \mathrm{Vol}_{g_H}(M^4)$. If $\mathrm{Vol}_g(M^4)= \mathrm{Vol}_{g_H}(M^4)$, then $
(M,g)$ is isometric to $(M,g_H).$
\end{theorem}

\begin{proof}

Since  the signature $\sigma(M^4)$ of  a closed, oriented Riemannian $4$-manifold $(M^4,g)$   is given by
\begin{equation*}
\sigma(M)=\frac{1}{48\pi^2}\int_M (|W^+|^2-|W^-|^2)dv_g,
\end{equation*}
and   $\sigma(M)=0$ for a  closed  manifold admitting  a hyperbolic metric,   the self-duality (or anti-self-duality) of $g$ implies that $g$ is locally conformally flat, i.e., $W_g \equiv 0$.

Since \( Y(M, [g]) < 0 \), the Yamabe--Trudinger--Aubin--Schoen theorem guarantees that \( g \) can be conformally deformed to a metric \( g_1 \)  with $\mathrm{Sc}_{g_1} = -12$. Then, by Proposition \ref{ncsc}, it follows that $\mathrm{Vol}_{g}(M) \geq \mathrm{Vol}_{g_1}(M)$. As local conformal flatness is conformally invariant, we  have $W_{g_1} \equiv 0$.

       The Gauss--Bonnet--Chern formula for a closed, oriented Riemannian $4$-manifold $(M^4,g)$ is given by
\begin{equation}\label{4d GBC}
32\pi^2 \chi(M)=\frac{1}{6}\int_M\mathrm{Sc}_g^2dv_g -2\int_M |\mathring{\mathrm{Ric}}_g|^2dv_g + \int_M|W_g|^2dv_g,
\end{equation}
where $\mathring{\mathrm{Ric}}$ denotes the traceless Ricci tensor, i.e., $\mathring{\mathrm{Ric}} = \mathrm{Ric}_g - (\mathrm{Sc}_g/4 \cdot g
)$.  Applying it to $(M^4,g_1)$, we obtain

\begin{equation}\label{4D}
24\mathrm{Vol}_{g_H}(M)=32\pi^2 \chi(M)= 24\mathrm{Vol}_{g_1}(M)-2\int_M|\mathring{\mathrm{Ric}}_{g_1}|^2dv_{g_1}.
\end{equation}

Thus,  the inequality $\mathrm{Vol}_{g_1}(M) \geq \mathrm{Vol}_{g_H}(M)$ holds. Hence, $\mathrm{Vol}_g(M)\geq   \mathrm{Vol}_{g_1}(M) \geq      \mathrm{Vol}_{g_H}(M)$. 

Now assume $\mathrm{Vol}_g(M)= \mathrm{Vol}_{g_H}(M)$. Then we have $\mathrm{Vol}_{g_1}(M)= \mathrm{Vol}_{g_H}(M)$. This implies $\mathring{\mathrm{Ric}}_{g_1}\equiv 0$ through the integral of the Gauss--Bonnet--Chern formula. Since also $W_{g_1}\equiv 0$ and $\mathrm{Sc}_{g_1}\equiv-12$, the metric $g_1$ has constant sectional curvature $-1$, and the Mostow rigidity theorem implies that $g_1$ is isometric to $g_H$. Finally, the rigidity part of Proposition \ref{ncsc} concludes the proof.
\end{proof}

\begin{remark}\label{Einstein hyperbolic}
Let $M^4$ be a closed connected manifold admitting  a hyperbolic metric. 
For an Einstein metric $g$ with $\mathrm{Ric}_g=-3g$ on $M^4$, the Gauss--Bonnet--Chern formula (\ref{4d GBC}) gives
\[
24\mathrm{Vol}_{g_H}(M)
=32\pi^2\chi(M)
=24\mathrm{Vol}_g(M)+\int_M|W_g|^2\,dv_g.
\]
The theorem of Besson--Courtois--Gallot gives $\mathrm{Vol}_g(M)\geq\mathrm{Vol}_{g_H}(M)$, so the displayed identity forces $W_g\equiv0$. Together with $\mathrm{Ric}_g=-3g$, this implies that $g$ has constant sectional curvature $-1$. Thus, by Mostow rigidity, $g$ is isometric to $g_H$.
\end{remark}

\begin{remark}
Compactness is necessary in the proof, as there exist complete noncompact manifolds with negative scalar curvature but positive Yamabe constant. For example, a complete, noncompact, simply connected, locally conformally flat Riemannian manifold has positive Yamabe constant, and there exists  a complete scalar-flat Riemannian manifold  with  positive Yamabe invariant \cite{zbMATH00124945}.
\end{remark}

\begin{remark}
There exists a closed $4$-manifold admitting a hyperbolic metric such that the space of conformal classes containing self-dual or anti-self-dual metrics, modulo the action of the diffeomorphism group, has nonzero topological dimension \cite{zbMATH04086464}. To the author’s knowledge, there are no known examples of closed hyperbolic $4$-manifolds for which this space is a single point.
\end{remark}

\begin{remark}
For a smooth 4-manifold \( M^4 \) admitting a self-dual metric with constant scalar curvature, Myers and Periwal develop a theory of topological gravity in \cite{1991NuPhB.361..290M}. Moreover, this theory yields non-trivial smooth invariants of the manifold.
\end{remark}

For a closed oriented hyperbolic $4$-manifold $M$, the identity $\sigma(M)=0$ gives $b_2(M)=2b_2^+(M)$. Hence $\chi(M)=2-2b_1(M)+b_2(M)$ is even. Applying (\ref{4d GBC}) to $g_H$ also gives
\[
\chi(M)=\frac{3}{4\pi^2}\mathrm{Vol}_{g_H}(M)>0.
\]
Thus $\chi(M)$ is a positive even integer, and consequently
\[
\mathrm{Vol}_{g_H}(M)\in\frac{8\pi^2}{3}\mathbb{Z}_{>0}.
\]

A related rigidity statement holds for locally conformally flat metrics of positive constant scalar curvature in dimensions four and six.

\begin{corollary} \label{LCF on spheres}
Let $n=4$ or $6$, and let $(M^n,g)$ be a closed, oriented, connected,  locally conformally flat manifold  with $\mathrm{Sc}_g = n(n-1)$  and $H_1(M^n;\mathbb{Z})=0$. Then $(M^n,g)$  is isometric to  $(S^n,g_{st})$.
\end{corollary}

\begin{proof}
For a closed, oriented, locally conformally flat Riemannian \( 2m \)-manifold \( (N^{2m}, g) \), the Bochner--Weitzenböck formula for \( m \)-forms reads
\[
\Delta = \nabla^{\ast} \nabla + \frac{m}{2(2m-1)}\, \mathrm{Sc}_g,
\]
where $\Delta$ is the Hodge Laplacian on \(m\)-forms.
Hence, if $(N^{2m},g)$ has positive scalar curvature, it admits no nonzero harmonic $m$-forms; that is, $H_m(N^{2m};\mathbb{R})=0$. Applying this to $M^n$ gives $b_2(M^4)=0$ when $n=4$ and $b_3(M^6)=0$ when $n=6$. Since $H_1(M^n;\mathbb{Z})=0$, Poincar\'e duality then yields
\[
\chi(M^4)=2,
\qquad
\chi(M^6)=2+2b_2(M^6)\geq2.
\]

Using the Gauss--Bonnet--Chern formula, Gursky \cite[Theorem~A]{zbMATH00750638} shows that a closed locally conformally flat manifold of dimension $n=4$ or $6$ with nonnegative scalar curvature satisfies $\chi(M)\leq2$, with equality if and only if it is conformally equivalent to the round sphere. The preceding calculation therefore forces $\chi(M)=2$, so $(M,g)$ is conformally equivalent to $(S^n,g_{st})$.

Since $\mathrm{Sc}_g=n(n-1)$, Obata's theorem~\cite[Proposition~6.1]{zbMATH03374588} implies that $(M,g)$ is isometric to $(S^n,g_{st})$.
\end{proof}

\begin{proof}[An alternative proof in dimension 4]
Retain the hypotheses of Corollary~\(\ref{LCF on spheres}\) with \(n=4\). Let \( M^4 \) be a closed Riemannian 4-manifold equipped with a locally conformally flat metric of positive scalar curvature. In dimension 4, such a metric necessarily has positive isotropic curvature. By applying the Ricci flow with surgery, Hamilton–Chen–Tang–Zhu classify all compact 4-manifolds with positive isotropic curvature~\cite{zbMATH06081388}. According to their result, every closed $4$-manifold $M$ admitting an LCF metric with PSC is diffeomorphic to one of the following spaces:$$S^4, \quad \mathbb{RP}^4, \quad (S^3 \times \mathbb{R})/G,$$or a connected sum of such spaces, where $G$ is a cocompact, fixed-point-free, discrete subgroup of $\mathrm{Isom}(S^3 \times \mathbb{R})$.

Since $H_1(X\#Y;\mathbb{Z})\cong H_1(X;\mathbb{Z})\oplus H_1(Y;\mathbb{Z})$ for closed $4$-manifolds $X$ and $Y$, the assumption $H_1(M;\mathbb{Z})=0$ implies that every summand in this decomposition of $M$ has vanishing first homology. This excludes $\mathbb{RP}^4$, since $H_1(\mathbb{RP}^4;\mathbb{Z})\cong\mathbb{Z}/2$. It also excludes every quotient $(S^3\times\mathbb{R})/G$. Indeed, $\mathrm{Isom}(S^3\times\mathbb{R})=\mathrm{O}(4)\times\mathrm{Isom}(\mathbb{R})$ with $\mathrm{O}(4)$ compact, and $G$ is discrete and cocompact, so the projection of $G$ to $\mathrm{Isom}(\mathbb{R})$ has finite kernel and its image is an infinite discrete subgroup of $\mathrm{Isom}(\mathbb{R})$, that is, an infinite cyclic or an infinite dihedral group. Since $S^3\times\mathbb{R}$ is simply connected, $H_1((S^3\times\mathbb{R})/G;\mathbb{Z})\cong G^{\mathrm{ab}}$, which surjects onto $\mathbb{Z}$ or onto $\mathbb{Z}/2\oplus\mathbb{Z}/2$ and is therefore nonzero. Hence every summand is $S^4$, and $M\cong S^4$.

By Kuiper’s theorem, a closed, simply connected, locally conformally flat manifold is conformally equivalent to the standard sphere. Since the metric has constant scalar curvature \(\mathrm{Sc}_g \equiv 12\),  Obata's theorem \cite[Proposition 6.1]{zbMATH03374588} implies that it must be isometric to the round sphere \( (S^4, g_{st}) \). This provides an alternative proof in dimension 4 that avoids the use of Gursky’s result.
\end{proof}

\begin{remark}
The condition \( H_1(M^n; \mathbb{Z}) = 0 \) in Corollary~\ref{LCF on spheres} cannot be replaced by the  assumption \( \chi(M^n) = 0 \). For instance, consider the product manifold \( S^1 \times S^3 \) equipped with the metric $g_\lambda = \lambda\,g_{S^1} \oplus \tfrac{1}{2}g_{st}$,
where \( \lambda > 0 \), \( g_{S^1} \) is the round metric on \( S^1 \), and \( g_{st} \) is the standard round metric on \( S^3 \). Then \( (S^1 \times S^3, g_\lambda) \) is a locally conformally flat manifold with constant scalar curvature \( \mathrm{Sc}_{g_\lambda} = 12 \) and Euler characteristic \( \chi(S^1 \times S^3) = 0 \). However, the volume of \( (S^1 \times S^3, g_\lambda) \) can be made arbitrarily small or arbitrarily large by taking \( \lambda \to 0 \) or \( \lambda \to \infty \), respectively, while keeping the scalar curvature fixed. 

\end{remark}

\begin{remark}
Let \( (M^4, g) \) be a smooth Riemannian 4-manifold. The biorthogonal curvature of a 2-plane \( \sigma \subset T_pM \) is defined by
\[
K^{\perp}_g(\sigma) := \frac{1}{2} \left( K_g(\sigma) + K_g(\sigma^{\perp}) \right),
\]
where \( K_g(\sigma) \) denotes the sectional curvature of \( \sigma \), and \( \sigma^{\perp} \) is the orthogonal complement of \( \sigma \) in \( T_pM \). The existence of a metric with positive biorthogonal curvature does not imply the existence of a metric with positive Ricci curvature. For instance, the product manifold \( S^3 \times S^1 \), equipped with the product metric \( g_{st} \oplus g \) (where \( g_{st} \) is the standard round metric on \( S^3 \) and \( g \) is the standard metric on \( S^1 \)), has positive biorthogonal curvature. However, \( S^3 \times S^1 \) does not admit any metric with positive Ricci curvature, as its fundamental group is infinite. Moreover, \( K^{\perp}_g \geq k \) implies \( \mathrm{Sc}_g \geq 12k \). Thus, the condition \( \mathrm{Sc}_g = 12 \) in Corollary~\ref{LCF on spheres} can be replaced by \( K^{\perp}_g = 1 \) for $n=4$.
\end{remark}

\begin{remark} \label{4D bio}
A compact Riemannian 4-manifold \( (M^4, g) \) is locally conformally flat if and only if \( K^{\perp}_g = \mathrm{Sc}_g / 12 \). In particular, Schur's lemma does not extend to biorthogonal curvature, as the pointwise constancy of \( K^{\perp}_g \) does not imply that it is constant globally. Indeed, by a theorem of Kulkarni, a Riemannian \( n \)-manifold ($n\geq 4$) is locally conformally flat if and only if, at every point,
every orthonormal four-frame satisfies
\[K_{12} + K_{34} = K_{13} + K_{24} = K_{14} + K_{23},\]
where \( K_{ij} \) denotes the sectional curvature of the plane spanned by \( e_i \) and \( e_j \). A straightforward computation shows that this condition is equivalent to \( K^{\perp}_g = \mathrm{Sc}_g / 12 \) in dimension 4. 
\end{remark}

\begin{remark}
The hyperbolic reference metric in Theorem~\ref{LCF} cannot, in
general, be replaced by a metric $h$ satisfying
$K_h^\perp\equiv-1$. Indeed, let $(N^3,g_H)$ be a closed, oriented
hyperbolic $3$-manifold, and let $g_{S^1}$ be the standard flat
metric on $S^1$. On $M=N^3\times S^1$, consider
\[
h=\frac12 g_H\oplus g_{S^1},
\qquad
g_\lambda=\frac12 g_H\oplus\lambda g_{S^1},
\qquad \lambda>0.
\]
These metrics are locally conformally flat and satisfy
\[
\mathrm{Sc}_h=\mathrm{Sc}_{g_\lambda}=-12,
\qquad
K_h^\perp=K_{g_\lambda}^\perp=-1.
\]
However,
$\mathrm{Vol}_{g_\lambda}(M)=\sqrt{\lambda}\,\mathrm{Vol}_h(M),$
so the inequality
$\mathrm{Vol}_{g_\lambda}(M)\geq\mathrm{Vol}_h(M)$
fails whenever $0<\lambda<1$. In fact, these volumes tend to zero
as $\lambda\to0$. Moreover, $\pi_1(M)=\pi_1(N)\times\mathbb{Z}$ contains a subgroup isomorphic to $\mathbb{Z}^2$, so Preissmann's theorem excludes a
hyperbolic metric on $M$. Since $\chi(M)=0$, the Gauss--Bonnet
argument used in Theorem~\ref{LCF} also yields no positive volume
lower bound in this example.
\end{remark}

Proposition~\ref{ncsc} reduces the volume comparison to the Yamabe representative with scalar curvature $-12$. We now impose a local comparison condition on this representative. This condition is inspired by one of the synthetic definitions of a lower scalar curvature bound in \cite{zbMATH07342230}.

\begin{theorem} \label{local}
Let \( g \) be a Riemannian metric with constant scalar curvature \( \mathrm{Sc}_g \equiv -12 \) on a closed, oriented hyperbolic 4-manifold \( (M^4, g_H) \). Assume that for every point \( x \in M^4 \), there exists a sufficiently small radius \( r(x) > 0 \) such that for all \( 0 < r \leq r(x) \), the volumes of the geodesic \( r \)-balls \( B_r(x) \) satisfy
\[
\mathrm{Vol}_g(B^g_r(x)) \geq \mathrm{Vol}_{g_H}(B^{g_H}_r(x)).
\]
Then
$ \mathrm{Vol}_g(M^4) \geq \mathrm{Vol}_{g_H}(M^4)$. Moreover, if equality holds, i.e., \( \mathrm{Vol}_g(M^4) = \mathrm{Vol}_{g_H}(M^4) \), then $(M^4,g)$ is isometric to $(M^4,g_H)$.

\end{theorem}

\begin{proof}
The asymptotic expansion of the volume of a small geodesic \( r \)-ball \( B_r(x) \) in a Riemannian $n$-manifold \( (M^n, g) \) is given by~\cite[Theorem 3.3]{zbMATH03667577}:

\begin{align*}
\mathrm{Vol}_g(B_r(x)) &= \mathrm{Vol}_{\mathbb{E}}(B_r) \Bigg[ 1 - \frac{\mathrm{Sc}_g(x)}{6(n+2)} r^2 \\
&\quad + \frac{1}{360(n+2)(n+4)} \left( -3|\mathrm{Rm}_g|^2 + 8|\mathrm{Ric}_g|^2 + 5\,\mathrm{Sc}_g^2 + 18\,\Delta_g \mathrm{Sc}_g \right) r^4 + O(r^6) \Bigg],
\end{align*}
as $r \to 0$, where $\mathrm{Rm}_g$ is the Riemannian curvature tensor of the metric $g$. All curvature quantities in the expansion are evaluated at $x$. On the other hand, the norm of the Riemann curvature tensor is given by
  \begin{equation*}
 |\mathrm{Rm}_g|^2= |W_g|^2+\frac{4}{n-2}|\mathrm{Ric}_g|^2 - \frac{2}{(n-1)(n-2)}\mathrm{Sc}_g^2,
  \end{equation*}
 or,  
  \begin{equation*}
 |\mathrm{Rm}_g|^2= |W_g|^2+\frac{4}{n-2}|\mathring{\mathrm{Ric}}_g|^2 + \frac{2}{n(n-1)}\mathrm{Sc}_g^2,
  \end{equation*}
  since \[|\mathring{\mathrm{Ric}}_g|^2=|\mathrm{Ric}_g|^2-\frac{\mathrm{Sc}_g^2}{n}.\]

 In the case of \( (M^4, g_H) \), we have the identity
\[
-3|\mathrm{Rm}_{g_H}|^2 + 8|\mathrm{Ric}_{g_H}|^2 = 216.
\]
By assumption, \( \mathrm{Sc}_g \equiv \mathrm{Sc}_{g_H} \equiv -12 \), and for every \( x \in M^4 \) and every \( 0 < r \leq r(x) \),
\[
\mathrm{Vol}_g(B^g_r(x)) \geq \mathrm{Vol}_{g_H}(B^{g_H}_r(x)).
\]
Since the scalar curvatures are equal constants, the expansions of the two sides have the same leading term and the same coefficient of \( r^2 \); in the coefficients of \( r^4 \), the terms \( 5\,\mathrm{Sc}^2 \) agree and the Laplacian terms vanish. Subtracting the two expansions, dividing by \( \mathrm{Vol}_{\mathbb{E}}(B_r)\,r^4/(360\cdot 6\cdot 8) \), and letting \( r \to 0 \), we obtain at every \( x \in M^4 \)
\[
-3|\mathrm{Rm}_g|^2 + 8|\mathrm{Ric}_g|^2 \geq -3|\mathrm{Rm}_{g_H}|^2 + 8|\mathrm{Ric}_{g_H}|^2 = 216.
\]

The curvature decomposition and $\mathrm{Sc}_g=-12$ give
\[
-3|\mathrm{Rm}_g|^2+8|\mathrm{Ric}_g|^2
=-3|W_g|^2+2|\mathring{\mathrm{Ric}}_g|^2+216.
\]
Combining this identity with the preceding lower bound yields
\[
|W_g|^2\leq\frac{2}{3}|\mathring{\mathrm{Ric}}_g|^2.
\]

The Gauss--Bonnet--Chern formula (\ref{4d GBC}) implies 
\begin{equation*}
32\pi^2 \chi(M)\leq \frac{1}{6}\int_M\mathrm{Sc}_g^2dv_g -2\int_M |\mathring{\mathrm{Ric}}_g|^2dv_g + \int_M\frac{2}{3}\,| \mathring{\mathrm{Ric}}_g|^{2}dv_g.
\end{equation*}
Thus

\begin{equation*}
24\mathrm{Vol}_{g_H}(M)=32\pi^2 \chi(M) \leq  24\mathrm{Vol}_{g}(M)- \frac{4}{3}\int_M|\mathring{\mathrm{Ric}}_g|^2dv_g.
\end{equation*}

Consequently, we have $ \mathrm{Vol}_g(M^4) \geq \mathrm{Vol}_{g_H}(M^4)$.  If the equality holds, then $\mathring{\mathrm{Ric}}_g\equiv 0$; hence $W_g\equiv 0$ by $|W_{g}|^{2}\le\frac{2}{3}|\mathring{\mathrm{Ric}}_g|^{2}$. Thus $g$ has constant sectional curvature $-1$, and $g$ is isometric to $g_H$ by the Mostow rigidity theorem.    

\end{proof}

\begin{remark}
If $g$ satisfies only the local volume comparison condition
$\mathrm{Vol}_g(B^g_r(x)) \ge \mathrm{Vol}_{g_H}(B^{g_H}_r(x))$ for all sufficiently small $r>0$,
then the global volume comparison
$\mathrm{Vol}_g(M^4) \ge \mathrm{Vol}_{g_H}(M^4)$ need not hold in general.
 As a counterexample, consider the metric \( g = g_H/2 \). Under this scaling, the scalar curvature becomes \( \mathrm{Sc}_g = -24 \).  Locally, for small \( r \), the geodesic balls still satisfy
$\mathrm{Vol}_g(B^g_r(x)) \geq \mathrm{Vol}_{g_H}(B^{g_H}_r(x))$. However, the total volume scales as
$\mathrm{Vol}_g(M^4) =  \mathrm{Vol}_{g_H}(M^4)/4$.
\end{remark}

The argument used in the proof of Theorem~\ref{local} also applies to the case of the \(4\)-sphere.

\begin{corollary}
Let \( g \) be a Riemannian metric with constant scalar curvature \( \mathrm{Sc}_g \equiv 12 \) on  $S^4$. Assume that for every point \( x \in S^4 \), there exists a sufficiently small radius \( r(x) > 0 \) such that for all \( 0 < r \leq r(x) \), the volumes of the geodesic \( r \)-balls \( B_r(x) \) satisfy
\[
\mathrm{Vol}_g(B^g_r(x)) \geq \mathrm{Vol}_{g_{st}}(B^{g_{st}}_r(x)).
\]
Then
$ \mathrm{Vol}_g(S^4) \geq \mathrm{Vol}_{g_{st}}(S^4)$. Moreover, if equality holds, i.e., \( \mathrm{Vol}_g(S^4) = \mathrm{Vol}_{g_{st}}(S^4) \), then \( g \) is isometric to $g_{st}$.

\end{corollary}

\begin{proof}

The assumption of the corollary still implies $|W_{g}|^{2}\;\le\;\frac{2}{3}\,| \mathring{\mathrm{Ric}}_g|^{2}$.
Combining it with  $\chi(S^4)=2$,  $\mathrm{Vol}_{g_{st}}(S^4)=8\pi^2/3$, $\mathrm{Sc}_g \equiv 12$, and the Gauss--Bonnet--Chern formula (\ref{4d GBC}), we have:  
\begin{equation*}
24\mathrm{Vol}_{g_{st}}(S^4)= 32\pi^2 \chi(S^4)\leq  24\mathrm{Vol}_g(S^4)  - \frac{4}{3}\int_{S^4}|\mathring{\mathrm{Ric}}_g|^2dv_g. 
\end{equation*} 
Hence, $\mathrm{Vol}_g(S^4)\geq\mathrm{Vol}_{g_{st}}(S^4)$, and equality implies $\mathring{\mathrm{Ric}}_g\equiv0$. The pinching inequality then gives $W_g\equiv0$. Since $\mathrm{Sc}_g=12$, the vanishing of $W_g$ and $\mathring{\mathrm{Ric}}_g$ implies that $g$ has constant sectional curvature $1$. Therefore, $g$ is isometric to the standard metric $g_{st}$ on $S^4$.

\end{proof}

\begin{proof}[Proof of Theorem \ref{A}]
Case (i) is Theorem \ref{LCF} applied to $g_1$, and case (ii)(a) is Proposition \ref{Besson} applied to $g_1$. In case (ii)(b), Proposition \ref{ncsc} and Theorem \ref{local} give $\mathrm{Vol}_{g_1}(M^4)\geq \mathrm{Vol}_{g}(M^4)\geq \mathrm{Vol}_{g_H}(M^4)$. If $\mathrm{Vol}_{g_1}(M^4)=\mathrm{Vol}_{g_H}(M^4)$, both inequalities are equalities; the rigidity part of Proposition \ref{ncsc} gives $g_1=g$, and the rigidity part of Theorem \ref{local} shows that $g$ is isometric to $g_H$.
\end{proof}

\section{The topology of  \(4\)-manifolds with PSC}\label{3}

The classifications of closed oriented $3$-manifolds admitting positive scalar curvature and of locally conformally flat $4$-manifolds with positive scalar curvature motivate the corresponding question in dimension four: which closed oriented $4$-manifolds admit positive scalar curvature, up to homeomorphism?

However, Carr \cite{MR936805} shows that for any finitely presented group \(\pi\), there exists a closed Riemannian \(4\)-manifold \(M\) with positive scalar curvature such that \(\pi_1(M)=\pi\). The $4$-manifolds in the proof of Markov's theorem, like Carr's, are obtained from $S^4$ by surgeries of codimension at least $3$, so they carry metrics of positive scalar curvature. Hence there is no algorithm deciding whether two closed \(4\)-manifolds with positive scalar curvature are homeomorphic. In fact, Carr’s result extends to all dimensions \(n\geq4\): for every finitely presented group \(\pi\) and \(n\ge4\), there exists a closed smooth stably parallelizable Riemannian \(n\)-manifold with positive scalar curvature and fundamental group \(\pi\); moreover, if \(\pi\) contains a subgroup of index two, there also exists a closed smooth non-orientable Riemannian \(n\)-manifold with positive scalar curvature and fundamental group \(\pi\) \cite[Theorems 6 and 7]{zbMATH07432160}.  

To circumvent this group-theoretic obstruction, one may instead aim to classify those Riemannian \(n\)-manifolds with a fixed fundamental group. The simplest case is the simply connected one. For \(n\ge5\), it is known that a closed, simply connected smooth manifold \(M^n\) admits a metric of positive scalar curvature if and only if \(M^n\) is non-spin, or it is spin and its \(\alpha\)-invariant vanishes in \(\mathrm{KO}_n\). In particular, since \(\mathrm{KO}_n = 0\) for \(n \equiv 3,5,6,7 \pmod{8}\) by Bott periodicity for the $\mathrm{KO}_n$-theory of a point, every closed, simply connected smooth \(n\)-manifold in these dimensions admits a metric of positive scalar curvature. Consequently, in these cases, the geometric classification problem reduces entirely to a topological one.

The classical results of Freedman and Donaldson in  dimension 4 imply that if 
\( M^4 \) is a smoothable, closed, simply connected, topological \(4\)-manifold, 
then \( M^4 \) is homeomorphic to one of the following:
$S^4$, $\#^m \mathbb{CP}^2 \,\#^n \overline{\mathbb{CP}}^{\,2}$, or $  \#^{\pm m} M_{E_8} \,\#^n (S^2 \times S^2),$
for some integers \( m,n \geq 0 \), where \(M_{E_8}\) denotes the topological $4$-manifold whose intersection form is given by the $E_8$ lattice.

 For a simply connected closed \(4\)-manifold \(M^4\), it admits a spin structure if and only if its intersection form is even. The intersection form of $\#^{\pm m} M_{E_8} \,\#^n (S^2 \times S^2)$
 is even, with rank \(8m+2n\) and signature \(\pm 8m\). Furthermore, by the Hirzebruch signature theorem, the \(\hat{A}\)-genus of a spin \(4\)-manifold equals \(-\tfrac{1}{8}\) times its signature. Hence, by the Lichnerowicz formula, for \(m\geq 1\) no smooth manifold homeomorphic to \(\#^{\pm m} M_{E_8} \,\#^n (S^2 \times S^2)\) admits a metric of positive scalar curvature.  

Therefore, a closed simply connected \(4\)-manifold admitting a metric of positive scalar curvature is homeomorphic to one of the following:
\[
S^4, \quad 
\#^m \mathbb{CP}^2 \,\#^n \overline{\mathbb{CP}}^{\,2}, \quad 
\#^n (S^2 \times S^2),
\]
or a connected sum of these manifolds (see \cite[Theorem 7.8]{MR3450199}).  The existence of a metric of positive scalar curvature is not
determined by the homeomorphism type alone. Indeed,
Teicher~\cite[Theorem~5.8]{MR1720873} constructs simply connected
spin complex surfaces of general type with vanishing signature.
By Freedman's classification, these surfaces are homeomorphic to
$\#^k(S^2\times S^2)$ for suitable $k$. Although the standard smooth
connected sums admit metrics of positive scalar curvature,
Teicher's surfaces do not, since they are of general type.

We next consider closed $4$-manifolds with nontrivial fundamental group, beginning with the finite case.

\begin{proposition}\label{finite}
Let $M^4$ be a connected, closed, non-simply connected Riemannian $4$-manifold with positive scalar curvature and finite fundamental group. Then:  
\begin{enumerate}
\item[(i)] If $M$ is spin, then its universal cover $\tilde{M}$ is homeomorphic to 
$\#^{k} (S^2 \times S^2)$ for some integer $k \ge 1$, where $k = \tfrac{1}{2} b_2(\tilde{M})$.

\item[(ii)] If $M$ is non-spin, then $\tilde{M}$ is homeomorphic to one of the following:
\begin{enumerate}
\item[(a)] $\#^{k} \mathbb{CP}^2 \,\#^{l} \overline{\mathbb{CP}}^{\,2}$, 
 for some integers $k,l\geq0$ with $k+l\geq1$, if $\tilde{M}$ is non-spin;
\item[(b)] $\#^{m} (S^2 \times S^2)$, for some integer $m \ge 0$, if $\tilde{M}$ is spin.
\end{enumerate}
\end{enumerate}
\end{proposition}

\begin{proof}
If $M$ is spin, then its universal cover $\tilde{M}$ is also spin  and the intersection form of $\tilde{M}$ is even. The lifted positive scalar curvature metric implies that the signature of $\tilde{M}$ vanishes. Hence, $\tilde{M}$ is homeomorphic to $\#^{k} (S^2 \times S^2)$.  
   If $k = 0$, then $2=\chi(\tilde{M})=|\pi_1(M)|\,\chi(M)$ with $|\pi_1(M)|\geq 2$, so $\chi(M)=1$. A closed orientable $4$-manifold with finite fundamental group has $\chi=2+b_2\geq 2$, so $M$ is non-orientable and hence not spin, a contradiction. Therefore, $k \ge 1$.

Now suppose $M$ is non-spin.  
If $\tilde{M}$ is spin, then as above $\tilde{M}$ is homeomorphic to $\#^{m} (S^2 \times S^2)$ for some $m \ge 0$.  
If $\tilde{M}$ is non-spin, then its intersection form is odd. If the form is definite, Donaldson's theorem diagonalizes it; if it is indefinite, the classification of odd indefinite unimodular forms gives the same conclusion. Freedman's theorem then implies that $\tilde{M}$ is homeomorphic to $\#^k\mathbb{CP}^2\,\#^l\overline{\mathbb{CP}}^{\,2}$ for some integers $k,l\geq0$ with $k+l\geq1$.
\end{proof}

\begin{remark}
A manifold satisfying condition (i) or (ii) in Proposition~\ref{finite} does not necessarily admit a metric of positive scalar curvature. There exist closed smooth  $4$-manifolds $M$ (both spin and non-spin) with finite cyclic fundamental group $\pi_1(M)$ that do not admit metrics of positive scalar curvature, while their universal covers do admit such metrics (see \cite{zbMATH01985408,zbMATH02078321}).
\end{remark}

\begin{proposition}\label{cobordant}
Every closed $4$-manifold is cobordant to one admitting a metric of positive scalar curvature.
\end{proposition}

\begin{proof}
Rohlin shows that the oriented cobordism group $\Omega_4^{\mathrm{SO}} \cong \mathbb{Z}$ is generated by $\mathbb{CP}^2$, with the isomorphism given by the signature $\sigma(M)$. Hence, two closed oriented $4$-manifolds are cobordant if and only if they have the same signature.  
Manifolds with $\sigma(M) = k > 0$ are cobordant to $\#^k \mathbb{CP}^2$, those with $\sigma(M) = -k < 0$ to $\#^k \overline{\mathbb{CP}}^{\,2}$, and those with $\sigma(M) = 0$ to $S^4$.  
Since these manifolds all admit metrics of positive scalar curvature, it follows that every oriented cobordism class contains a representative with positive scalar curvature.  

For unoriented manifolds, the unoriented cobordism group $\Omega_4^{\mathrm{O}} \cong \mathbb{Z}/2\mathbb{Z} \oplus \mathbb{Z}/2\mathbb{Z}$ admits an analogous description.  
It is generated by the standard representatives $\mathbb{RP}^4$ and $\mathbb{RP}^2 \times \mathbb{RP}^2$, both of which carry metrics of positive scalar curvature.  
The trivial class (the zero element) is represented by $S^4$, or equivalently by any bounding manifold.  
The fourth class is represented by $\mathbb{RP}^4\#(\mathbb{RP}^2\times\mathbb{RP}^2)$, which carries a metric of positive scalar curvature since positive scalar curvature is preserved under connected sums.  

Therefore, in both the oriented and unoriented settings, every $4$-dimensional cobordism class contains a representative admitting a metric of positive scalar curvature.
\end{proof}

\begin{remark}
 Proposition~\ref{cobordant} extends to higher dimensions in the sense that the generators of the \( n \)-dimensional oriented and unoriented cobordism groups admit metrics of positive scalar curvature.  For instance, the Wu manifold \( \mathrm{SU}(3)/\mathrm{SO}(3) \) generates the group \( \Omega_5^{\mathrm{SO}} \). It is non-spin and simply connected. Thus,  it carries a metric of positive scalar curvature. Consequently, any closed, connected, oriented \( 5 \)-manifold is cobordant to a manifold admitting a metric of positive scalar curvature.
\end{remark}

For a closed oriented $4$-manifold $M$, let $Q_M$ denote the intersection form on $H_2(M;\mathbb{Z})/\mathrm{Tor}$. Positive scalar curvature constrains its stable type as follows.

\begin{proposition}\label{intersection form}
Let $M$ be a closed oriented $4$-manifold admitting a metric of positive scalar curvature. Then, after stabilization by connected sums with copies of $S^2 \times S^2$ (which adjoin hyperbolic pairs $\mathcal{H}$ to the intersection form), the intersection form of $M$ falls into one of the following stable types:
\begin{itemize}
    \item If the form is \emph{odd} (i.e., not all quadratic values are even), it is stably isomorphic to 
   $ m \langle 1 \rangle \oplus n \langle -1 \rangle$
    for some nonnegative integers $m,n$, with $m - n = \sigma(M)$, where $\sigma(M)$ denotes the signature of $M$.
    \item If the form is \emph{even} (i.e., all quadratic values are even), it is stably isomorphic to 
    $ k \mathcal{H}$  for some nonnegative integer $k$, and $\sigma(M) = 0$.
\end{itemize}
\end{proposition}

\begin{proof}
Taking the connected sum $M \# (S^2 \times S^2)$ adjoins a hyperbolic pair 
\[
\mathcal{H} = \begin{pmatrix} 0 & 1 \\ 1 & 0 \end{pmatrix}
\]
to the intersection form $Q_M$ of $M$, preserving both the signature $\sigma(M)$ and the parity (even or odd) of the form. After one stabilization the form is indefinite. If $Q_M$ is odd, the classification of indefinite unimodular forms gives $Q_M \oplus \mathcal{H} \cong m \langle 1 \rangle \oplus n \langle -1 \rangle$ with $m-n=\sigma(M)$.

Suppose $Q_M$ is even. By the Wu formula, $\langle w_2(M),x\rangle\equiv x\cdot x\equiv0\pmod2$ for every $x\in H_2(M;\mathbb{Z})$, so $w_2(M)$ lies in $\mathrm{Ext}(H_1(M;\mathbb{Z}),\mathbb{Z}/2)\subset H^2(M;\mathbb{Z}/2)$.

Write $H_1(M;\mathbb{Z})=F\oplus T$, with $F$ free and $T$ finite, and let $p:\widehat{M}\to M$ be the finite cover associated with the kernel of $\pi_1(M)\to H_1(M;\mathbb{Z})\to T$. The homomorphism $p_*:H_1(\widehat{M};\mathbb{Z})\to H_1(M;\mathbb{Z})$ factors through the free summand $F$. Since $\mathrm{Ext}(F,\mathbb{Z}/2)=0$, naturality of the universal coefficient sequence implies that $p^*$ annihilates the $\mathrm{Ext}$-term. Thus $w_2(\widehat{M})=p^*w_2(M)=0$.

The lifted metric has positive scalar curvature, so the Lichnerowicz formula gives $\hat{A}(\widehat{M})=0$. Consequently,
\[
\sigma(\widehat{M})=-8\hat{A}(\widehat{M})=0,
\qquad
\sigma(M)=\frac{\sigma(\widehat{M})}{|T|}=0.
\]
Thus $Q_M\oplus\mathcal{H}$ is even, indefinite, and of signature zero, hence isomorphic to $l\mathcal{H}$. In particular, the intersection form of a closed oriented $4$-manifold with positive scalar curvature is never even with nonzero signature; an $E_8$-summand can occur in odd forms, e.g. $Q_{\#^8\mathbb{CP}^2\#\overline{\mathbb{CP}}^{\,2}}\cong E_8\oplus\langle-1\rangle$.
\end{proof}

\begin{remark}
The intersection-form conditions in Proposition~\ref{intersection form} are not sufficient for the existence of positive scalar curvature: $T^4$ has intersection form $3\mathcal{H}$, but admits no such metric.
\end{remark}

A closed smooth manifold \( M \) is called a \emph{splitting manifold} if it is diffeomorphic to a product $M_1 \times M_2$ of smooth manifolds \( M_1 \) and \( M_2 \) of dimension at least one.

\begin{proposition}\label{split4}
Let \( M^4 \) be a connected, closed, oriented, splitting \(4\)-manifold.  
If \( M^4 \) admits a Riemannian metric of positive scalar curvature, then \( M^4 \) is diffeomorphic to one of the following:
\begin{enumerate}
    \item \( S^2 \times \Sigma \), where \( \Sigma \) is a closed oriented surface, or
    \item $S^1\times X^3$, where $X^3$ is a connected sum of spherical $3$-manifolds and copies of $S^2\times S^1$, allowing either type of summand to be absent.
\end{enumerate}
\end{proposition}

\begin{proof}
Without loss of generality, since \( M^4 \) is a splitting \(4\)-manifold, it suffices to consider the cases where \( \dim M_1 = 2 \) or \( \dim M_1 = 1 \).  
Note that for closed manifolds \( M_1 \) and \( M_2 \), the product \( M_1 \times M_2 \) is orientable if and only if both \( M_1 \) and \( M_2 \) are orientable.

\emph{Case 1:} \(\dim M_1 = 2\).  
Then \(\dim M_2 = 2\). Closed, oriented surfaces are classified by their genus. If both \( M_1 \) and \( M_2 \) have genus at least \(1\), then \( M_1 \times M_2 \) is aspherical.  
However, Schoen and Yau \cite[Theorem~6]{zbMATH04075990} (see also \cite{zbMATH07817078} for further details) show that closed aspherical \(4\)-manifolds cannot admit Riemannian metrics of positive scalar curvature.  Therefore, at least one of \( M_1 \) or \( M_2 \) must be \( S^2 \). Without loss of generality, let \( M_1 = S^2 \). In this case \( M^4 \cong S^2 \times \Sigma \), where \( \Sigma \) is a closed oriented surface, and such a product admits a Riemannian metric of positive scalar curvature.

\emph{Case 2:} \(\dim M_1 = 1\).  
Then \( M_1 = S^1 \) and \( M_2 \) is a closed, oriented \(3\)-manifold. By applying the virtual Haken theorem, Agol \cite{215872} proves that \( S^1 \times M_2 \) admits a Riemannian metric of positive scalar curvature if and only if \( M_2 \) does. Since \( M^4 \) is diffeomorphic to \( S^1 \times M_2 \), it follows that \( M_2 \) admits a Riemannian metric of positive scalar curvature.
By the geometrization theorem, \( M_2 \) is diffeomorphic to a spherical \(3\)-manifold, to \( S^2 \times S^1 \), or to a connected sum of such manifolds. If \( M_2 = S^2 \times S^1 \), then \( S^1 \times M_2 \) falls into Case~1. Thus, in the present case, \( M^4 \) is diffeomorphic to  $S^1 \times X^3,$ where $X^3$ is a connected sum of spherical $3$-manifolds and
copies of $S^2\times S^1$, allowing either type of summand
to be absent.
\end{proof}

\begin{remark}
Proposition~\ref{split4} does not directly generalize to dimension five, even under the assumption that the classification of closed \(4\)-manifolds admitting metrics of positive scalar curvature is known. Indeed, there exist closed \(4\)-manifolds \(M^4\) which do not admit metrics of positive scalar curvature, but for which \(M^4 \times S^1\) does admit such a metric.  For example, let \(M^4\) be a smooth complex hypersurface of odd degree \(\geq 5\) in \(\mathbb{CP}^3\). Then \(M^4\) is a non-spin, simply connected, smooth \(4\)-manifold with a Kähler structure and \(b^+_2 > 1\). Hence, its Seiberg--Witten invariant does not vanish for some \(\mathrm{Spin}^c\)-structure, which implies that \(M^4\) does not admit a metric of positive scalar curvature.  On the other hand, \(M^4 \times S^1\) is a closed, oriented \(5\)-manifold with fundamental group \(\pi_1(M^4 \times S^1) = \mathbb{Z}\), and its universal cover \(M^4 \times \mathbb{R}\) is non-spin because \(M^4\) is non-spin. Since its fundamental class represents \(0\) in 
$H_5(B\mathbb{Z}; \mathbb{Z}) = 0,$ it follows that \(M^4 \times S^1\) admits a metric of positive scalar curvature (see \cite[Remark~1.25]{zbMATH05342785}). 
\end{remark}

A product is the total space of a trivial bundle. We now consider general smooth fiber bundles.

\begin{proposition}\label{S^2}
The total space of any $S^2$-bundle over a closed surface, as well as any closed, oriented surface bundle over $S^2$, admits a Riemannian metric of positive scalar curvature.
\end{proposition}

\begin{proof}
It is known  that if $(F,g_F)$ is a compact Riemannian manifold with $\mathrm{Sc}_{g_F}>0$, and if a compact Lie group $G$ acts on $F$ by isometries, then for any compact manifold $B$ and any smooth fiber bundle $E \to B$ with fiber $F$ and structure group $G$, the total space $E$ admits a Riemannian metric of positive scalar curvature.

In  dimension 2, any two homeomorphic smooth manifolds are in fact diffeomorphic. 
 By Smale’s theorem, the inclusion $\mathrm{O}(3)\hookrightarrow\mathrm{Diff}(S^2)$ is a homotopy equivalence; hence so is the induced map $B\mathrm{O}(3)\to B\mathrm{Diff}(S^2)$. Consequently, every smooth $S^2$-bundle over a base manifold, classified by a map into $B\mathrm{Diff}(S^2)$, admits a reduction of its structure group to $\mathrm{O}(3)$.

Since the isometry group of the standard round metric on $S^2$ is precisely the compact Lie group $\mathrm{O}(3)$, the structure group of any $S^2$-bundle over a closed  surface can be reduced to $\mathrm{O}(3)$. Therefore,  its total space admits a Riemannian metric of positive scalar curvature.

For the total space of a $\Sigma_g$-bundle over $S^2$, if $g \ge 2$, then the bundle is trivial.
Since $S^1$ is connected,   the image $\phi(S^1)$ of the clutching map lies in a single path component 
$C$ of the group of orientation-preserving diffeomorphisms $\mathrm{Diff}^+(\Sigma_g)$. 
Choose $h=\phi(1)\in C$ and multiply $\phi$ on the left by the constant map $h^{-1}$, which extends over the upper hemisphere. The resulting equivalent clutching loop $\phi(1)^{-1}\phi$ lies in $\mathrm{Diff}_0^+(\Sigma_g)$ and is based at the identity.
Thus, every $\Sigma_g$-bundle over $S^2$ is isomorphic to one whose clutching map is a loop in $\mathrm{Diff}_0^+(\Sigma_g)$ based at $\mathrm{id}$. By the Earle–Eells theorem \cite[Theorem 1]{MR212840}, $\mathrm{Diff}_0^+(\Sigma_g)$ is contractible for $g \ge 2$. 
Hence, $\pi_1(\mathrm{Diff}_0^+(\Sigma_g)) = 0.$ Consequently, there are no nontrivial smooth bundles over $S^2$ with fiber $\Sigma_g$ when $g \ge 2$.

 The total spaces of  non-trivial oriented $T^2$-bundles over $S^2$, classified by the positive integer invariant $d \geq 1$, are diffeomorphic to $S^1 \times L(d, 1)$, where $L(d, 1)$ denotes the lens space $S^3 / \mathbb{Z}_d$. Indeed, oriented $T^2$-bundles over $S^2$ are classified up to isomorphism by non-negative integers $\mathbb{Z}_{\ge 0}$.  The invariant $d$ is the greatest common divisor of $m$ and $n$, where $(m, n) \in H^2(S^2; \mathbb{Z}^2) \cong \mathbb{Z}^2$ is the Euler class, considered modulo the action of $\mathrm{SL}(2, \mathbb{Z})$.  The case $d = 0$ corresponds to the trivial bundle with total space diffeomorphic to $T^2 \times S^2$.  For each $d \ge 1$, there is a unique isomorphism class of non-trivial bundles (see \cite[Section~5.1]{KasuyaNoda2026}).

On the other hand, for $d\geq 1$,  the manifold $S^1 \times L(d,1)$ is the total space of a principal $T^2$-bundle over $S^2$ with Euler class  $(0,d) \in H^2(S^2;\mathbb{Z}^2) \cong \mathbb{Z}^2$.  Indeed,
the lens space $L(d,1) \to S^2$ is the principal $S^1$-bundle with Euler class $d$, obtained from the Hopf fibration $S^1 \to S^3 \xrightarrow{h} S^2$ with $e(h)=1$ by quotienting $S^3$ by the cyclic subgroup 
$\mu_d = \{e^{2\pi i k/d} : k=0,1,\dots,d-1\} \subset S^1,$ which acts freely and diagonally by 
$\zeta \cdot (z_1,z_2) = (\zeta z_1,\zeta z_2).$
The quotient $S^3/\mu_d = L(d,1)$ defines a principal circle bundle $\pi_L : L(d,1) \to S^2$ with Euler class $e(\pi_L)=d$.  

Let $E = S^1 \times L(d,1)$ and define $p : E \to S^2$ by $p(u,y) = \pi_L(y)$.  
Writing $T^2 = S^1_a \times S^1_b$, define a right $T^2$-action by $(u,y)\cdot(a,b) = (ua,\, y\cdot b),$
where $y\cdot b$ denotes the principal $S^1_b$-action on $L(d,1)$.  
This action is free and fiberwise transitive; hence $p : E \to S^2$ is a principal $T^2$-bundle.  

Trivialize the bundle over the two closed hemispheres $D_\pm$. On their common equator, the transition in the $S^1_b$-factor is $\theta\mapsto e^{id\theta}$, so the clutching map of $p$ is
$\gamma(\theta) = (1, e^{i d \theta}) \in S^1_a \times S^1_b.$
Thus the classifying map $S^2 \to BT^2 = BS^1 \times BS^1$ has components $(0,d)$, so the Euler class of the bundle is $(0,d)$ (up to orientation).  For $d=1$,  the total space is $S^1\times S^3$ and not diffeomorphic to the trivial product $S^2 \times T^2$.  Hence $S^1\times L(d,1)$ realizes the oriented $T^2$-bundles over $S^2$ with invariant $d$: as principal $T^2$-bundles these have Euler classes $(m,n)$ with $\gcd(m,n)=d$, and an $\mathrm{SL}(2,\mathbb{Z})$ change of fiber basis takes $(m,n)$ to $(0,d)$.

Since the classification is up to bundle isomorphism, and a bundle isomorphism is a fiber-preserving diffeomorphism covering $\mathrm{Id}_{S^2}$, two isomorphic bundles have diffeomorphic total spaces. 
Hence, any oriented $T^2$-bundle over $S^2$ with invariant $d \ge 1$ has total space diffeomorphic to 
$S^1 \times L(d,1)$.

The products $\Sigma_g\times S^2$ carry the product of any metric on $\Sigma_g$ with a round metric of sufficiently small radius on $S^2$, and $S^1\times L(d,1)$ carries a product of round metrics; both have positive scalar curvature. Thus, the total space of any closed, oriented surface $\Sigma_g$-bundle over $S^2$ admits a Riemannian metric of positive scalar curvature.
\end{proof}

\begin{remark}
For the non-orientable case, Gramain gives a new proof of the Earle–Eells theorem and further shows in \cite[Theorem~1]{MR326773} that $\mathrm{Diff}_0(\mathbb{RP}^2)$ is homotopy equivalent to $\mathrm{SO}(3)$, $\mathrm{Diff}_0(\mathcal{K})$ is homotopy equivalent to $\mathrm{SO}(2)$ for the Klein bottle $\mathcal{K}$, and $\mathrm{Diff}_0(N_g)$ is contractible for non-orientable closed surfaces $N_g$ of the non-orientable genus $g \ge 3$. Hence, by the same argument as above,  any $N_g$-bundle over $S^2$ for $g \ge 3$ is trivial. Consequently, for $g\geq3$, the total space of an $N_g$-bundle over $S^2$ also admits a Riemannian metric of positive scalar curvature.

\end{remark}

\begin{remark}
The above proof  shows that the total space of any \( S^2 \)-bundle over a closed oriented manifold admits a Riemannian metric of positive scalar curvature.   However, there exist manifolds which are total spaces of \( S^7 \)-bundles over \( S^2 \) that do not admit such metrics.  For example, let \( \Sigma^9 \) be an exotic \(9\)-sphere with \( \alpha(\Sigma^9) \neq 0 \). Then $(S^7 \times S^2) \# \Sigma^9$ is the total space of a fiber bundle over \( S^2 \) with fiber \( S^7 \). Nevertheless, \( (S^7 \times S^2) \# \Sigma^9 \) does not admit a Riemannian metric of positive scalar curvature.  On the other hand, the total space $M^4$ of a $\Sigma_g$-bundle over $\Sigma_{g_1}$, with $g > 0$ and $g_1 > 0$, is aspherical. Thus, it does not admit any Riemannian metric of positive scalar curvature. 
\end{remark}

\begin{proposition}\label{S^1-fiber}
Let $M^4$ be a closed, oriented manifold that is the total space of an $S^1$-bundle over a closed, oriented, connected $3$-manifold $N^3$. Then $M^4$ admits a metric of positive scalar curvature if and only if $N^3$ does.
\end{proposition}

\begin{proof}
Assume that $M^4$ admits a metric of positive scalar curvature. For contradiction, suppose that $N^3$ does not admit a metric of positive scalar curvature. Since $N^3$ is closed, oriented, and connected, the classification of $3$-manifolds admitting metrics of positive scalar curvature implies that $N^3$ must contain an aspherical summand in its Kneser–Milnor prime decomposition. In other words, $N^3 = K \# X,$ where $K$ is a closed aspherical $3$-manifold. Since $M^4$ and $N^3$  are  oriented, the $S^1$-bundle is fiberwise orientable.   Hanke’s argument in the proof of Theorem~B of Kumar–Sen \cite{zbMATH08137326} applies in this setting and shows that the total space $M^4$ cannot admit a metric of positive scalar curvature. This is a contradiction. 

For the reader's convenience, we present Hanke's argument following the exposition of Kumar--Sen~\cite[Section~3.2, proof of Theorem~B]{zbMATH08137326}.
Write $p:M^4\to N^3$ for the bundle projection, and recall that
$N$ contains an aspherical prime summand $K$.

We first remove torsion from the fundamental group of the base.
Let $P_1,\ldots,P_\ell$ be the prime summands of $N$ with finite
fundamental group, and consider the natural surjection
\[
\rho:\pi_1(N)\longrightarrow
\prod_{i=1}^{\ell}\pi_1(P_i),
\]
which embeds each finite free factor into the corresponding
coordinate and kills all remaining free factors. Its kernel has
finite index, and it is torsion-free by the Kurosh subgroup theorem,
since every element of finite order in $\pi_1(N)$ is conjugate into
some $\pi_1(P_i)$, on which $\rho$ is injective. Since $\ker\rho$
contains $\pi_1(K)$, the corresponding cover restricts over
$K\setminus\operatorname{int}(D^3)$ to disjoint copies, each mapped
diffeomorphically onto $K\setminus\operatorname{int}(D^3)$. Thus this
cover of $N$ again contains $K$ as a prime summand. Pulling back the
bundle gives a finite cover of $M$, to which the positive scalar
curvature metric lifts. We retain the notation $p:M\to N$ for this
pullback bundle.

As in the proof of \cite[Theorem~B]{zbMATH08137326}, the positive
scalar curvature metric excludes fibers that are homotopically
nontrivial in $M$, since $N$ contains an aspherical prime summand.
The fibers are therefore $\pi_1$-null, and the homotopy exact
sequence gives
\[
G:=\pi_1(M)\cong\pi_1(N).
\]
Because $\pi_1(N)$ is torsion-free, $G$ is a free product of
fundamental groups of closed aspherical $3$-manifolds and of copies
of $\mathbb Z$; hence it has cohomological dimension at most $3$.
By \cite[Theorem~1.1]{zbMATH05379509} and the geometrization theorem,
it satisfies the Baum--Connes conjecture with coefficients.
Schick's real--complex comparison
\cite[Corollary~2.13]{zbMATH02104390} therefore implies that the
real Baum--Connes assembly map is an isomorphism.
Since $G$ is torsion-free, this is the map
\[
KO_*(BG)\longrightarrow KO_*(C^*_{r,\mathbb R}G),
\]
which is consequently injective.

Since the bundle is fiberwise orientable, a connection gives
$TM\cong p^*TN\oplus{\mathbb R}.$
Every closed oriented $3$-manifold is spin, so $M$ is spin.
The argument proving Bolotov's  Theorem
\cite[Main Theorem, Remark~1.7, and Section~2]{zbMATH05530571} therefore shows that
the classifying map
\[
f_M:M\longrightarrow K(G,1)
\]
can be homotoped into the $2$-skeleton. On the other hand,
\cite[Lemma~3.2]{zbMATH08137326} rules out such a homotopy,
because $N$ has an aspherical prime summand and the fibers
are $\pi_1$-null. This contradiction completes the argument.

Suppose that $N^3$ admits a metric of positive scalar curvature. Then one can construct a metric of positive scalar curvature on $M^4$. In fact, this holds in all higher dimensions:
\begin{lemma}
Let $\pi : E \to B$ be a principal $S^1$-bundle over a closed Riemannian manifold $B^n$ with $n \ge 2$. If $B$ admits a Riemannian metric of positive scalar curvature, then the total space $E$ also admits a Riemannian metric of positive scalar curvature.
\end{lemma}

\begin{proof}
Assume $B$ carries a Riemannian metric $g_B$ with $\mathrm{Sc}(g_B)\geq \sigma>0$. Choose a principal connection on $\pi$, given by a connection 
$1$-form $\theta\in\Omega^1(E)$ such that 
$\theta(\xi)=1$ on the fundamental vertical vector field $\xi$ and 
$\theta$ is $S^1$-invariant.  
Since the structure group $S^1$ is abelian, its curvature is
$\Omega = d\theta,$ and $\Omega$ is horizontal and basic.  
Hence there exists a unique $2$-form $\omega\in\Omega^2(B)$ such that
$\Omega = \pi^*\omega.$ For $\varepsilon>0$, define the connection metric
$g_\varepsilon := \pi^* g_B + \varepsilon^2\,\theta\otimes\theta.$
Let $H=\ker\theta$ and $\mathcal V=\ker(d\pi)$ denote the horizontal and vertical 
distributions. Then $T_pE = H_p\oplus \mathcal V_p$, and
$\pi_*:H_p\to T_{\pi(p)}B$ is an isometry, so 
$\pi:(E,g_\varepsilon)\to(B,g_B)$ is a Riemannian submersion.

The vertical direction is spanned by $\xi$, and
$g_\varepsilon(\xi,\xi)=\varepsilon^2$, so 
$V:=\xi/\varepsilon$ is a vertical unit vector field.
Since $g_\varepsilon$ and $\theta$ are $S^1$-invariant, $\xi$ is a Killing
field of constant length.  
A standard identity for Killing fields implies 
$\nabla_\xi\xi=0$; hence each fiber $S^1$ is a geodesic submanifold.
Thus the fibers have intrinsic scalar curvature $0$, and O’Neill’s 
$T$-tensor vanishes.

Let $A$ denote O’Neill’s $A$-tensor,
$A_{\tilde X}\tilde Y := \tfrac12([\tilde X,\tilde Y]^{\mathrm{vert}})$,
where $\tilde X$ and $\tilde Y$ denote the horizontal lifts of
$X$ and $Y$, respectively.
For vector fields $X,Y$ on $B$, using $\Omega=d\theta$ and 
$\Omega=\pi^*\omega$, we have
\[
\pi^*\omega(\tilde X,\tilde Y)
= \Omega(\tilde X,\tilde Y)
= -\theta([\tilde X,\tilde Y]).
\]
Since $\theta(\xi)=1$ and $\xi$ spans the vertical bundle,
$
[\tilde X,\tilde Y]^{\mathrm{vert}}
= -\,\pi^*\omega(\tilde X,\tilde Y)\,\xi.
$ Hence
\[
A_{\tilde X}\tilde Y
= -\tfrac12\,\pi^*\omega(\tilde X,\tilde Y)\,\xi
= -\frac{\varepsilon}{2}\,\pi^*\omega(\tilde X,\tilde Y)\,V.
\]

Let $\{e_i\}_{i=1}^n$ be a local $g_B$-orthonormal frame on $B$, and
$\{\tilde e_i\}$ their horizontal lifts. Then
\[
\|A_{\tilde e_i}\tilde e_j\|_{g_\varepsilon}^2
= \frac{\varepsilon^2}{4}\,\omega(e_i,e_j)^2\circ\pi.
\]
Define the pointwise norm $\|\omega\|_{g_B}^2 := \sum_{i,j=1}^n \omega(e_i,e_j)^2.$
The squared Hilbert–Schmidt norm of the horizontal-horizontal
restriction of $A$ is then
\[
\|A\|^2 
:= \sum_{i,j=1}^n \|A_{\tilde e_i}\tilde e_j\|_{g_\varepsilon}^2
= \frac{\varepsilon^2}{4}\,\|\omega\|_{g_B}^2\circ\pi.
\]

O’Neill’s scalar curvature formula for a Riemannian submersion with
totally geodesic fibers states:
\[
\mathrm{Sc}(g_\varepsilon)
= (\mathrm{Sc}(g_B)\circ\pi)
  + \mathrm{Sc}_{\mathrm{vert}}
  - \|A\|^2.
\]
Here $\mathrm{Sc}_{\mathrm{vert}}=0$;
hence
\[
\mathrm{Sc}(g_\varepsilon)
= \Bigl(\mathrm{Sc}(g_B)
        - \tfrac{\varepsilon^2}{4}\,\|\omega\|_{g_B}^2\Bigr)\circ\pi.
\]

Since $B$ is compact, 
$\|\omega\|_{g_B}^2\le K^2$ for some $K>0$.  
Thus for all $p\in E$,
\[
\mathrm{Sc}(g_\varepsilon)(p)
\;\ge\; \sigma - \frac{\varepsilon^2}{4}K^2.
\]
Choose $\varepsilon>0$ small enough that $\frac{\varepsilon^2}{4}K^2 < \frac{\sigma}{2}.$
Then $\mathrm{Sc}(g_\varepsilon) > \sigma/2 > 0$
everywhere on $E$.
Thus $E$ admits a Riemannian metric of positive scalar curvature. 
\end{proof}

Since both $M^4$ and $N^3$ are oriented, the $S^1$-bundle is fiberwise orientable, and every fiberwise orientable $S^1$-bundle admits the structure of a principal $S^1$-bundle. Hence, if $N^3$ admits a metric of positive scalar curvature, the total space $M^4$ also admits a metric of positive scalar curvature.
\end{proof}

\begin{remark}
However, for $n \ge 5$, there exist  $S^1$-bundles whose total spaces $E^n$ admit metrics of positive scalar curvature, while their base manifolds do not.
 For example,  let \( E \) be a nontrivial \( S^1 \)-bundle over the K3 surface with a primitive Euler class.  Then \( E \) is a closed oriented  simply connected 5-manifold.  Thus, \( E \) admits a Riemannian metric of positive scalar curvature.   However, the K3 surface itself does not admit such a metric, due to the nonvanishing of its \(\hat{A}\)-genus. Furthermore, for any $n \ge 5$, Kumar and Sen \cite[Theorem~A]{zbMATH08137326} use Donaldson divisors to construct closed manifolds serving as base spaces that do not admit metrics of positive scalar curvature, while the total spaces of certain associated $S^1$-bundles do admit metrics of positive scalar curvature.
\end{remark}

\begin{proposition}\label{base manifold}
Let $N^3$ be a closed, oriented, connected $3$-manifold, and let
\[
M^4 = M_\varphi := N \times [0,1] \,/\, (x,1) \sim (\varphi(x),0)
\]
be its mapping torus associated to an orientation-preserving  diffeomorphism $\varphi : N \to N$.  
If $M^4$ admits a Riemannian metric of positive scalar curvature, then $N$ also admits a Riemannian metric of positive scalar curvature.
\end{proposition}

\begin{proof}
By the classification of closed oriented $3$-manifolds admitting positive scalar curvature, a closed oriented $3$-manifold $X$ carries a positive scalar curvature metric if and only if it is \emph{rationally inessential}, i.e.
\[
(c_X)_*[X] \;=\; 0 \;\in\; H_3(B\pi_1(X);\mathbb{Q}),
\]
where $c_X:X\to B\pi_1(X)$ is the classifying map of the universal cover.
  Indeed, if $X$ admits such a metric, its prime summands are spherical or $S^2\times S^1$, so $B\pi_1(X)$ is homotopy equivalent to a wedge of classifying spaces of finite groups and circles, and $H_3(B\pi_1(X);\mathbb{Q})=0$. Conversely, if $X$ admits no such metric, it has an aspherical prime summand $K$; the composition of $c_X$ with the map $B\pi_1(X)\to B\pi_1(K)\simeq K$ induced by the projection of the free product onto $\pi_1(K)$ is homotopic to the degree-one collapse $X\to K$, so $(c_X)_*[X]$ maps to $[K]\neq0$ in $H_3(K;\mathbb{Q})$.

We use this characterization for $N$ and for the hypersurfaces arising in the proof. We begin with a general codimension-$1$ lemma.

\begin{lemma}\label{lem:codim1-PD}
Let $X^n$ be a closed oriented manifold and $a\in H^1(X;\mathbb{Z})$.
Suppose that the Poincaré dual of $a$, $\mathrm{PD}(a)\in H_{n-1}(X;\mathbb{Z})$, is represented by an embedded, oriented, two-sided hypersurface
$i:S\hookrightarrow X,$ where $S$ is a finite disjoint union of connected components.
Then $i^*a = 0 \;\in\; H^1(S;\mathbb{Z}).$
\end{lemma}

\begin{proof}
Since $S$ is an oriented hypersurface in the oriented manifold $X^n$, its normal line bundle is oriented, hence trivial. Thus there exists a global nonvanishing normal vector field along $S$, which we use to push $S$ slightly to obtain a parallel copy $S'$ disjoint from $S$. The submanifolds $S$ and $S'$ are homologous in $X$, so
$[S'] = [S] \;\in\; H_{n-1}(X;\mathbb{Z}).$

Since $H^1(S;\mathbb{Z})\cong\mathrm{Hom}(H_1(S;\mathbb{Z}),\mathbb{Z})$ and $H_1(S;\mathbb{Z})$ is generated by classes of loops in the components of $S$, it suffices to consider classes represented by loops.
Let $[\gamma]\in H_1(S;\mathbb{Z})$ be represented by a loop $\gamma\subset S$, and consider its image $i_*[\gamma]\in H_1(X;\mathbb{Z})$. Since $S'$ and $\gamma$ are disjoint, the intersection number satisfies
$[S']\cdot i_*[\gamma] \;=\; 0.$ Because $[S']=[S]$, we also have $[S]\cdot i_*[\gamma]=0$.

By Poincaré duality, the pairing of $a$ with $i_*[\gamma]$ can be written as
\[
\langle a, i_*[\gamma]\rangle
\;=\;
\mathrm{PD}(a)\cdot i_*[\gamma]
\;=\;
[S]\cdot i_*[\gamma]
\;=\;
0.
\]
On the other hand $\langle i^*a,[\gamma]\rangle = \langle a,i_*[\gamma]\rangle$, so we obtain
$\langle i^*a,[\gamma]\rangle = 0$ for every $[\gamma]\in H_1(S;\mathbb{Z})$. Hence $i^*a=0$ in $H^1(S;\mathbb{Z})$, as claimed.
\end{proof}

Let $f:M^4\to S^1$ be the bundle projection of the mapping torus, and let $\beta\in H^1(S^1;\mathbb{Z})$ be the standard generator. Set
\[
\alpha := f^*\beta \;\in\; H^1(M;\mathbb{Z}).
\]

Fix $t_0\in S^1$, identify the fiber $f^{-1}(t_0)$ with $N$, and let $\iota:N\hookrightarrow M$ denote its inclusion. This fiber is an embedded oriented $3$-dimensional submanifold of $M$. By transversality and Poincar\'e duality, its pushforward fundamental class represents $\mathrm{PD}(\alpha)$ up to sign. We choose the orientation of $N$ so that
\[
\mathrm{PD}(\alpha)=[N]_M:=\iota_*[N]_N\in H_3(M;\mathbb{Z}),
\]
where $[N]_N\in H_3(N;\mathbb{Z})$ is the fundamental class of $N$.

Assume $M$ admits a metric $g$ of positive scalar curvature. Minimizing area in the class $\mathrm{PD}(\alpha)$ gives disjoint, closed, embedded, oriented, two-sided, stable minimal hypersurfaces $\Sigma_1,\ldots,\Sigma_s$ of $(M,g)$ and integers $m_1,\ldots,m_s\geq1$ with $\sum_{l=1}^s m_l[\Sigma_l]_M=\mathrm{PD}(\alpha)$; regularity holds since $\dim M=4$. By Schoen--Yau's formula, each $\Sigma_l$ admits a metric of positive scalar curvature. Replacing each $\Sigma_l$ by $m_l$ disjoint parallel copies, obtained by pushing $\Sigma_l$ along its unit normal, we obtain an embedded, oriented, two-sided hypersurface $i:\widehat S\hookrightarrow M$, each component of which is diffeomorphic to some $\Sigma_l$, such that
\[
i_*[\widehat S] = \mathrm{PD}(\alpha) = [N]_M \in H_3(M;\mathbb{Z}).
\]

The hypersurface $\widehat S$ may be disconnected: $\widehat S = S_1\sqcup\cdots\sqcup S_r,$
where each $S_k$ is a closed, oriented, connected $3$-manifold. Each $S_k$ is diffeomorphic to some $\Sigma_l$, hence admits a metric of positive scalar curvature.

If $[S_k]_M\in H_3(M;\mathbb{Z})$ denotes the pushforward of the fundamental class $[S_k]\in H_3(S_k;\mathbb{Z})$, then
\[
[\widehat S]_M
= i_*[\widehat S]
= \sum_{k=1}^r [S_k]_M
= [N]_M.
\]

Apply Lemma~\ref{lem:codim1-PD} to $a=\alpha$ and the two-sided hypersurface $\widehat S\subset M$. We obtain
\[
i^*\alpha = 0 \;\in\; H^1(\widehat S;\mathbb{Z}).
\]
Since $\alpha=f^*\beta$, this implies
\[
(f\circ i)^*\beta = i^* f^*\beta = i^*\alpha = 0 \;\in\; H^1(\widehat S;\mathbb{Z}).
\]

We have a decomposition
\[
H^1(\widehat S;\mathbb{Z}) \;\cong\; \bigoplus_{k=1}^r H^1(S_k;\mathbb{Z}),
\]
and $(f\circ i)|_{S_k} = f\circ i_k$, where $i_k:S_k\hookrightarrow M$ is the inclusion of the $k$-th component. The equality $(f\circ i)^*\beta = 0$ then yields
\[
(f\circ i_k)^*\beta = 0 \;\in\; H^1(S_k;\mathbb{Z})
\quad\text{for all }k.
\]

For any connected CW complex $X$, homotopy classes of maps $X\to S^1$ are in bijection with $H^1(X;\mathbb{Z})$, via $h\mapsto h^*\beta$. Thus a map $h:X\to S^1$ is null-homotopic if and only if $h^*\beta=0$. It follows that each map
\[
f\circ i_k : S_k\longrightarrow S^1
\]
is null-homotopic. Choosing null-homotopies $H_k:S_k\times[0,1]\to S^1$ and assembling them, we obtain a continuous map
\[
H:\widehat S\times[0,1]\to S^1
\]
such that
\[
H(\cdot,0) = f\circ i,
\qquad
H(\cdot,1)\equiv t_0.
\]
Hence the map $f\circ i:\widehat S\to S^1$ is globally null-homotopic.

Since $f:M\to S^1$ is a fiber bundle, it has the homotopy lifting property. The null-homotopy $H$ lifts to a map
\[
\widetilde H:\widehat S\times[0,1]\longrightarrow M
\]
such that
\[
f\circ \widetilde H = H,
\qquad
\widetilde H(\cdot,0) = i.
\]
Define
\[
j:=\widetilde H(\cdot,1):\widehat S\longrightarrow N.
\]
With the fixed identification $f^{-1}(t_0)=N$, the map $\widetilde H$ is a homotopy between
\[
i \quad\text{and}\quad \iota\circ j : \widehat S\to M.
\]

Passing to $H_3(-;\mathbb{Z})$, we use the homotopy $i\simeq \iota\circ j$ to obtain
\[
i_*[\widehat S]
= (\iota\circ j)_*[\widehat S]
= \iota_*\, (j_*[\widehat S])_N,
\]
where $[\widehat S]\in H_3(\widehat S;\mathbb{Z})$ is the fundamental class and $(j_*[\widehat S])_N\in H_3(N;\mathbb{Z})$ its image.

Since
\[
i_*[\widehat S] = [\widehat S]_M = [N]_M = \iota_*[N]_N,
\]
we obtain
\[
\iota_*\, (j_*[\widehat S])_N \;=\; \iota_*[N]_N
\quad\in\; H_3(M;\mathbb{Z}).
\]

For the mapping torus \(M_\varphi\) of a fiber bundle \(N^3 \to M^4 \to S^1\), 
the Wang exact sequence gives
\[
\dots \longrightarrow 
H_4(M;\mathbb{Z}) \longrightarrow 
H_3(N;\mathbb{Z}) \xrightarrow{\,1 - \varphi_*\;} 
H_3(N;\mathbb{Z}) \xrightarrow{\,\iota_*\;} 
H_3(M;\mathbb{Z}) 
\longrightarrow \dots
\]
Since \( \varphi : N^3 \to N^3 \) is an orientation-preserving diffeomorphism,  we have $\varphi_* = \text{id}: H_3(N;\mathbb{Z})\to H_3(N;\mathbb{Z})$. Hence, the Wang exact sequence implies that the induced map $\iota_* : H_3(N;\mathbb{Z}) \longrightarrow H_3(M;\mathbb{Z})$
is injective. By injectivity of $\iota_*$, the preceding equality implies
\[
(j_*[\widehat S])_N = [N]_N \;\in\; H_3(N;\mathbb{Z}).
\]

Writing $\widehat S = \bigsqcup_{k=1}^r S_k$ and letting $j_k := j|_{S_k}:S_k\to N$, we have
\[
H_3(\widehat S;\mathbb{Z}) \cong \bigoplus_{k=1}^r H_3(S_k;\mathbb{Z}),\qquad
[\widehat S] = \sum_{k=1}^r [S_k],
\]
so that
\[
(j_*[\widehat S])_N
= \sum_{k=1}^r (j_k)_*[S_k]
= [N]_N.
\]

For each $k$, the closed $3$-manifold $S_k$ admits a  metric of positive scalar curvature; hence, by the $3$-dimensional classification, it is rationally inessential. Writing
\[
c_{S_k}:S_k\longrightarrow B\pi_1(S_k)
\]
for the classifying map of the universal cover, we have
\[
(c_{S_k})_*[S_k] = 0 \;\in\; H_3(B\pi_1(S_k);\mathbb{Q}).
\]

Let $c_N:N\to B\pi_1(N)$ be the classifying map for $N$. Each map $j_k:S_k\to N$ induces a group homomorphism
\[
(j_k)_*:\pi_1(S_k)\to\pi_1(N),
\]
and hence a continuous map
\[
B((j_k)_*):B\pi_1(S_k)\to B\pi_1(N)
\]
such that, up to homotopy,
$c_N\circ j_k \simeq B((j_k)_*)\circ c_{S_k}.$ By naturality on homology (with $\mathbb{Q}$-coefficients),
\[
(c_N)_*\circ (j_k)_*
= (B((j_k)_*))_*\circ (c_{S_k})_*.
\]
Thus, we obtain
\[
(c_N)_*[N]_N
= (c_N)_*\!\Bigl(\sum_{k=1}^r (j_k)_*[S_k]\Bigr)
= \sum_{k=1}^r (c_N)_*(j_k)_*[S_k].
\]
For each $k$,
\[
(c_N)_*(j_k)_*[S_k]
= (B((j_k)_*))_*(c_{S_k})_*[S_k]
= (B((j_k)_*))_*(0)
= 0.
\]
Thus every summand vanishes and
\[
(c_N)_*[N]_N = 0 \;\in\; H_3(B\pi_1(N);\mathbb{Q}).
\]

Hence $N$ is rationally inessential. By the $3$-dimensional classification, this is equivalent to the existence of a positive scalar curvature metric on $N$. This completes the proof.

\end{proof}

On the other hand, if the fiber \(N\) admits a metric of positive scalar curvature, then one can use  Bamler and Kleiner's result \cite[Theorem 1]{2019arXiv190908710B} to  construct a metric of positive scalar curvature on its mapping torus.

\begin{proposition}\label{mapping torus}
Let $N$ be a closed, oriented $3$-manifold that admits a Riemannian metric of positive scalar curvature. Let $\varphi : N \to N$ be an orientation-preserving diffeomorphism, and let $M = M_\varphi$ be its mapping torus. Then $M$ admits a Riemannian metric of  positive scalar curvature.
\end{proposition}

\begin{proof}
 Let $\mathcal{R}^+(X)$ be the space of smooth Riemannian metrics with positive scalar curvature, endowed with the $C^\infty$ topology.  Bamler and Kleiner \cite[Theorem 1]{2019arXiv190908710B} show that, for a connected, oriented, closed smooth $3$-manifold $X^3$, the space $\mathcal{R}^+(X^3)$ is either empty or contractible in the $C^\infty$ topology. Hence, $\mathcal{R}^+(N)$ is  path-connected.

Fix a reference  metric $h \in \mathcal{R}^+(N)$. Because $\varphi$ is a diffeomorphism, the pullback $\varphi^{*} h$ also lies in $\mathcal{R}^+(N)$.  Since $\mathcal{R}^+(N)$ is  path-connected, one has  a continuous path
$ \gamma : [0,1] \to \mathcal{R}^+(N),$ $
  \gamma(0) = h,$ $\gamma(1) = \varphi^{*}h,$
continuous with respect to the $C^\infty$ topology. Smoothing this path relative to its endpoints, while preserving positive scalar curvature, gives an isotopy from $h$ to $\varphi^*h$.

Two metrics \( g_0, g_1 \in \mathcal{R}^+(N) \) are said to be \emph{concordant} if there exists a positive scalar curvature metric  $\bar{g}$ on $ N\times I$ such that \( \bar{g}|_{N \times \{0\}} = g_0 \), \( \bar{g}|_{N \times \{1\}} = g_1 \), and \( \bar{g} \) is a product metric in neighborhoods of the boundary components \( N \times \{0\} \sqcup N \times \{1\} \).

Since \( h \) and \( \varphi^{*}h \) are isotopic, they are also concordant (see \cite[Proposition~3.3]{MR1818778}).  
Thus, there exists a metric of positive scalar curvature \( \bar{g} \) on \( N \times I \) such that  $\bar{g}|_{N \times \{0\}} = h,$ $\bar{g}|_{N \times \{1\}} = \varphi^{*}h,$
and \( \bar{g} \) is a product metric near the boundary  components \( N \times \{0\} \sqcup N \times \{1\} \).  
In particular, the boundary components are totally geodesic, i.e., their second fundamental form vanishes.  

Moreover, since \( \varphi : (N,\varphi^{*}h) \to (N,h) \) is an isometry, we can glue \((N \times I, \bar{g})\) along the identification \((x,1) \sim (\varphi(x),0)\) defining \( M_\varphi \). Since \( \bar{g} \) equals \( \varphi^{*}h\oplus dt^2 \) on \( N\times(1-\epsilon,1] \) and \( h\oplus dt^2 \) on \( N\times[0,\epsilon) \), and the product collar metrics agree under the boundary identification, the glued metric is a smooth metric of positive scalar curvature on \( M_\varphi \).
\end{proof}

\begin{remark}
There exists a closed, connected, simply connected smooth \(4\)-manifold \(X^4\) for which \(\mathcal{R}^+(X^4)\) has more than one path component.  
Hence the proof relying on the Bamler--Kleiner theorem does not extend to dimension $5$.
\end{remark}

\begin{proof}[Proof of Theorem \ref{B}]
If the bundle is trivial, Proposition \ref{split4} implies the conclusion.  
If the fiber or the base is $S^2$, Proposition \ref{S^2} implies the conclusion.  
If the fiber is $S^1$, Proposition \ref{S^1-fiber} implies the conclusion.  

Any smooth fiber bundle $\pi: M \to S^1$ with fiber $N$ is isomorphic to the mapping torus $M_\varphi$ of a diffeomorphism $\varphi: N \to N$, which is orientation-preserving since $M$ and $N$ are oriented.  
Thus, when the base is $S^1$, $M^4$ is the mapping torus of $N^3$, and Propositions \ref{base manifold} and \ref{mapping torus} complete the proof.
\end{proof}

It remains an open problem to classify the topological types of closed
$4$-manifolds that admit scalar-flat metrics but no metrics of positive
scalar curvature. Since every scalar-flat metric on such a manifold
is Ricci-flat, this problem is part of the broader problem of
classifying the topological types of closed Ricci-flat $4$-manifolds.  In particular, no examples are known of closed simply connected Ricci-flat manifolds with generic holonomy, and the existence of a Ricci-flat metric on $S^n$ for $n \ge 4$ remains a major open problem.

\section{Rigidity theorems for scalar curvature}\label{4}

Obata \cite[Proposition~6.1]{zbMATH03374588} shows that if $g\in[g_{st}]$ on $S^n$, $n\geq3$, and $\mathrm{Sc}_g\equiv n(n-1)$, then $g$ is isometric to $g_{st}$. On a closed, connected Einstein manifold of positive scalar curvature, the metric is the unique constant scalar curvature metric in its conformal class, up to scaling, unless the manifold is conformally equivalent to the round sphere. In the spherical case, uniqueness holds up to scaling and pullback by conformal diffeomorphisms. Constant positive scalar curvature metrics in a conformal class containing no Einstein metric need not be unique up to scaling.

The sharp Sobolev inequality on $(S^n,g_{st})$ gives another approach; see \cite[P. 121, Theorem 5.1]{MR1688256}. For $n\geq3$, every $f\in C^\infty(S^n)$ satisfies
\begin{equation*}
\begin{aligned}
\left(\int_{S^n}|f|^{\frac{2n}{n-2}}\,dv_{g_{st}}\right)^{\frac{n-2}{n}}
&\leq \frac{4}{n(n-2)\mathrm{Vol}_{g_{st}}(S^n)^{\frac{2}{n}}}
\int_{S^n}|\nabla f|^2_{g_{st}}\,dv_{g_{st}}\\
&\quad+\frac{1}{\mathrm{Vol}_{g_{st}}(S^n)^{\frac{2}{n}}}
\int_{S^n}f^2\,dv_{g_{st}}.
\end{aligned}
\end{equation*}
For a positive function $f$, equality holds if and only if
$f^{4/(n-2)}g_{st}=c\,\Phi^*g_{st}$ for some $c>0$ and some conformal diffeomorphism $\Phi$ of $S^n$.
The sharp Sobolev inequality yields the following weaker version of Obata's result.

\begin{theorem}\label{LCF n}
Let $(M^n,g)$, $n\geq3$, be an oriented, closed, simply connected, locally conformally flat manifold with $\mathrm{Sc}_g=n(n-1)$. Then $\mathrm{Vol}_g(M)\geq\mathrm{Vol}_{g_{st}}(S^n)$. If equality holds, then $(M^n,g)$ is isometric to $(S^n,g_{st})$.
\end{theorem}

\begin{proof}
By Kuiper's theorem, there exist a positive smooth function $u$ on $S^n$ and a diffeomorphism $\phi:S^n\to M$ such that $\phi^*g=u^{4/(n-2)}g_{st}$. We continue to denote $\phi^*g$ by $g$. The conformal scalar-curvature formula and $\mathrm{Sc}_g=n(n-1)$ give
\begin{equation}\label{n(n-1)}
\frac{4(n-1)}{n-2}\Delta_{g_{st}}u+n(n-1)u
=n(n-1)u^{\frac{n+2}{n-2}}.
\end{equation}

The sharp Sobolev inequality applied to $u$ gives
\begin{equation*}
\begin{aligned}
\left(\int_{S^n}u^{\frac{2n}{n-2}}\,dv_{g_{st}}\right)^{\frac{n-2}{n}}
&\leq \frac{1}{n(n-1)\mathrm{Vol}_{g_{st}}(S^n)^{\frac{2}{n}}}\\
&\quad\times\int_{S^n}\left(\frac{4(n-1)}{n-2}|\nabla u|^2_{g_{st}}
+n(n-1)u^2\right)\,dv_{g_{st}}.
\end{aligned}
\end{equation*}
Multiplying \eqref{n(n-1)} by $u$ and integrating by parts gives
\begin{equation*}
\int_{S^n}\left(\frac{4(n-1)}{n-2}|\nabla u|^2_{g_{st}}
+n(n-1)u^2\right)\,dv_{g_{st}}
=n(n-1)\int_{S^n}u^{\frac{2n}{n-2}}\,dv_{g_{st}}.
\end{equation*}
Consequently,
\begin{equation}\label{Sobolev}
\left(\int_{S^n}u^{\frac{2n}{n-2}}\,dv_{g_{st}}\right)^{\frac{n-2}{n}}
\leq \frac{1}{\mathrm{Vol}_{g_{st}}(S^n)^{\frac{2}{n}}}
\int_{S^n}u^{\frac{2n}{n-2}}\,dv_{g_{st}}.
\end{equation}

Since
\begin{equation*}
\mathrm{Vol}_g(M)=\int_{S^n}u^{\frac{2n}{n-2}}\,dv_{g_{st}},
\end{equation*}
inequality \eqref{Sobolev} becomes
\begin{equation*}
\mathrm{Vol}_g(M)^{\frac{n-2}{n}}
\leq\frac{\mathrm{Vol}_g(M)}{\mathrm{Vol}_{g_{st}}(S^n)^{\frac{2}{n}}}.
\end{equation*}
Hence $\mathrm{Vol}_g(M)\geq\mathrm{Vol}_{g_{st}}(S^n)$.
If equality holds, then equality also holds in the sharp Sobolev inequality. Thus $g=c\,\Phi^*g_{st}$ for some conformal diffeomorphism $\Phi$ of $S^n$ and some $c>0$. Since $\mathrm{Sc}_g=n(n-1)$, we have $c=1$. Therefore $(M,g)$ is isometric to $(S^n,g_{st})$.
\end{proof}

\begin{remark}
The assumption of local conformal flatness in Theorem~\ref{LCF n} is necessary in every dimension $n\geq3$, even when the underlying manifold is $S^n$. Indeed, choose a smooth metric $q$ so  close to $g_{st}$ in the $C^2$ topology that $\mathrm{Sc}_q>0$, but not conformally equivalent to the round metric. The strict Aubin--Schoen inequality gives
\[
0<Y(S^n,[q])<Y(S^n,[g_{st}]).
\]
Let $g\in[q]$ be a Yamabe metric normalized by $\mathrm{Sc}_g=n(n-1)$. Then
\[
\mathrm{Vol}_g(S^n)
=\left(\frac{Y(S^n,[q])}{n(n-1)}\right)^{n/2}
<\mathrm{Vol}_{g_{st}}(S^n).
\]
\end{remark}

\begin{remark}
If metrics $\{g_i\}_{i\in\mathbb N}$ on $S^n$ satisfy $\mathrm{Sc}_{g_i}\equiv n(n-1)$ and belong to a fixed conformal class $[g]$, then
\[
Y(S^n,[g])\leq n(n-1)\mathrm{Vol}_{g_i}(S^n)^{2/n}.
\]
Since $[g]$ contains a metric of positive scalar curvature, $Y(S^n,[g])>0$, and therefore
\[
\mathrm{Vol}_{g_i}(S^n)
\geq\left(\frac{Y(S^n,[g])}{n(n-1)}\right)^{n/2}>0.
\]
However, for each $n\geq25$, Gong and Li \cite{MR4937973} construct a smooth metric $\tilde g$ on $S^n$ whose conformal class contains a sequence of smooth metrics $\{g_i\}_{i\in\mathbb N}$ with $\mathrm{Sc}_{g_i}=n(n-1)$ and $\mathrm{Vol}_{g_i}(S^n)\to\infty$ as $i\to\infty$.
\end{remark}

\begin{remark}
Let $g$ be a Yamabe metric on $S^4$ satisfying $\mathrm{Sc}_g=12$ and $\mathrm{Vol}_g(S^4)=\mathrm{Vol}_{g_{st}}(S^4)$. Then
\[
Y(S^4,[g])=12\mathrm{Vol}_g(S^4)^{1/2}
=12\mathrm{Vol}_{g_{st}}(S^4)^{1/2}=Y(S^4,[g_{st}]).
\]
Thus $(S^4,[g])$ is conformally equivalent to the round sphere. By Obata's theorem, $g$ is isometric to $g_{st}$. The scalar-curvature and volume assumptions also imply, by the Gauss--Bonnet--Chern formula \eqref{4d GBC}, that
\[
\int_{S^4}|W_g|^2\,dv_g
=2\int_{S^4}|\mathring{\mathrm{Ric}}_g|^2\,dv_g,
\]
which motivates the following conjecture.
\end{remark}

\begin{conjecture} \label{4D CSC}
Assume that \( M^4 \) is a closed, connected, oriented, smooth \( 4 \)-manifold admitting a hyperbolic metric, and let \( g \) be a Yamabe metric on \( M^4 \) satisfying
\[
\int_{M^4} |W_g|^2\, dv_g = 2 \int_{M^4} |\mathring{\mathrm{Ric}}_g|^2\, dv_g.
\]
Then the sectional curvature of \( (M^4, g) \) is constant.

\end{conjecture}

 Conjecture \ref{4D CSC} implies that in dimension 
4, if a Yamabe metric has the same volume and scalar curvature as the hyperbolic metric, then it must in fact be isometric to the hyperbolic metric.

Inspired by the examples of manifolds for which the Yamabe invariant has been computed up to 2019, Macbeth asks in \cite[Question~1.3]{MR3920349}:
\begin{quote}
Do all conformal classes attaining the Yamabe invariant contain an Einstein metric?
\end{quote}
For closed hyperbolic manifolds, every conformal class has negative Yamabe constant, so its Yamabe metric is unique up to scaling, and \cite[Lemma~1.4]{MR3920349} shows that a conformal class attaining the Yamabe invariant contains an Einstein metric. We show that, after rescaling, this Einstein metric is isometric to the hyperbolic metric.
\begin{theorem} \label{attained}
Let $(M^n,g_H)$, $n\geq3$, be a closed hyperbolic manifold. If a conformal class $[g]$ attains $Y(M^n)$, then $[g]$ contains a metric isometric to $g_H$.
\end{theorem}

\begin{proof}
Since $M$ does not admit a metric with positive scalar curvature, one has $Y(M^n) \le 0$.
Suppose first that $Y(M^n)=0$. We show that no conformal class attains the Yamabe invariant. Such a class $[g]$ satisfies $Y(M^n,[g])=0$ and hence contains a scalar-flat metric, so it suffices to show that $M^n$ admits no scalar-flat metric.
Indeed, suppose there exists a scalar-flat metric $\hat g$ on $M^n$. Then $\hat g$ must be Ricci-flat. Otherwise, the scalar-flat metric $\hat g$ can be deformed to a metric with positive scalar curvature, contradicting the assumption that $Y(M^n) \le 0$.
 
By the Cheeger--Gromoll splitting theorem, the universal cover of $(M^n,\hat g)$ endowed with the lifted metric
splits isometrically as $\bar M^{\,n-k} \times \mathbb{E}^k$, where $\bar M$ is a closed, simply connected manifold.
Since $M^n$ admits a hyperbolic metric, its universal cover is homeomorphic to $\mathbb{R}^n$.
Consequently, $\bar M^{\,n-k}$ is a closed, contractible manifold.
Moreover, since $\bar M$ is closed, its top-dimensional mod-$2$
homology group satisfies
$
H_{n-k}(\bar M;\mathbb{Z}/2\mathbb{Z})
\cong \mathbb{Z}/2\mathbb{Z}.
$
Since $\bar M$ is contractible, its homology vanishes in positive
degrees, so $n=k$. Hence, $\hat g$ must be flat.

  However, since $M^n$ is closed, it cannot simultaneously admit both a flat metric and a hyperbolic metric $g_H$.
Indeed, a flat metric on $M^n$ implies that its fundamental group has polynomial growth, whereas a hyperbolic metric on $M^n$ implies exponential growth, yielding a contradiction. Thus, for a closed manifold \( M^n \) admitting a hyperbolic metric, the Yamabe constant satisfies \( Y(M^n, [g]) < 0 \) for every conformal class \([g]\). 

Thus, if $[g]$ attains the Yamabe invariant, then $Y(M^n)<0$. By Proposition~\ref{Yamabe}, the Yamabe metric in $[g]$ is unique up to rescaling. Since $[g]$ attains the Yamabe invariant, this metric is Einstein \cite[Lemma~1.4]{MR3920349}; denote it by $g_E$. Its Einstein constant is negative, so we may rescale it to arrange that $\mathrm{Ric}_{g_E}=-(n-1)$. Proposition~\ref{Besson} then gives $\mathrm{Vol}_{g_E}(M^n)\geq\mathrm{Vol}_{g_H}(M^n)$. Therefore
\[
\begin{aligned}
Y(M^n)
&=Y(M^n,[g_E])\\
&=-n(n-1)\mathrm{Vol}_{g_E}(M^n)^{2/n}\\
&\leq-n(n-1)\mathrm{Vol}_{g_H}(M^n)^{2/n}\\
&=Y(M^n,[g_H])\leq Y(M^n).
\end{aligned}
\]
All inequalities are equalities, so $\mathrm{Vol}_{g_E}(M^n)=\mathrm{Vol}_{g_H}(M^n)$. The rigidity statement in Proposition~\ref{Besson} implies that $g_E$ is isometric to $g_H$, proving that $[g]$ contains a metric isometric to $g_H$.
\end{proof}

 Consequently, Schoen's conjecture holds on a closed hyperbolic manifold $(M^n,g_H)$ if and only if $Y(M^n)$ is attained. The forward implication is established in Section~\ref{2};
the converse follows from Theorem~\ref{attained} and
Proposition~\ref{ncsc}, applied to the normalized Yamabe metric
in each conformal class.

\begin{remark}\label{volume attanded}
If Schoen’s conjecture holds, then Theorem~\ref{attained} implies that any metric $g$ on $(M^n,g_H)$ satisfying
$\mathrm{Sc}_g = -n(n-1)$ and $\mathrm{Vol}_g(M^n) = \mathrm{Vol}_{g_H}(M^n)$
must be isometric to the hyperbolic metric $g_H$.
Indeed, $g$ is a Yamabe metric because its scalar curvature is constant and negative. Since Schoen's conjecture implies that $[g_H]$ attains the Yamabe invariant, the volume equality gives
\[
\begin{aligned}
Y(M^n,[g])
&=-n(n-1)\mathrm{Vol}_g(M^n)^{2/n}\\
&=Y(M^n,[g_H])=Y(M^n).
\end{aligned}
\]
Theorem~\ref{attained} therefore supplies a metric in $[g]$ isometric to $g_H$. This metric and $g$ have the same constant negative scalar curvature, so Proposition~\ref{Yamabe} identifies them.
In particular, this conclusion holds in dimension~$3$.
\end{remark}

The Einstein part of the proof of Theorem~\ref{attained} uses only $Y(M^n)\leq0$; hence every conformal class attaining the Yamabe invariant of a closed manifold without metrics of positive scalar curvature contains an Einstein metric. We record the aspherical case in dimensions $4$ and $5$; see \cite{zbMATH07375613} for further aspherical manifolds without metrics of positive scalar curvature.

\begin{corollary}\label{asphical}
For $n \in \{4,5\}$, if $M^n$ is a closed aspherical $n$-manifold, then every conformal class attaining the Yamabe invariant contains an Einstein metric.
\end{corollary}

\begin{proof}
Since a closed aspherical manifold $M^n$ with $n=4,5$ does not admit a metric of positive scalar curvature
\cite{zbMATH04075990}, one has $Y(M^n) \le 0$.

Suppose first that $Y(M^n)=0$ and that a conformal class $[g]$ attains it.
Then there exists a scalar-flat metric $\hat g \in [g]$.
Because $M^n$ does not admit metrics of positive scalar curvature, $\hat g$ must in fact be Ricci-flat; in particular, $\hat g$ is Einstein.
Moreover, since $M^n$ is aspherical, the Cheeger--Gromoll splitting theorem implies that $\hat g$ is flat.

If $Y(M^n)<0$, then for any conformal class $[g]$ one has $Y(M,[g])<0$.
By Proposition~\ref{Yamabe}, the Yamabe metric in $[g]$ is unique up to rescaling.
Hence, if a conformal class attains the Yamabe invariant, the corresponding Yamabe metric must be Einstein
\cite[Lemma~1.4]{MR3920349}.
\end{proof}

\begin{remark}
In \cite[Corollary~B]{2026arXiv260902267D}, posted on arXiv on September 2, 2026, the author proves that no closed aspherical $n$-manifold with $n\geq 5$ admits a locally conformally flat metric of positive scalar curvature. Families 335 and 336, released by OpenAI on October 6, 2026, remove the local conformal flatness assumption. If the results announced in these manuscripts are correct,  Corollary~\ref{asphical} extends to all $n\geq 6$.
\end{remark}

Motivated by Conjecture \ref{4D CSC},  Theorem \ref{attained}, and Remark \ref{volume attanded},  we propose the following volume rigidity conjecture in higher dimensions:
\begin{conjecture}[Volume Rigidity Conjecture] \label{HYC}
Let $g$ be a Riemannian metric on a closed hyperbolic $n$-manifold $(M^n,g_H)$, $n\geq5$, satisfying $\mathrm{Sc}_g=-n(n-1)$ and $\mathrm{Vol}_g(M^n)=\mathrm{Vol}_{g_H}(M^n)$. Then $g$ is isometric to $g_H$.
\end{conjecture}

A fixed positive scalar curvature does not give a universal lower bound by the round volume on spheres, as the preceding examples show.

Scaling $g_H$ by arbitrarily large constants shows that the volumes of metrics on a closed hyperbolic $n$-manifold satisfying $\mathrm{Sc}_g\geq-n(n-1)$ are unbounded. For a closed, oriented hyperbolic $3$-manifold $M$ with $b_1(M)>0$, Reznikov \cite[Theorems~1 and~2]{zbMATH01054076} uses an inequality relating the area norm to the Thurston norm to show  that the volumes remain unbounded even under the constraint $\mathrm{Sc}_g\equiv-6$.

\begin{remark}
On every closed, connected smooth manifold $M^n$, $n\geq3$, the volumes of metrics with constant scalar curvature $-n(n-1)$ are unbounded above. In particular, this holds for every closed hyperbolic manifold in these dimensions.

To see this, fix a smooth metric $q_0$ on $M$ and set
$\delta=1-2^{-2/n}\in(0,1).$
For every sufficiently large integer $j$, we have $-j<\min_M\mathrm{Sc}_{q_0}$. Apply Lohkamp's scalar curvature deformation theorem \cite[Theorem 1]{zbMATH01295406} with $U=M$, prescribed function $f_j\equiv-j$, and curvature tolerance $\varepsilon=1$. The $C^0$-closeness assertion immediately following that theorem allows the resulting metric $q_j$ to be chosen so that, simultaneously,
\[
-j-1\leq\mathrm{Sc}_{q_j}\leq-j,
\qquad
(1-\delta)q_0\leq q_j\leq(1+\delta)q_0.
\]
The eigenvalues of $q_0^{-1}q_j$ are at least $1-\delta$, so the associated volume densities satisfy
\[
dv_{q_j}\geq(1-\delta)^{n/2}\,dv_{q_0}
=\tfrac12\,dv_{q_0}.
\]
Consequently,
\[
\begin{aligned}
Y(M,[q_j])
&\leq S(q_j)\\
&\leq-j\,\mathrm{Vol}_{q_j}(M)^{2/n}\\
&\leq-j\left(\tfrac12\mathrm{Vol}_{q_0}(M)\right)^{2/n}
\longrightarrow-\infty.
\end{aligned}
\]
Let $g_j\in[q_j]$ be the Yamabe metric normalized by $\mathrm{Sc}_{g_j}=-n(n-1)$. Then
\[
\mathrm{Vol}_{g_j}(M)
=\left(\frac{-Y(M,[q_j])}{n(n-1)}\right)^{n/2}
\longrightarrow\infty.
\]
\end{remark}

\begin{question}
Let $N^n$, $n\geq 4$, be a closed, connected, oriented smooth manifold with $Y(N)<0$, and let $(M^n,g_H)$ be a closed, connected, oriented hyperbolic manifold. Suppose there exists a map $F:N\to M$ of nonzero degree. Must every metric $g$ on $N$ satisfying $\mathrm{Sc}_g\geq-n(n-1)$ also satisfy $\mathrm{Vol}_g(N)\geq\mathrm{Vol}_{g_H}(M)$?
\end{question}

\begin{remark}
A positive lower bound for scalar curvature alone gives no positive lower bound for volume, even on $S^n$. Indeed, for $0<t\leq1$, the metrics $g_t=t^2g_{st}$ satisfy
\[
\mathrm{Sc}_{g_t}=t^{-2}n(n-1)\geq n(n-1),
\qquad
\mathrm{Vol}_{g_t}(S^n)=t^n\mathrm{Vol}_{g_{st}}(S^n)\longrightarrow0
\]
as $t\to0$.
\end{remark}

Moreover,  we  establish additional rigidity theorems for negative scalar curvature without requiring the existence of a hyperbolic metric.

\begin{proposition}\label{conformal}
Let $(M^n,g)$, $n\geq3$, be a closed, connected Riemannian manifold with $\mathrm{Sc}_g\leq0$ and $\mathrm{Sc}_g\not\equiv0$. If $\tilde g$ is conformal to $g$ and satisfies $\mathrm{Sc}_{\tilde g}\geq\mathrm{Sc}_g$ and $\tilde g\leq g$, then $\tilde g=g$.
\end{proposition}

\begin{proof}
Write $\tilde g=u^{4/(n-2)}g$ for some positive $u\in C^\infty(M)$. The inequality $\tilde g\leq g$ gives $u\leq1$. The conformal scalar-curvature formula and $\mathrm{Sc}_{\tilde g}\geq\mathrm{Sc}_g$ give
\[
\begin{aligned}
\frac{4(n-1)}{n-2}\Delta_gu
&=\mathrm{Sc}_{\tilde g}u^{\frac{n+2}{n-2}}-\mathrm{Sc}_gu\\
&\geq\mathrm{Sc}_gu\left(u^{\frac4{n-2}}-1\right)\geq0,
\end{aligned}
\]
where the last inequality uses $\mathrm{Sc}_g\leq0$ and $0<u\leq1$. Since $M$ is closed, integrating gives $\Delta_gu\equiv0$. Connectedness then implies $u\equiv c$ for some $0<c\leq1$. Hence
\[
\mathrm{Sc}_{\tilde g}=c^{-4/(n-2)}\mathrm{Sc}_g\geq\mathrm{Sc}_g.
\]
Since $\mathrm{Sc}_g<0$ somewhere, this inequality implies $c\geq1$. Thus $c=1$, and $\tilde g=g$.
\end{proof}

\begin{remark}
Proposition~\ref{conformal} is a modified version of Ho's observation \cite{zbMATH07926133}. The same method gives a boundary version. Let $(M^n,g)$, $n\geq3$, be a compact, connected Riemannian manifold with boundary $\partial M$, with $\mathrm{Sc}_g\leq0$ and $\mathrm{Sc}_g\not\equiv0$ in $M$, and with $H_g\equiv0$ on $\partial M$. If $\tilde g$ is conformal to $g$ and satisfies $\tilde g\leq g$ and $\mathrm{Sc}_{\tilde g}\geq\mathrm{Sc}_g$ in $M$, together with $H_{\tilde g}\equiv0$ on $\partial M$, then $\tilde g=g$.
Indeed, writing $\tilde g=u^{4/(n-2)}g$, the conformal transformation law for mean curvature gives $\partial_{\nu_g}u=0$ on $\partial M$, where $\nu_g$ is the outward unit normal. The proof above gives $\Delta_gu\geq0$. The divergence theorem and the Neumann boundary condition yield $\int_M\Delta_gu\,dv_g=0$, so $\Delta_gu\equiv0$. Integration by parts then gives $\int_M|\nabla u|_g^2\,dv_g=0$. The rest of the argument is unchanged.
\end{remark}

The author \cite{zbMATH07544449} uses harmonic maps to prove the following result for Einstein targets of positive scalar curvature, without using the Yamabe--Trudinger--Aubin--Schoen theorem or index theory. Here and in the next theorem, $C$ denotes the scalar invariant of the map defined in \cite[equation~(10)]{zbMATH07544449}.

\begin{theorem}
Let $(N,g_N)$ be a closed, connected, oriented Einstein $n$-manifold with scalar curvature $R^*>0$, and let $(M,g_M)$ be a closed, connected, oriented Riemannian $n$-manifold with scalar curvature $R\geq R^*$. Suppose there exists a smooth $(1,\wedge^1)$-contracting map $f:M\to N$ of nonzero degree, that is,
\[
|f_*v|_{g_N}\leq|v|_{g_M}\qquad(v\in TM).
\]
If $f$ is harmonic and satisfies $C\leq0$, then $f$ is a local isometry. If $N$ is simply connected, then $f$ is an isometry.
\end{theorem}

The argument used to prove Theorem~A in \cite{zbMATH07544449} also yields the following result.

\begin{theorem}\label{Einstein}
Let $(N,g_N)$ be a closed, connected, oriented Einstein $n$-manifold with scalar curvature $R^*<0$, and let $(M,g_M)$ be a closed, connected, oriented Riemannian $n$-manifold with $\mathrm{Sc}_{g_M}\geq R^*$. Suppose there exists a smooth $1$-expansive map $f:(M,g_M)\to(N,g_N)$ of nonzero degree, that is,
\[
|f_*v|_{g_N}\geq|v|_{g_M}\qquad(v\in TM).
\]
If $f$ is harmonic and satisfies $C\leq0$, then $f$ is a local isometry.
\end{theorem}

\begin{proof}
We use the notation from the proof of Theorem~A in \cite{zbMATH07544449}, translating its Laplacian to our positive-spectrum convention. Let $d\nu_M$ and $d\nu_N$ denote the Riemannian volume forms of $g_M$ and $g_N$, and set
\[
V:=\left(\frac{f^*d\nu_N}{d\nu_M}\right)^2.
\]
Since $M$ is compact, $V$ attains its maximum at some $x\in M$. The nonzero degree of $f$ implies $V(x)>0$.
Choose local $g_M$-orthonormal and $g_N$-orthonormal frames $\{e_i\}$ and $\{e_a^*\}$, respectively, and write
\[
f_*e_i=\sum_a A_i^a e_a^*.
\]
For the harmonic map $f$, the identity in \cite[equation~(11)]{zbMATH07544449} becomes
\[
-\frac12\Delta_{g_M}V
=2\sum_j\mathbb A_j^2
+V\left(\mathrm{Sc}_{g_M}-\sum_{b,c,j}R^*_{bc}A_j^bA_j^c\right)-C,
\]
where $R^*_{bc}$ are the components of $\mathrm{Ric}_{g_N}$ and $\mathbb A_j$ are the frame components of the differential of $f^*d\nu_N/d\nu_M$.

At the maximum point $x$, our sign convention gives $\Delta_{g_M}V(x)\geq0$. Diagonalizing $f^*g_N$ with respect to $g_M$ at $x$, choose the frames so that
\[
f_*e_i=\lambda_i e_i^*,\qquad
\lambda_1\geq\cdots\geq\lambda_n>0.
\]
Since $f$ is $1$-expansive, every $\lambda_i\geq1$. The Einstein condition on $g_N$ gives, at $x$,
\[
\sum_{b,c,j}R^*_{bc}A_j^bA_j^c
=\frac{R^*}{n}\sum_i\lambda_i^2.
\]
Consequently,
\[
\begin{aligned}
0&\geq-\frac12\Delta_{g_M}V(x)\\
&=2\sum_j\mathbb A_j(x)^2-C(x)\\
&\quad+V(x)\left[\mathrm{Sc}_{g_M}(x)-R^*
+\frac{-R^*}{n}\left(\sum_i\lambda_i^2-n\right)\right]\geq0.
\end{aligned}
\]
Every term on the right is nonnegative. Since $V(x)>0$ and $R^*<0$, it follows that $\mathrm{Sc}_{g_M}(x)=R^*$ and $\lambda_i=1$ for every $i$. Thus $V(x)=1$.

At any point $y\in M$, all singular values $\lambda_i(y)$ of $df_y$ are at least one, so
\[
1\leq V(y)=\prod_i\lambda_i(y)^2\leq V(x)=1.
\]
Hence $V\equiv1$, and every singular value equals one at every point. Therefore $f^*g_N=g_M$, and $f$ is a local isometry.
\end{proof}

Motivated by Theorem~\ref{Einstein}, we propose the following rigidity conjecture for the identity map, without assuming harmonicity or the condition $C\leq0$.

\begin{conjecture}\label{Einstein Rigidity}
Let $(M^n, g_E)$ be a closed Einstein $n$-manifold ($n \geq 4$) with  negative scalar curvature $R^* < 0$.  If $g$ is another metric on $M$ satisfying $\mathrm{Sc}_g \geq R^*$ and $g \leq g_E$, then $g$ is isometric to the Einstein metric $g_E$.
\end{conjecture}

Inspired by Theorem \ref{Einstein}, we propose the following Llarull-type rigidity conjecture in the hyperbolic setting.
\begin{conjecture}[Hyperbolic Rigidity Conjecture] \label{hyper rigidity}
Let $(N^n,g)$, $n\geq 4$, be a closed, connected, oriented Riemannian manifold with $\mathrm{Sc}_g\geq-n(n-1)$, and let $(M^n,g_H)$ be a closed, connected, oriented hyperbolic manifold. Suppose there exists a smooth $1$-expansive map $f:(N^n,g)\to(M^n,g_H)$ of nonzero degree, that is,
\[
|f_*v|_{g_H}\geq|v|_g\qquad(v\in TN).
\]
Then $f$ is a Riemannian covering. If, in addition, $|\deg f|=1$, then $f$ is an isometry.
\end{conjecture}

For $f=\mathrm{id}$ and metrics sufficiently close to $g_H$ in the $C^2$ topology, the conclusion of Conjecture~\ref{hyper rigidity} follows from the local rigidity theorem of Besson--Courtois--Gallot \cite{zbMATH04192565}. Using the rescaled Ricci flow and the stability results of Ye and Bamler \cite{zbMATH00427933,zbMATH06323440}, we improve this to a $C^0$-local rigidity statement.

\begin{proposition}
Let $(N^n,h)$ be a closed, connected hyperbolic manifold with
sectional curvature $-1$, where $n\geq3$. There exists
$\varepsilon(h)\in(0,1)$ such that every smooth Riemannian metric
$g$ on $N$ satisfying
\[
 (1-\varepsilon(h))h\leq g\leq h,
 \qquad \mathrm{Sc}_g\geq-n(n-1)
\]
equals $h$.
\end{proposition}

The proof uses the stability theorem \cite[Theorem~1.1]{zbMATH06323440} to evolve $g$ by a rescaled Ricci flow that converges smoothly to $h$, modulo diffeomorphisms. The scalar curvature lower bound makes
volume nonincreasing, so $\mathrm{Vol}(N,g)\geq\mathrm{Vol}(N,h)$.
The reverse inequality follows from $g\leq h$, and equality of
volumes then forces equality of the metrics.

\begin{proof}
Let $\delta(h)>0$ be the stability constant in Bamler's theorem
\cite[Theorem~1.1]{zbMATH06323440}, and set
\[
 \varepsilon(h)=\min\left\{\frac{\delta(h)}2,\frac12\right\}.
\]
This constant depends only on $h$. The assumed metric bounds give
\[
 (1-\delta(h))h\leq(1-\varepsilon(h))h
 \leq g\leq h\leq(1+\delta(h))h,
\]
so $g$ lies in the neighborhood to which Bamler's theorem applies.
The theorem therefore supplies a smooth solution $g_t$, defined
for all $t\geq0$, of the rescaled Ricci flow
\begin{equation}\label{eq:flow}
 \partial_tg_t=-2\mathrm{Ric}_{g_t}-2(n-1)g_t,
 \qquad g_0=g.
\end{equation}
This normalization makes $h$ stationary because
$\mathrm{Ric}_h=-(n-1)h$. The same theorem supplies diffeomorphisms $\Psi_t:N\to N$
such that $\Psi_t^*g_t$ converges smoothly to $h$ on every compact
subset of $N$. Since $N$ itself is compact, this convergence is
global:
\begin{equation}\label{eq:convergence}
 \Psi_t^*g_t\longrightarrow h
 \quad\text{in }C^\infty(N)\quad\text{as }t\to\infty.
\end{equation}

We first derive the scalar curvature evolution under
\eqref{eq:flow}. Throughout this calculation, $\nabla^{g_t}$
denotes the Levi-Civita connection of $g_t$, and
$\Delta_{g_t}=-g_t^{ij}\nabla^{g_t}_i\nabla^{g_t}_j$
is the positive-spectrum Laplacian. For a metric $q$ and a symmetric $2$-tensor $k$, write
\[
 D\mathrm{Sc}_q[k]
 :=\left.\frac{d}{ds}\right|_{s=0}\mathrm{Sc}_{q+sk}
\]
for the first variation of scalar curvature.
The chain rule and linearity of this variation give
\begin{equation}\label{eq:scalar-splitting}
 \begin{aligned}
 \partial_t\mathrm{Sc}_{g_t}
 &=D\mathrm{Sc}_{g_t}[\partial_tg_t]\\
 &=D\mathrm{Sc}_{g_t}[-2\mathrm{Ric}_{g_t}]
   +D\mathrm{Sc}_{g_t}[-2(n-1)g_t].
 \end{aligned}
\end{equation}
We compute these two contributions at a fixed time $t$.

To compute the Ricci contribution, start an auxiliary
unnormalized Ricci flow at $g_t$. Short-time existence supplies a smooth family
$\widehat g_s$, for $0\leq s<\eta$, satisfying
\[
 \partial_s\widehat g_s=-2\mathrm{Ric}_{\widehat g_s},
 \qquad \widehat g_0=g_t.
\]
Here $s$ is an auxiliary time parameter, while $t$ remains fixed.
The initial metric and velocity of this family are
$g_t$ and $-2\mathrm{Ric}_{g_t}$, respectively. The chain rule
therefore gives
\[
 \left.\frac{d}{ds}\right|_{s=0}\mathrm{Sc}_{\widehat g_s}
 =
 D\mathrm{Sc}_{\widehat g_0}
 \left[\left.\partial_s\widehat g_s\right|_{s=0}\right]
 =
 D\mathrm{Sc}_{g_t}[-2\mathrm{Ric}_{g_t}].
\]
The auxiliary family $\widehat g_s$ solves the unnormalized
Ricci flow. Along this family, the scalar curvature
evolution formula \cite[Corollary~7.5]{zbMATH03794902},
written with our positive-spectrum Laplacian, gives
\[
 (\partial_s+\Delta_{\widehat g_s})\mathrm{Sc}_{\widehat g_s}
 =2|\mathrm{Ric}_{\widehat g_s}|_{\widehat g_s}^2.
\]
Evaluating this equation at $s=0$, where $\widehat g_0=g_t$,
and using the preceding chain-rule identity, we obtain
\[
 \begin{aligned}
 D\mathrm{Sc}_{g_t}[-2\mathrm{Ric}_{g_t}]
 &=\left.\frac{d}{ds}\right|_{s=0}
   \mathrm{Sc}_{\widehat g_s}\\
 &=-\Delta_{g_t}\mathrm{Sc}_{g_t}
   +2|\mathrm{Ric}_{g_t}|_{g_t}^2.
 \end{aligned}
\] 

To compute the normalization contribution, consider instead
the scaling family
\[
q_s=e^{-2(n-1)s}g_t.
\]
This family starts at $g_t$ with velocity $-2(n-1)g_t$.
The scaling law for scalar curvature gives
$\mathrm{Sc}_{q_s} =e^{2(n-1)s}\mathrm{Sc}_{g_t}.$
Applying the same chain rule and differentiating at $s=0$ gives
\[
 D\mathrm{Sc}_{g_t}[-2(n-1)g_t]
 =\left.\frac{d}{ds}\right|_{s=0}\mathrm{Sc}_{q_s}
 =2(n-1)\mathrm{Sc}_{g_t}.
\]
Substituting both contributions into
\eqref{eq:scalar-splitting} now gives the evolution equation
along the rescaled flow:
\begin{equation}\label{eq:scalar-evolution}
 (\partial_t+\Delta_{g_t})\mathrm{Sc}_{g_t}
 =2|\mathrm{Ric}_{g_t}|_{g_t}^2
  +2(n-1)\mathrm{Sc}_{g_t}.
\end{equation}

To compare this evolution with the initial scalar-curvature
bound, set
\[
 c=n(n-1),
 \qquad r=\mathrm{Sc}_{g_t}+c.
\]
Using
$\langle\mathrm{Ric}_{g_t},g_t\rangle_{g_t}
=\mathrm{Sc}_{g_t}$ and $|g_t|_{g_t}^2=n$,
we rewrite \eqref{eq:scalar-evolution} as
\begin{equation}\label{eq:defect-evolution}
 \bigl(\partial_t+\Delta_{g_t}+2(n-1)\bigr)r
 =2\bigl|\mathrm{Ric}_{g_t}+(n-1)g_t\bigr|_{g_t}^2\geq0.
\end{equation}

Set $u=e^{2(n-1)t}r$. Since the exponential factor depends only
on time, \eqref{eq:defect-evolution} gives
\[
 (\partial_t+\Delta_{g_t})u
 =e^{2(n-1)t}
  \bigl(\partial_t+\Delta_{g_t}+2(n-1)\bigr)r
 \geq0.
\]
The initial scalar curvature bound gives
\[
 u(\cdot,0)=\mathrm{Sc}_g+c\geq0.
\]
Fix $T>0$. On the closed manifold $N$, the parabolic minimum
principle, equivalently the maximum principle applied to $-u$,
therefore yields
\[
 \min_{x\in N}u(x,t)
 \geq\min_{x\in N}u(x,0)\geq0
 \qquad(0\leq t\leq T).
\]
Since $T$ is arbitrary, $u\geq0$ for all $t\geq0$.
Multiplying by the positive factor $e^{-2(n-1)t}$ gives
$r=e^{-2(n-1)t}u\geq0$. Consequently,
\begin{equation}\label{eq:preservation}
 r(t)\geq0,
 \qquad \mathrm{Sc}_{g_t}\geq-c
 \quad(t\geq0).
\end{equation}

The preserved scalar curvature bound controls the volume along
the flow. Indeed, for $V(t)=\mathrm{Vol}(N,g_t)$, the volume
variation formula and \eqref{eq:flow} give
\begin{equation}\label{eq:volume-monotonicity}
 V'(t)
 =\frac12\int_N\mathrm{tr}_{g_t}(\partial_tg_t)\,d\mu_{g_t}
 =-\int_N r(t)\,d\mu_{g_t}\leq0.
\end{equation}
Thus $V(t)$ is nonincreasing. To identify its limit, use invariance of volume under
diffeomorphisms and the smooth convergence
\eqref{eq:convergence}. Since $N$ is compact,
\[
 V(t)=\mathrm{Vol}(N,\Psi_t^*g_t)
 \longrightarrow\mathrm{Vol}(N,h).
\]
Monotonicity therefore gives
\begin{equation}\label{eq:volume-lower}
 \mathrm{Vol}(N,g)=V(0)
 \geq\lim_{t\to\infty}V(t)=\mathrm{Vol}(N,h).
\end{equation}

The initial domination $g\leq h$ gives the reverse volume
inequality. To see this pointwise, let $A=h^{-1}g$ be the positive
$h$-self-adjoint endomorphism defined by $g(v,w)=h(Av,w)$.
Its eigenvalues satisfy $0<a_i(x)\leq1$, so
\[
 d\mu_g=\rho\,d\mu_h,
 \qquad
 \rho=\sqrt{\mathrm{det}\,A}
      =\sqrt{a_1\cdots a_n}\leq1.
\]
Integrating gives $\mathrm{Vol}(N,g)\leq\mathrm{Vol}(N,h)$.
Together with \eqref{eq:volume-lower}, this forces equality:
\[
 \int_N(1-\rho)\,d\mu_h=0.
\]
The integrand is continuous and nonnegative, so $\rho\equiv1$.
At each point, the eigenvalues $a_i\in(0,1]$ therefore have
product one, which forces every $a_i=1$. Since $A$ is
self-adjoint, $A=I$, and hence $g=h$.
\end{proof}

For a weighted Riemannian manifold \( (M^n, g, e^{-f} \, d\mathrm{Vol}_g) \), where \( f \in C^{\infty}(M^n) \), the weighted scalar curvature \( \mathrm{Sc}_{\alpha, \beta} \) is introduced by the author in \cite{zbMATH07342230} and is defined by
\[
\mathrm{Sc}_{\alpha, \beta} := \mathrm{Sc}_g + \alpha \, \Delta_g f - \beta \lvert \nabla_g f \rvert^2_g,
\]
where $\alpha,\beta\in\mathbb R$ are constants. Since \cite{zbMATH07342230} uses the trace-of-Hessian Laplacian, its parameter corresponding to $\alpha$ is $-\alpha$ in our convention. Under suitable restrictions on the parameters, positivity of the weighted scalar curvature yields vanishing results for harmonic spinors on closed spin manifolds and an analogue of the Schoen--Yau argument for two-sided, stable $f$-minimal hypersurfaces.

\begin{remark}
For $\alpha=-2$ and $\beta=\tfrac{m+1}{m}$ with $m>0$, this definition agrees with Case's weighted scalar curvature. Case \cite{zbMATH06537658} defines the weighted Yamabe constant $\Lambda(g,e^{-f}\,d\mathrm{Vol}_g,m)$ and establishes existence and nonexistence results for the associated minimization problem. It is natural to seek an appropriate Yamabe-type invariant for more general pairs $(\alpha,\beta)$ and to ask whether weighted analogues of the results in this section can be established.

\end{remark}

\section{LCF $4$-manifolds with scalar-flat metrics}\label{5}

Let \( (M^n, g) \) be a closed, oriented, smooth, locally conformally flat Riemannian manifold with \( n \geq 4 \) and positive scalar curvature. Schoen and Yau~\cite[Theorem~4.5]{zbMATH04075988} prove that the  developing (conformal) map of \( (\widetilde{M}^n, \tilde{g}) \) is injective, and that \( \pi_1(M^n) \) admits a faithful holonomy representation into a discrete subgroup of \( \operatorname{Conf}(S^n) \). Let \( \Gamma \) denote the image of \( \pi_1(M^n) \) under the representation induced by the developing map. The image of the developing map is an open subset \( \Omega \subset S^n \) on which \( \Gamma \) acts properly discontinuously, and \( M^n \) is diffeomorphic to the quotient \( \Omega / \Gamma \).

Moreover, \( \Omega \) coincides with the domain of discontinuity of \( \Gamma \), and can be written as \( \Omega = S^n \setminus \Lambda \), where \( \Lambda \subset S^n \) is the limit set of \( \Gamma \).  The set \( \Omega \) serves both as the image of the developing map and as the maximal open set on which the action of \( \Gamma \) is properly discontinuous.

The above results of Schoen and Yau~\cite{zbMATH04075988}   hold in dimensions not less than four under the assumption \( \mathrm{Sc}_g \geq c > 0 \).
  In that paper, they also remark that their result \cite[Proposition 4.4$^{\prime}$]{zbMATH04075988}  would remain valid under the assumption of nonnegative scalar curvature, provided the positive energy theorem can be extended to the case of complete manifolds—namely, when \( M \) has one asymptotically flat end and other ends that are merely complete.  
Recently, Lesourd, Unger, and Yau~\cite[Theorem~1.2]{MR4836036} use minimal hypersurfaces to prove that there does not exist a complete Riemannian metric with positive scalar curvature on \( T^3 \# X \), where \( X \) is an arbitrary (possibly noncompact) manifold. Furthermore, they show that the nonexistence of a complete smooth metric with positive scalar curvature on \( T^n \# M^n \) implies the Liouville theorem below in dimension \(n\)~\cite[Theorem~1.6]{MR4836036}. Schoen and Yau \cite{zbMATH04075990} prove this nonexistence result for all $n \leq 7$; see also \cite[Theorem~3]{zbMATH07817078} for a more detailed exposition of their proof. For $n \geq 7$, the Liouville theorem is due to Schoen and Yau~\cite[Proposition~3.3(iii)]{zbMATH04075988}. Combining these results gives the following Schoen–Yau Liouville theorem in all dimensions:

\begin{theorem}[Liouville theorem]\label{Liouville theorem}
Let \( (M^n, g) \), \( n \ge 3 \), be a complete, locally conformally flat manifold with nonnegative scalar curvature.  
If \( \Phi : M^n \to S^n \) is a conformal map, then \( \Phi \) is injective and the boundary \( \partial \Phi(M) \) has zero Newtonian capacity (i.e.,  $2$-capacity).
\end{theorem}

In another paper, Lesourd, Unger, and Yau~\cite[Theorem~1.2]{MR4773185} use $\mu$-bubbles to prove the above Schoen–Yau conjecture for $3 \le n \le 7$ (for an asymptotically Schwarzschild end) and obtain a new proof of the Liouville theorem in these dimensions, which is completely in the spirit of the approach outlined in~\cite{zbMATH04075988}.

 A Kleinian group $\Gamma$ is called \textit{elementary} if its limit set $\Lambda$ is empty or consists of at most two points; otherwise, it is called \textit{non-elementary}. If $\Lambda$ is empty, then $\Gamma$ is finite. If $\Lambda$ consists of a single point, then $\Gamma$ contains an abelian subgroup of finite index of rank $k$ with $k \leq n$. If $\Lambda$ consists of two points, then $\Gamma$ contains an infinite cyclic subgroup of finite index. If $\Gamma$ is non-elementary, then $\Lambda(\Gamma)$ is the unique minimal nonempty closed \( \Gamma \)-invariant subset of \( S^n \).

 Let \( B^{n+1} := \{ x \in \mathbb{R}^{n+1} \mid |x| < 1 \} \) be the Poincaré ball model endowed with the hyperbolic metric  $
g_{H} = 4(1 - |x|^2)^{-2} \sum_{i=1}^{n+1} (dx^i)^2.
$ 
Every element of the conformal group \( \mathrm{Conf}(S^n) \) extends to a diffeomorphism of the closed ball \( \overline{B^{n+1}} := B^{n+1} \cup S^n \), and restricts to an isometry of \( B^{n+1} \) with respect to the hyperbolic metric. Conversely, every isometry in \( \mathrm{Isom}(B^{n+1}) \) extends continuously to the boundary \( S^n \), acting as a conformal transformation of the standard sphere \( (S^n, g_{st}) \). Therefore, there is a natural group isomorphism: $\mathrm{Isom}(B^{n+1}) \cong \mathrm{Conf}(S^n).$

For a Kleinian group \( \Gamma \), the \emph{critical exponent} \( \delta(\Gamma) \) is defined by  
\[
\delta(\Gamma) := \inf\left\{ s > 0 \,\middle|\, \sum_{\gamma \in \Gamma} \exp\left( -s \, \mathrm{dist}(x, \gamma y) \right) < \infty \right\},
\]
where  $x, y \in B^{n+1}$ and \( \mathrm{dist}(\cdot, \cdot) \) denotes the hyperbolic distance function on the Poincaré ball model \( B^{n+1} \). The value of \( \delta(\Gamma) \) is independent of the particular choice of \( x \) and \( y \). If \( \Gamma \) is non-elementary, then \( 0 < \delta(\Gamma) \leq n \). If \( \Gamma'  \leq \Gamma \) is a subgroup, then   $\delta(\Gamma') \leq \delta(\Gamma).$

A non-elementary Kleinian group \( \Gamma \) is said to be \emph{convex cocompact} if the quotient
$(\Omega(\Gamma) \cup B^{n+1}) / \Gamma$ is compact. Equivalently, \( \Gamma \) is convex cocompact if the hyperbolic convex hull \( \mathrm{CH}(\Gamma) \subset B^{n+1} \) of the limit set \( \Lambda(\Gamma) \) has a nonempty compact quotient \( \mathrm{CH}(\Gamma)/\Gamma \). Here, \( \mathrm{CH}(\Gamma) \) denotes the smallest closed convex subset of \( B^{n+1} \) whose closure in the compactified ball \( \overline{B^{n+1}} := B^{n+1} \cup S^n \) contains \( \Lambda(\Gamma) \).

A Kleinian group \( \Gamma \) is said to be \emph{geometrically finite} if there exists a uniform bound on the orders of its finite subgroups and the \( \epsilon \)-neighborhood of \( \mathrm{CH}(\Gamma)/\Gamma \) in \( B^{n+1}/\Gamma \) has finite volume for some \( \epsilon > 0 \).

For non-elementary Kleinian groups, convex cocompactness is equivalent to geometric finiteness together with the absence of parabolic elements.

If \( \Gamma \) is a non-elementary geometrically finite Kleinian group, then the Patterson–Sullivan theorem~\cite[Theorem~1]{zbMATH03903608} states that
\begin{equation*}
\delta(\Gamma) = \dim_{\mathcal{H}}(\Lambda(\Gamma)),
\end{equation*}
where \( \dim_{\mathcal{H}} \) denotes the Hausdorff dimension of the limit set \( \Lambda(\Gamma) \).

We now study the topological types of locally conformally flat \(4\)-manifolds with scalar-flat metrics. One reason this is interesting is that the Riemannian connection associated with a locally conformally flat scalar-flat metric achieves an absolute minimum of the Yang–Mills functional on a closed oriented \(4\)-manifold.
Based on the topological classification of orientable, simply connected, closed \(4\)-manifolds, LeBrun and Maskit~\cite[Corollary~1.2]{zbMATH05319680} characterize the manifolds that admit scalar-flat anti-self-dual metrics as follows: A compact, simply connected topological \(4\)-manifold \(M\) admits a smooth structure supporting a scalar-flat anti-self-dual metric \(g\) if and only if \(M\) is homeomorphic to one of the following:
$K3$,  $\mathbb{CP}^2 \# k\, \overline{\mathbb{CP}}^{\,2} \ (k \ge 10)$, or $ k\, \overline{\mathbb{CP}}^{\,2} \ (k \ge 5).$

We now consider a closed \(4\)-manifold \(M^4\) with a locally conformally flat scalar-flat metric \(g\), without assuming simple connectedness.  By the Gauss–Bonnet–Chern formula, one has \(\chi(M^4) \le 0\), with equality if and only if \(g\) is flat. In 1978, Brown, Bülow, Neubüser, Wondratschek, and Zassenhaus complete a computer-assisted classification of all isomorphism classes of \(4\)-dimensional crystallographic groups, thereby classifying closed flat \(4\)-manifolds \cite{MR484179}. They obtain \(74\) homeomorphism classes of closed flat \(4\)-manifolds, consisting of \(27\) orientable classes and \(47\) non-orientable classes. 
 
From now on, suppose that \(\chi(M^4) < 0\), and let \(g\) be a locally conformally flat scalar-flat metric on \(M^4\). In this case, the Ricci tensor of \(g\) is not identically zero, since otherwise \(\chi(M^4) = 0\). 
Thus, the metric \(g\) can be deformed to a metric \(g'\) with positive scalar curvature. 
However, the metric \(g'\) may not have a vanishing Weyl tensor, 
because the existence of a locally conformally flat metric with positive scalar curvature implies that the second Betti number \(b_2(M)\) vanishes. 

Moreover, even when \(b_2(M) = 0\), the metric \(g\) cannot, in general, be deformed to a locally conformally flat metric with positive scalar curvature, as the following example shows.

\begin{example}\label{example scalar flat}
Let \((\Sigma, g_H)\) be a closed oriented hyperbolic surface that admits an orientation-reversing isometric involution \(r: \Sigma \to \Sigma\). Consider \((X, g) = (S^2 \times \Sigma, g_{st} \oplus g_H)\), and let \(A: (S^2, g_{st}) \to (S^2, g_{st})\) be the antipodal map, 
which is an orientation-reversing fixed-point-free isometry. 
Then \(F := A \times r\) is a fixed-point-free isometric involution of
\((X, g) = (S^2 \times \Sigma, g_{st} \oplus g_H)\), so \(\langle F\rangle \cong \mathbb{Z}/2\mathbb{Z}\) acts freely and isometrically on \(X\). 
Since \(\mathrm{deg}(F)=\mathrm{deg}(A)\mathrm{deg}(r) = (-1)\cdot(-1) = +1\), the map \(F\) is orientation-preserving. 

Let \(M := X / \langle F \rangle\). 
Because \(F\) is an isometry, the metric \(g\) descends to a smooth metric \(\bar{g}\) on \(M\); and since both the locally conformally flat and scalar-flat properties are local, \((M, \bar{g})\) is still locally conformally flat and scalar-flat. For a free finite group action one has $H^2(M; \mathbb{Q}) \cong H^2(X; \mathbb{Q})^{\langle F \rangle}$ (the $\langle F\rangle$-invariant subspace of $H^2(X; \mathbb{Q})$).
Now, $H^2(X; \mathbb{Q}) \cong H^2(S^2; \mathbb{Q}) \oplus H^2(\Sigma; \mathbb{Q}),$
which is generated by the area classes of the two factors. 
Since both \(A\) and \(r\) reverse orientation on their respective factors, 
\(F^*\) acts by \(-1\) on each summand. 
Hence, the \(\langle F \rangle\)-invariant subspace of \(H^2(X; \mathbb{Q})\) is trivial, and therefore $b_2(M) = \dim H^2(M; \mathbb{Q}) = 0.$ Thus, \((M, \bar{g})\) is a closed, nonflat, locally conformally flat, scalar-flat \(4\)-manifold with \(b_2(M) = 0\).

The manifold \(M\) does not admit a locally conformally flat metric with positive scalar curvature.
Indeed, such a metric would pull back to a locally conformally flat metric with positive scalar curvature on the finite cover \(X = S^2 \times \Sigma\).
The Bochner--Weitzenb\"ock formula for \(2\)-forms would then imply \(b_2(X)=0\), contradicting \(b_2(X)=2\).
\end{example}

For a closed, orientable, locally conformally flat \(2n\)-manifold $N^{2n}$ \((n \ge 2)\) with nonnegative scalar curvature, 
Noronha~\cite[Theorem~2]{MR1235219} analyzes the possible holonomy groups case by case and shows that either \(b_n(N^{2n}) = 0\), or \(N^{2n}\) is isometrically covered, up to homothety, by \(\mathbb{E}^{2n}\) or by \(S^n \times \mathbb{H}^n\).  

Furthermore, by imposing the additional condition that, on a closed locally conformally flat scalar-flat \(4\)-manifold, the largest eigenvalue of the Ricci operator is not greater than the absolute value of its smallest eigenvalue, Noronha applies the Weitzenböck formula to \(2\)-forms to show that \(\nabla \mathrm{Rm} \equiv 0\) in that paper; that is, the manifold is locally symmetric, and is isometrically covered, up to homothety, by either \(\mathbb{E}^4\) or \(S^2 \times \mathbb{H}^2\).

Noronha's results are motivated by the following question posed in her paper concerning compact \(4\)-manifolds:
\begin{quote}
If \( \pi_2(M^4) \neq 0 \) and \( M^4 \) admits a conformally flat metric with zero scalar curvature, is \( M^4 \) covered by \( S^2 \times \mathbb{H}^2 \)?
\end{quote}

The following theorem answers her question affirmatively.

\begin{theorem} \label{the universal cover}
Let \( (M^4, g) \) be a closed, locally conformally flat, and scalar-flat \(4\)-manifold with \( \pi_2(M^4) \neq 0 \).  Then its Riemannian universal cover \( (\widetilde{M}, \tilde{g}) \) is isometric to $(\mathbb{H}^2  \times S^2 ,  g_{H} \oplus g_{st})$ up to homothety.
\end{theorem}

\begin{proof}
The manifold \( M^4 \) admits neither a flat metric nor a locally conformally flat metric with positive scalar curvature, since \( \pi_2(M^4) \neq 0 \). Apply Theorem~\ref{Liouville theorem} to the developing map of the complete Riemannian universal cover \((\widetilde M,\tilde g)\). This map is a conformal embedding whose image has boundary of zero Newtonian capacity. Its injectivity identifies the deck group with a faithful, discrete holonomy image \(\Gamma\). The compact-quotient argument in Schoen and Yau~\cite[proof of Theorem~4.1]{zbMATH04075988} identifies the developing image with the full domain of discontinuity \(\Omega(\Gamma)\). Thus we have the diffeomorphisms
$M \cong \Omega / \Gamma$, $\widetilde{M} \cong \Omega = S^4 \setminus \Lambda(\Gamma).$

The group \( \Gamma \) is non-elementary.  Otherwise, \(\Lambda(\Gamma)\) would have at most two points, which would imply \(\pi_2(S^4 \setminus \Lambda(\Gamma)) = 0\).  
However, by assumption, \(\pi_2(\widetilde{M}) \cong \pi_2(\Omega) = \pi_2(S^4 \setminus \Lambda(\Gamma)) \neq 0\), leading to a contradiction.

The group \( \Gamma \) is geometrically finite. 
Since \( M \cong \Omega / \Gamma \) is compact and admits a locally conformally flat, scalar-flat metric, Schoen and Yau~\cite[Theorem~4.7]{zbMATH04075988} give 
\[
\dim_{\mathcal{H}}(\Lambda(\Gamma)) \le 1 < 4.
\]
Chang, Qing, and Yang~\cite[Theorem 0.1]{MR2070141} show that if \( \Gamma \) is a non-elementary, finitely generated, conformally finite  subgroup of \( \mathrm{Conf}(S^n) \), then \( \Gamma \) is geometrically finite if and only if \( \dim_\mathcal{H}(\Lambda(\Gamma)) < n \). Here,  conformally finite means that 
$\Omega(\Gamma)/\Gamma$ is the disjoint union of a compact set and finitely many standard conformal cusp ends.  Since \( M \cong \Omega / \Gamma \) is a closed manifold, the group \( \Gamma \) is finitely generated and conformally finite. Combining this with \( \dim_{\mathcal{H}}(\Lambda(\Gamma)) \le 1 < 4 \), we conclude that \( \Gamma \) is geometrically finite.

The Čech–Alexander duality theorem states that for any nonempty compact subset 
\( K \subset S^n \),
\[
\widetilde{H}_k(S^n \setminus K; G) 
\cong 
\check{\widetilde{H}}^{\,n-k-1}(K; G),
\]
that is, the reduced singular homology of the complement is isomorphic to the reduced Čech cohomology of the set,
for every abelian coefficient group \( G \) and every \( k \ge 0 \).

Since \( \widetilde{M} \) is simply connected and \( \pi_2(M) \neq 0 \), 
the Hurewicz theorem gives
\[
0\neq \pi_2(\widetilde{M}) \cong H_2(\widetilde{M}; \mathbb{Z}) \cong \widetilde{H}_2(S^4 \setminus \Lambda(\Gamma); \mathbb{Z})  \cong \check{\widetilde{H}}^{\,1}(\Lambda; \mathbb{Z})\cong \check{H}^{\,1}(\Lambda; \mathbb{Z}).
\]
Hence \( H_2(\widetilde{M}; \mathbb{Z}) \neq 0 \) implies $\check{H}^{\,1}(\Lambda; \mathbb{Z}) \neq 0.$

For a compact metric space $X$ and an abelian group $G$, the cohomological dimension $\dim_G X$ is defined by
\[
\dim_G X := \sup \{\, n : \exists\, A \subset X \text{ closed such that } \check{H}^{\,n}(X, A; G) \neq 0 \,\}.
\]

For such $X$ and $G$, one has
\[
\dim_G X \leq \dim_{\mathbb{Z}} X \leq \dim X,
\]
where \( \dim X \) denotes the topological (Lebesgue covering) dimension of \( X \).

In fact, by the Alexandroff theorem, if \( X \) is a compact metric space of finite topological dimension, then $\dim_{\mathbb{Z}} X = \dim X.$

Thus, \( \check{H}^{\,1}(\Lambda; \mathbb{Z}) \neq 0 \) implies that 
\( \dim_{\mathbb{Z}}(\Lambda(\Gamma)) \geq 1 \), 
and hence the topological dimension of the limit set satisfies 
\( \dim(\Lambda(\Gamma)) \geq 1 \).
Therefore,
\[
1 
\le \dim(\Lambda(\Gamma)) 
\le \dim_{\mathcal{H}}(\Lambda(\Gamma)) 
\le 1
\;\implies\;
\dim(\Lambda(\Gamma)) 
= \dim_{\mathcal{H}}(\Lambda(\Gamma)) = 1.
\]
Then, by the Patterson–Sullivan theorem, we have 
$\delta(\Gamma) = \dim_{\mathcal{H}}(\Lambda(\Gamma)) = 1,$ where \( \delta(\Gamma) \) is the critical exponent of \( \Gamma \).

Finally, we need the following theorem of Kapovich~\cite[Theorem~1.3]{MR2491697}:  

Suppose that \( \Gamma \subset \mathrm{Isom}(\mathbb{H}^n) \) is a non-elementary, geometrically finite group such that the Hausdorff dimension of its limit set equals its topological dimension \(d\).  
Then the limit set of \( \Gamma \) is a round \(d\)-sphere.  

Thus, in our case, \( \Lambda(\Gamma) \) is a round circle.  Consequently, $(\widetilde{M},\tilde{g})$ is conformal to $(\Omega = S^4 \setminus S^1, g_{st}|_{S^4 \setminus S^1})$.  By Lemma~\ref{Stereographic}, \( \widetilde{M} \) is conformal to $(\mathbb{H}^2  \times S^2 ,  g_{H} \oplus g_{st}).$ Finally, Lemma~\ref{iso} implies that \( \widetilde{M} \) is isometric to  
\( (\mathbb{H}^{2} \times S^{2},\, g_{H} \oplus g_{st}) \) up to homothety.
\end{proof}

\begin{lemma}\label{Stereographic}
Let $m\ge 3$, $1\le k\le m-2$, and let $S^k\subset S^m$ be a round subsphere. Then
$$(S^{m}\setminus S^{k}, g_{st}|_{S^{m}\setminus S^{k}})\ \cong_{\mathrm{conf}}\ (\mathbb H^{k+1}\times S^{\,m-k-1}, g_{H}\oplus g_{st}).$$
In particular, for $m=2n$ and $k=n-1$,
$S^{2n}\setminus S^{n-1}\ \cong_{\mathrm{conf}}\ \mathbb H^{n}\times S^{n}.$
\end{lemma}

\begin{proof}

Fix $p\in S^{k}\subset S^{m}$ and let $\sigma_p:\ S^{m}\setminus\{p\}\longrightarrow \mathbb R^{m}$ be stereographic projection from $p$. If a round subsphere passes through $p$,
its image under $\sigma_p$ is an affine linear subspace. Hence
$\Pi:=\sigma_p(S^{k}\setminus\{p\})$ is an affine $k$-plane in $\mathbb R^{m}$ and
\[
\sigma_p:\ S^{m}\setminus S^{k}\xrightarrow{\ \cong\ }\mathbb R^{m}\setminus \Pi.
\]
Choose a Euclidean isometry $T:\mathbb R^m\to\mathbb R^m$ taking $\Pi$ to $\Pi_0:=\mathbb R^k\times\{0\}$.
In the resulting coordinates, write $x=(u,v)$ with $u\in\mathbb R^{k}$ and $v\in\mathbb R^{m-k}$.
Then $\mathbb R^{m}\setminus\Pi_0\ \cong\ \mathbb R^{k}\times\big(\mathbb R^{m-k}\setminus\{0\}\big),$ and since $m-k\ge 2$, the complement is connected.

For convenience in this proof, denote by \( g_H = g_{\mathbb{H}^{k+1}} \) the hyperbolic metric on \( \mathbb{H}^{k+1} \), and by \( g_{\mathrm{st}} = g_{S^{m-k-1}} \) the standard round metric on \( S^{m-k-1} \). Write $v=\rho\,\omega$ with $\rho:=|v|>0$ and $\omega\in S^{m-k-1}$. The Euclidean metric splits as
\[
ds_{\mathbb E^{m}}^{2}=|du|^{2}+d\rho^{2}+\rho^{2}\,d\omega^{2}.
\]
Rescaling by $\rho^{-2}$ yields
\begin{equation*}
\rho^{-2}\,ds_{\mathbb E^{m}}^{2}
=\frac{|du|^{2}+d\rho^{2}}{\rho^{2}}+d\omega^{2}.
\end{equation*}
Identify $\big(\mathbb R^{k}\times(0,\infty), (|du|^{2}+d\rho^{2})/\rho^{2}\big)$ with the upper half-space
model of $\mathbb H^{k+1}$ and $(S^{m-k-1},d\omega^{2})$ with the round sphere. Thus
\[
\rho^{-2}\,ds_{\mathbb E^{m}}^{2} \;=\; g_{\mathbb H^{k+1}}\oplus g_{S^{m-k-1}}.
\]
Define
\[
J:\ \mathbb R^{m}\setminus\Pi_0\longrightarrow \mathbb H^{k+1}\times S^{m-k-1},\qquad
J(u,\rho,\omega)=\big((u,\rho),\,\omega\big),
\]
so that
\begin{equation}\label{eq:pull}
J^{*}\big(g_{\mathbb H^{k+1}}\oplus g_{S^{m-k-1}}\big)=\rho^{-2}\,ds_{\mathbb E^{m}}^{2}.
\end{equation}

Stereographic projection $\sigma_p$ is conformal, and $T$ is a Euclidean isometry. By \eqref{eq:pull}, $J$ is conformal. Consequently,
$F:=J\circ T\circ\sigma_p:\ S^{m}\setminus S^{k}\to \mathbb H^{k+1}\times S^{m-k-1}$
is a conformal diffeomorphism, proving
\[
S^{m}\setminus S^{k}\ \cong_{\mathrm{conf}}\ \mathbb H^{k+1}\times S^{m-k-1}.
\]
\end{proof}

\begin{lemma}\label{iso}
Suppose the Riemannian universal cover \( (\widetilde{M}, \tilde{g}) \) of a closed manifold \( (M^{2n}, g) \) ($n\geq 2$) with \( \mathrm{Sc}_{\tilde{g}} = 0 \)  
is conformal to
$
\bigl( S^{n} \times \mathbb{H}^{n},\;
g_{0} := g_{st} \oplus g_{H} \bigr).$
Then there exists a constant \( c > 0 \) such that
$(\widetilde{M}, \tilde{g})$  is isometric to 
$\bigl( S^{n} \times \mathbb{H}^{n},\, c^{2} g_{0} \bigr).
$
\end{lemma}

\begin{proof}
Let $F : (\widetilde M, \tilde g) \longrightarrow S^{n} \times \mathbb{H}^{n}$
be a conformal diffeomorphism.  
Then there exists a smooth function 
$\varphi : \widetilde M \longrightarrow \mathbb{R}$
such that
\begin{equation*}
F^{*}(g_{0}) = e^{-2\varphi}\,\tilde g .
\end{equation*}
Equivalently,
$\tilde g = e^{2\varphi}\, F^{*}(g_{0}).$

Since $\mathrm{Sc}_{g_{0}} = 0$ and $\mathrm{Sc}_{\tilde g}=0$, the standard conformal-change formula for the scalar curvature, with our convention $\Delta=-\operatorname{div}\nabla$, implies that
\[
\Delta_{\tilde g}\varphi 
= -(n-1)\,|\nabla\varphi|^{2}_{\tilde g} \;\leq\; 0
\quad\text{on } \widetilde M.
\]

Let $\Gamma = \pi_{1}(M)$ denote the group of deck transformations of the universal covering 
$\widetilde M \to M$. 
Each $\gamma \in \Gamma$ acts on $\widetilde M$ by an isometry of $(\widetilde M,\tilde g)$.
Define
\[
f_{\gamma} := F \circ \gamma \circ F^{-1} : 
S^{n} \times \mathbb H^{n} \longrightarrow S^{n} \times \mathbb H^{n}.
\]
We now compute the conformal factor of $f_{\gamma}$ with respect to $g_{0}$.

Let $\psi := \varphi \circ F^{-1} : S^{n} \times \mathbb H^{n} \to \mathbb{R}$.
Then we can rewrite
\[
\tilde g = e^{2\varphi}\,F^{*}(g_{0})
\qquad\Longrightarrow\qquad
(F^{-1})^{*}\tilde g = e^{2\psi}\,g_{0}.
\]
Since $\gamma$ is an isometry of $(\widetilde M,\tilde g)$, we have $\gamma^{*}\tilde g = \tilde g$, and hence
\[
\gamma^{*}F^{*}g_{0}
= \gamma^{*}\bigl(e^{-2\varphi}\tilde g\bigr)
= e^{-2\varphi\circ\gamma}\,\tilde g.
\]
Using $f_{\gamma} = F\circ\gamma\circ F^{-1}$ and pulling back by $F^{-1}$, we obtain
\begin{align*}
f_{\gamma}^{*}g_{0}
&= (F\gamma F^{-1})^{*}g_{0}
= (F^{-1})^{*}\bigl(\gamma^{*}F^{*}g_{0}\bigr) \\
&= (F^{-1})^{*}\bigl(e^{-2\varphi\circ\gamma}\tilde g\bigr) \\
&= e^{-2(\varphi\circ\gamma)\circ F^{-1}} \,(F^{-1})^{*}\tilde g \\
&= e^{-2\psi\circ f_{\gamma}} \, e^{2\psi} g_{0}
= e^{2(\psi - \psi\circ f_{\gamma})}\,g_{0}.
\end{align*}
Thus each $f_{\gamma}$ is a conformal diffeomorphism of $(S^{n} \times \mathbb H^{n},g_{0})$ with conformal factor $e^{2(\psi - \psi\circ f_{\gamma})}$.

By a result of Jimenez and Tojeiro \cite[Corollary~2]{MR4374774}, every conformal diffeomorphism of a product 
$\mathbb H^{k} \times S^{\ell}$ is in fact an isometry. 
Applying this to $f_{\gamma}$, we conclude that $f_{\gamma}$ is an isometry of $(S^{n} \times \mathbb H^{n},g_{0})$ for every $\gamma \in \Gamma$. 
Hence its conformal factor must be identically $1$, so that
\[
\psi \circ f_{\gamma} = \psi
\qquad \text{for all } \gamma\in\Gamma.
\]

Returning to $\varphi$ on $\widetilde M$, we note that
\[
\psi \circ f_{\gamma} 
= \psi \quad\Longleftrightarrow\quad 
(\varphi \circ F^{-1}) \circ (F\circ\gamma\circ F^{-1}) = \varphi \circ F^{-1}
\quad\Longleftrightarrow\quad 
\varphi \circ \gamma = \varphi.
\]
Thus $\varphi$ is $\Gamma$-invariant and descends to a smooth function
$\tilde\varphi : M \to \mathbb R$ on the closed manifold $M$, with
$\varphi = \tilde\varphi \circ \pi$, where $\pi:\widetilde M\to M$ is the covering projection.
Because $\pi : (\widetilde M,\tilde g) \to (M,g)$ is a Riemannian covering, it is a local isometry and 
$\tilde g = \pi^{*}g$. In particular, for any smooth function $\tilde\varphi$ on $M$ one has
\[
d\pi\bigl(\nabla_{\tilde g}(\tilde\varphi \circ \pi)\bigr) = (\nabla_{g}\tilde\varphi)\circ \pi,
\qquad
\Delta_{\tilde g}(\tilde\varphi \circ \pi) = (\Delta_{g}\tilde\varphi)\circ \pi,
\]
and therefore
\[
|\nabla\varphi|^{2}_{\tilde g}
= |\nabla(\tilde\varphi\circ\pi)|^{2}_{\tilde g}
= \bigl(|\nabla\tilde\varphi|^{2}_{g}\bigr)\circ\pi.
\]
Substituting $\varphi = \tilde\varphi\circ\pi$ into the identity
\[
\Delta_{\tilde g}\varphi = -(n-1)\,|\nabla\varphi|^{2}_{\tilde g}
\quad\text{on } \widetilde M,
\]
we obtain
\[
\bigl(\Delta_{g}\tilde\varphi\bigr)\circ\pi
= -(n-1)\,\bigl(|\nabla\tilde\varphi|^{2}_{g}\bigr)\circ\pi.
\]
Since $\pi$ is surjective, it follows that
\begin{equation*}
\Delta_{g}\tilde\varphi = -(n-1)\,|\nabla\tilde\varphi|^{2}_{g}
\qquad\text{on } M.
\end{equation*}

Since $\tilde\varphi$ is now globally defined on $M$, we may integrate the identity
\[
\Delta_{g}\tilde\varphi = -(n-1)\,|\nabla\tilde\varphi|^{2}_{ g}
\]
over $M$:
\[
0 = \int_{M} \Delta_{ g}\tilde\varphi\, d\mathrm{vol}_{ g}
= -(n-1) \int_{M} |\nabla\tilde\varphi|^{2}_{g}\, d\mathrm{vol}_{ g}.
\]
It follows that $|\nabla\tilde\varphi|_{ g} \equiv 0$ on $M$, and hence $\varphi$ is constant on $\widetilde M$ as well. 
Set $c := e^{\varphi} > 0$. 
Thus, we obtain
\[
\tilde g = e^{2\varphi} F^{*} g_{0} = c^{2} F^{*}g_{0},
\]
and, upon identifying $\widetilde M$ with $S^{n} \times \mathbb H^{n}$ via $F$, we conclude
$\tilde g = c^{2} g_{0}.$
Thus $(\widetilde M,\tilde g)$ is isometric to $\bigl(S^{n} \times \mathbb H^{n}, c^{2} g_{0}\bigr)$, 
as claimed.
\end{proof}

\begin{remark}
The condition \(\pi_2(M^4) \neq 0\) in Theorem~\ref{the universal cover} cannot be replaced by the nonexistence of a flat metric.  
For example, LeBrun and Maskit~\cite[Proposition~3.3]{zbMATH05319680} show that there exists a smooth family of metrics \(h_t\) on \(M:=(S^1 \times S^3) \# (S^1 \times S^3)\), \(t \in [-1,1]\), such that for each \(t\), the metric \(h_t\) is locally conformally flat, with $Y(M, [h_1]) > 0$ and  $Y(M, [h_{-1}]) < 0.$ On the other hand, Große and Nardmann~\cite[Theorem~1.1]{MR3192307} show that if \(N\) is compact, then the Yamabe constant \(Y(N,[g])\) is continuous at \(g\) with respect to the compact-open \(C^2\)-topology.  Therefore, there exists \(s \in (-1,1)\) such that \(Y(M, [h_s]) = 0\).  
Thus, \(h_s\) is conformal to a locally conformally flat scalar-flat metric. The manifold admits no flat metric: its fundamental group is the free group on two generators, which is not virtually abelian, contrary to Bieberbach's theorem for closed flat manifolds. However, the universal cover of \((S^1 \times S^3) \# (S^1 \times S^3)\) is homeomorphic to the complement of a Cantor set in \(S^4\). Its second homotopy group is trivial, so it is not of the form \(S^2 \times \mathbb{H}^2\).  
 
 LeBrun and Maskit’s examples also show that the Hausdorff dimension of the limit set of a Kleinian group depends not only on the group’s algebraic structure but also on its specific representation in the isometry group of hyperbolic space. Let \(\rho_t\) denote the holonomy representation associated with \(h_t\). For instance, for 
\[
(M, h_1) = \bigl((S^1 \times S^3) \# (S^1 \times S^3),\, h_1\bigr),
\quad \Gamma_1 = \rho_1(\pi_1(M)),
\]
one has \( \dim_{\mathcal{H}}(\Lambda(\Gamma_1)) < 1 \) by \cite[Lemma~2.1]{zbMATH05319680},
whereas for 
\[
(M, h_s) = \bigl((S^1 \times S^3) \# (S^1 \times S^3),\, h_s\bigr),
\quad \Gamma_s = \rho_s(\pi_1(M)),
\]
one obtains \( \dim_{\mathcal{H}}(\Lambda(\Gamma_s)) = 1 \) by the same lemma.

\end{remark}

In fact, the argument in the proof of Theorem~\ref{the universal cover} extends to higher dimensions if we replace the condition \( \pi_2(M) \neq 0 \) by \( \widetilde{H}_n(\widetilde{M}^{2n}; \mathbb{Z}) \neq 0 \).

\begin{corollary}\label{2n}
Let \( (M^{2n}, g) \) (\( n \geq 2 \)) be a closed, locally conformally flat, and scalar-flat \( 2n \)-manifold satisfying \( \widetilde{H}_n(\widetilde{M}; \mathbb{Z}) \neq 0 \).  
Then its Riemannian universal cover \( (\widetilde{M}, \tilde{g}) \) is isometric to 
\( (\mathbb{H}^n \times S^n,\, g_{H} \oplus g_{st}) \) up to homothety.
\end{corollary}

\begin{proof}
As in the proof of Theorem~\ref{the universal cover}, Theorem~\ref{Liouville theorem} and the compact-quotient argument in Schoen and Yau~\cite[proof of Theorem~4.1]{zbMATH04075988} identify the universal cover with the full domain of discontinuity. Thus we have the diffeomorphisms
$M^{2n} \cong \Omega / \Gamma$ and $\widetilde{M} \cong \Omega = S^{2n} \setminus \Lambda(\Gamma).$
By the Čech–Alexander duality theorem, 
$\widetilde{H}_n(\widetilde{M}; \mathbb{Z}) \cong \check{H}^{\,n-1}(\Lambda; \mathbb{Z}) \neq 0.$
Hence, \( \dim_\mathbb{Z}(\Lambda(\Gamma)) =\dim(\Lambda(\Gamma)) \geq n-1 \), which implies that \( \Gamma \) is non-elementary.  

By Schoen and Yau~\cite[Theorem~4.7]{zbMATH04075988}, the existence of a locally conformally flat scalar-flat metric on the compact quotient implies 
\( \dim_{\mathcal{H}}(\Lambda(\Gamma)) \leq \tfrac{2n-2}{2} = n-1 < 2n \).
Thus, by the theorem of Chang–Qing–Yang, \( \Gamma \) is geometrically finite and
$\dim(\Lambda(\Gamma)) = \dim_{\mathcal{H}}(\Lambda(\Gamma)) = n-1.$  Kapovich's theorem therefore implies that \( \Lambda(\Gamma) \) is a round \( (n-1) \)-sphere. Lemma~\ref{Stereographic} then gives 
$
S^{2n} \setminus S^{n-1} \cong_{\mathrm{conf}}
(\mathbb{H}^{n} \times S^{n},\, g_{H} \oplus g_{st}),
$
and Lemma~\ref{iso} implies that \( \widetilde{M} \) is isometric, up to homothety, to 
$(\mathbb{H}^{n} \times S^{n},\, g_{H} \oplus g_{st}).$
This completes the proof.
\end{proof}

\begin{remark}
The proof of Corollary~\ref{2n}, as well as the statement itself, differs from Noronha’s result mentioned above, which asserts that if \( b_n(M^{2n}) > 0 \), then \( M^{2n} \) is isometrically covered, up to homothety, by either \( \mathbb{E}^{2n} \) or \( S^n \times \mathbb{H}^n \).  
Indeed, the conditions \( b_n(M^{2n}) > 0 \) and \( \widetilde{H}_n(\widetilde{M}; \mathbb{Z}) \neq 0 \) are not equivalent in general.  
For example, the manifold \( M := X / \langle F \rangle \) in Example~\ref{example scalar flat} admits a locally conformally flat scalar-flat metric and satisfies \( \widetilde{H}_2(\widetilde{M}; \mathbb{Z}) \neq 0 \), yet \( b_2(M) = 0 \).  
Conversely, for the torus \( T^{2n} \), one has \( b_n(T^{2n}) > 0 \), but \( \widetilde{H}_n(\widetilde{T^{2n}}; \mathbb{Z}) = 0 \).
\end{remark}

The proof of Corollary~\ref{2n} does not extend to odd dimensions: the standard product
\((\mathbb{H}^{m} \times S^{n},\, g_{H} \oplus g_{st})\), with sectional curvatures \(-1\) and \(+1\) on the respective factors, has scalar curvature \(n(n-1)-m(m-1)\). Thus the scalar-flat condition forces \(m = n\).  It remains unclear how to characterize odd-dimensional locally conformally flat, scalar-flat, non-flat manifolds. We can, however, show that their fundamental groups must be infinite.

\begin{proposition}
Assume that $(M^n, g)$ is a closed, locally conformally flat, scalar-flat $n$-manifold ($n \ge 3$). Then the fundamental group of $M$ is infinite.
\end{proposition}
\begin{proof}
If $\pi_1(M)$ is finite, then the Riemannian universal cover $(\widetilde{M},\widetilde{g})$ is conformally equivalent to $(S^n, g_{st})$. Hence
$\widetilde{g} = u^{\frac{4}{n-2}} g_{st}$
for some $0 < u \in C^\infty(S^n)$, and $\widetilde{g}$ has scalar curvature $\mathrm{Sc}_{\widetilde{g}} \equiv 0$. However, the positivity of the Yamabe constant $Y(S^n,[g_{st}])$ implies that no such positive function $u$ exists.

\end{proof}

Under stronger topological assumptions, one can obtain characterization results in odd dimensions as well. In dimension $3$, every complete open contractible locally conformally flat manifold with nonnegative scalar curvature is homeomorphic to $\mathbb{R}^3$ \cite[Theorem~E(i)]{2026arXiv260902267D}. In contrast, for every $n\geq 4$, there exists a complete open contractible locally conformally flat $n$-manifold with positive scalar curvature that is not homeomorphic to $\mathbb{R}^n$; moreover, for every $n\geq 6$, such examples can be constructed with uniformly positive scalar curvature \cite[Theorem~F]{2026arXiv260902267D}.

\addcontentsline{toc}{section}{\refname}
\bibliographystyle{alpha}
\bibliography{reference}

\newcommand{\etalchar}[1]{$^{#1}$}
\begin{thebibliography}{MOOP08}

\bibitem[AIL07]{zbMATH05130246}
Kazuo Akutagawa, Masashi Ishida, and Claude LeBrun.
\newblock Perelman's invariant, {Ricci} flow, and the {Yamabe} invariants of
  smooth manifolds.
\newblock {\em Arch. Math.}, 88(1):71--76, 2007.

\bibitem[And06]{MR2213687}
Michael~T. Anderson.
\newblock Canonical metrics on 3-manifolds and 4-manifolds.
\newblock {\em Asian J. Math.}, 10(1):127--163, 2006.

\bibitem[Bam14]{zbMATH06323440}
Richard~H. Bamler.
\newblock Stability of hyperbolic manifolds with cusps under {Ricci} flow.
\newblock {\em Adv. Math.}, 263:412--467, 2014.

\bibitem[BBN{\etalchar{+}}78]{MR484179}
Harold Brown, Rolf B\"ulow, Joachim Neub\"user, Hans Wondratschek, and Hans
  Zassenhaus.
\newblock {\em Crystallographic groups of four-dimensional space}.
\newblock Wiley Monographs in Crystallography. Wiley-Interscience [John Wiley
  \& Sons], New York-Chichester-Brisbane, 1978.

\bibitem[BCG91]{zbMATH04192565}
G.~Besson, G.~Courtois, and S.~Gallot.
\newblock Volume et entropie minimale des espaces localement sym{\'e}triques.
  ({Minimal} volume and entropy of locally symmetric spaces).
\newblock {\em Invent. Math.}, 103(2):417--445, 1991.

\bibitem[BCG95]{zbMATH00847576}
G.~Besson, G.~Courtois, and S.~Gallot.
\newblock Entropy and rigidity of locally symmetric spaces of strictly negative
  curvature.
\newblock {\em Geom. Funct. Anal.}, 5(5):731--799, 1995.

\bibitem[Bet15]{MR3450199}
Renato~Ghini Bettiol.
\newblock {\em On different notions of positivity of curvature}.
\newblock ProQuest LLC, Ann Arbor, MI, 2015.
\newblock Thesis (Ph.D.)--University of Notre Dame.

\bibitem[BK19]{2019arXiv190908710B}
Richard~H. {Bamler} and Bruce {Kleiner}.
\newblock {Ricci flow and contractibility of spaces of metrics}.
\newblock {\em arXiv e-prints}, page arXiv:1909.08710, September 2019.

\bibitem[Bol09]{zbMATH05530571}
Dmitry Bolotov.
\newblock About the macroscopic dimension of certain {PSC}-manifolds.
\newblock {\em Algebr. Geom. Topol.}, 9(1):21--27, 2009.

\bibitem[Car88]{MR936805}
Rodney Carr.
\newblock Construction of manifolds of positive scalar curvature.
\newblock {\em Trans. Amer. Math. Soc.}, 307(1):63--74, 1988.

\bibitem[Cas15]{zbMATH06537658}
Jeffrey~S. Case.
\newblock A {Yamabe}-type problem on smooth metric measure spaces.
\newblock {\em J. Differ. Geom.}, 101(3):467--505, 2015.

\bibitem[CL24]{zbMATH07817078}
Otis Chodosh and Chao Li.
\newblock Generalized soap bubbles and the topology of manifolds with positive
  scalar curvature.
\newblock {\em Ann. Math. (2)}, 199(2):707--740, 2024.

\bibitem[CQY04]{MR2070141}
Sun-Yung~A. Chang, Jie Qing, and Paul~C. Yang.
\newblock On finiteness of {K}leinian groups in general dimension.
\newblock {\em J. Reine Angew. Math.}, 571:1--17, 2004.

\bibitem[CTZ12]{zbMATH06081388}
Bing-Long Chen, Siu-Hung Tang, and Xi-Ping Zhu.
\newblock Complete classification of compact four-manifolds with positive
  isotropic curvature.
\newblock {\em J. Differ. Geom.}, 91(1):41--80, 2012.

\bibitem[Del92]{zbMATH00124945}
Philippe Delanoe.
\newblock Generalized stereographic projections with prescribed scalar
  curvature.
\newblock In {\em Geometry and nonlinear partial differential equations.
  Proceedings of the AMS special session, held at the University of Arkansas,
  Fayetteville, Arkansas, March 23-24, 1990}, pages 17--25. Providence, RI:
  American Mathematical Society, 1992.

\bibitem[Den21a]{zbMATH07342230}
Jialong Deng.
\newblock Curvature-dimension condition meets {Gromov}'s {{\(n\)}}-volumic
  scalar curvature.
\newblock {\em SIGMA, Symmetry Integrability Geom. Methods Appl.}, 17:paper
  013, 20, 2021.

\bibitem[Den21b]{zbMATH07375613}
Jialong Deng.
\newblock Enlargeable length-structure and scalar curvatures.
\newblock {\em Ann. Global Anal. Geom.}, 60(2):217--230, 2021.

\bibitem[Den22]{zbMATH07544449}
Jialong Deng.
\newblock Sphere theorems with and without smoothing.
\newblock {\em J. Geom.}, 113(2):15, 2022.
\newblock Id/No 33.

\bibitem[{Den}26]{2026arXiv260902267D}
Jialong {Deng}.
\newblock {Locally Conformally Flat Manifolds with Positive Scalar Curvature:
  Kleinian Groups, Moduli Spaces, and Euclidean Rigidity}.
\newblock {\em arXiv e-prints}, page arXiv:2609.02267, September 2026.

\bibitem[EE67]{MR212840}
C.~J. Earle and J.~Eells.
\newblock The diffeomorphism group of a compact {R}iemann surface.
\newblock {\em Bull. Amer. Math. Soc.}, 73:557--559, 1967.

\bibitem[Fis04]{zbMATH02063627}
Arthur~E. Fischer.
\newblock An introduction to conformal {Ricci} flow.
\newblock {\em Classical Quantum Gravity}, 21(3):s171--s218, 2004.

\bibitem[GeN14]{MR3192307}
Nadine Gro\ss~e and Marc Nardmann.
\newblock The {Y}amabe constant on noncompact manifolds.
\newblock {\em J. Geom. Anal.}, 24(2):1092--1125, 2014.

\bibitem[GL25]{MR4937973}
Liuwei Gong and Yanyan Li.
\newblock Conformal metrics of constant scalar curvature with unbounded
  volumes.
\newblock {\em Proc. Lond. Math. Soc. (3)}, 131(1):Paper No. e70069, 55, 2025.

\bibitem[Gra73]{MR326773}
Andr\'e Gramain.
\newblock Le type d'homotopie du groupe des diff\'eomorphismes d'une surface
  compacte.
\newblock {\em Ann. Sci. \'Ecole Norm. Sup. (4)}, 6:53--66, 1973.

\bibitem[Gur94]{zbMATH00750638}
Matthew~J. Gursky.
\newblock Locally conformally flat four- and six-manifolds of positive scalar
  curvature and positive {Euler} characteristic.
\newblock {\em Indiana Univ. Math. J.}, 43(3):747--774, 1994.

\bibitem[GV79]{zbMATH03667577}
Alfred Gray and Lieven Vanhecke.
\newblock Riemannian geometry as determined by the volumes of small geodesic
  balls.
\newblock {\em Acta Math.}, 142:157--198, 1979.

\bibitem[ha]{215872}
Ian~Agol (https://mathoverflow.net/users/1345/ian agol).
\newblock For a 3-manifold $y$, when does $y\times s^{1}$ admits a riemannian
  metric with positive scalar curvature?
\newblock MathOverflow.
\newblock URL:https://mathoverflow.net/q/215872 (version: 2017-06-08).

\bibitem[Ham82]{zbMATH03794902}
Richard~S. Hamilton.
\newblock Three-manifolds with positive {Ricci} curvature.
\newblock {\em J. Differ. Geom.}, 17:255--306, 1982.

\bibitem[Heb99]{MR1688256}
Emmanuel Hebey.
\newblock {\em Nonlinear analysis on manifolds: {S}obolev spaces and
  inequalities}, volume~5 of {\em Courant Lecture Notes in Mathematics}.
\newblock New York University, Courant Institute of Mathematical Sciences, New
  York; American Mathematical Society, Providence, RI, 1999.

\bibitem[HKW03]{zbMATH02078321}
B.~Hanke, D.~Kotschick, and J.~Wehrheim.
\newblock Dissolving four-manifolds and positive scalar curvature.
\newblock {\em Math. Z.}, 245(3):545--555, 2003.

\bibitem[Ho24]{zbMATH07926133}
Pak~Tung Ho.
\newblock A rigidity result for the product of spheres.
\newblock {\em Can. Math. Bull.}, 67(3):574--581, 2024.

\bibitem[JM87]{zbMATH04086464}
Dennis Johnson and John~J. Millson.
\newblock Deformation spaces associated to compact hyperbolic manifolds.
\newblock Discrete groups in geometry and analysis, {Pap}. {Hon}. {G}. {D}.
  {Mostow} 60th {Birthday}, {Prog}. {Math}. 67, 48-106 (1987)., 1987.

\bibitem[JT22]{MR4374774}
M.~I. Jimenez and R.~Tojeiro.
\newblock Umbilical submanifolds of {$\Bbb H^k\times\Bbb S^{n-k+1}$}.
\newblock {\em Differential Geom. Appl.}, 81:Paper No. 101862, 19, 2022.

\bibitem[Kap09]{MR2491697}
Michael Kapovich.
\newblock Homological dimension and critical exponent of {K}leinian groups.
\newblock {\em Geom. Funct. Anal.}, 18(6):2017--2054, 2009.

\bibitem[KN26]{KasuyaNoda2026}
Naohiko Kasuya and Issei Noda.
\newblock Classification of orientable torus bundles over closed orientable
  surfaces.
\newblock {\em J. Math. Soc. Japan}, pages 1--41, 2026.
\newblock Advance Publication.

\bibitem[KS25]{zbMATH08137326}
Aditya Kumar and Balarka Sen.
\newblock Circle bundles with {PSC} over large manifolds.
\newblock {\em Geom. Funct. Anal.}, 35(5):1400--1423, 2025.

\bibitem[KSX25]{KazarasSongXu2025}
Demetre Kazaras, Antoine Song, and Kai Xu.
\newblock Scalar curvature and volume entropy of hyperbolic 3-manifolds.
\newblock {\em J. Eur. Math. Soc.}, 2025.

\bibitem[LeB03]{zbMATH01985408}
Claude LeBrun.
\newblock Scalar curvature, covering spaces, and {Seiberg}-{Witten} theory.
\newblock {\em New York J. Math.}, 9:93--97, 2003.

\bibitem[LM08]{zbMATH05319680}
Claude Lebrun and Bernard Maskit.
\newblock On optimal 4-dimensional metrics.
\newblock {\em J. Geom. Anal.}, 18(2):537--564, 2008.

\bibitem[Loh99]{zbMATH01295406}
Joachim Lohkamp.
\newblock Scalar curvature and hammocks.
\newblock {\em Math. Ann.}, 313(3):385--407, 1999.

\bibitem[LUY24a]{MR4773185}
Martin Lesourd, Ryan Unger, and Shing-Tung Yau.
\newblock The positive mass theorem with arbitrary ends.
\newblock {\em J. Differential Geom.}, 128(1):257--293, 2024.

\bibitem[LUY24b]{MR4836036}
Martin Lesourd, Ryan Unger, and Shing-Tung Yau.
\newblock Positive scalar curvature on noncompact manifolds and the {L}iouville
  theorem.
\newblock {\em Comm. Anal. Geom.}, 32(5):1311--1337, 2024.

\bibitem[Mac19]{MR3920349}
Heather~R. Macbeth.
\newblock Conformal classes realizing the {Y}amabe invariant.
\newblock {\em Int. Math. Res. Not. IMRN}, (5):1333--1349, 2019.

\bibitem[MOOP08]{zbMATH05379509}
Michel Matthey, Herv{\'e} Oyono-Oyono, and Wolfgang Pitsch.
\newblock Homotopy invariance of higher signatures and 3-manifold groups.
\newblock {\em Bull. Soc. Math. Fr.}, 136(1):1--25, 2008.

\bibitem[MP91]{1991NuPhB.361..290M}
Robert {Myers} and Vipul {Periwal}.
\newblock {Invariants of smooth 4-manifolds from topological gravity}.
\newblock {\em Nuclear Physics B}, 361(1):290--310, January 1991.

\bibitem[MT21]{zbMATH07432160}
Agnese Mantione and Rafael Torres.
\newblock Geography of 4-manifolds with positive scalar curvature.
\newblock {\em Expo. Math.}, 39(4):566--582, 2021.

\bibitem[Nor93]{MR1235219}
Maria~Helena Noronha.
\newblock Some compact conformally flat manifolds with nonnegative scalar
  curvature.
\newblock {\em Geom. Dedicata}, 47(3):255--268, 1993.

\bibitem[Oba71]{zbMATH03374588}
Morio Obata.
\newblock The conjectures on conformal transformations of {Riemannian}
  manifolds.
\newblock {\em J. Differ. Geom.}, 6:247--258, 1971.

\bibitem[Rei09]{zbMATH05607319}
Martin Reiris.
\newblock Energy and volume: {A} proof of the positivity of {ADM} energy using
  the yamabe invariant of three-manifolds.
\newblock {\em Commun. Math. Phys.}, 287(1):383--393, 2009.

\bibitem[Rez97]{zbMATH01054076}
Alexander~G. Reznikov.
\newblock Yamabe spectra.
\newblock {\em Duke Math. J.}, 89(1):87--94, 1997.

\bibitem[Ros07]{zbMATH05342785}
Jonathan Rosenberg.
\newblock Manifolds of positive scalar curvature: a progress report.
\newblock In {\em Metric and comparison geometry. Surveys in differential
  geometry. Vol. XI.}, pages 259--294. Somerville, MA: International Press,
  2007.

\bibitem[RS01]{MR1818778}
Jonathan Rosenberg and Stephan Stolz.
\newblock Metrics of positive scalar curvature and connections with surgery.
\newblock In {\em Surveys on surgery theory, {V}ol. 2}, volume 149 of {\em Ann.
  of Math. Stud.}, pages 353--386. Princeton Univ. Press, Princeton, NJ, 2001.

\bibitem[Sch89]{MR994021}
Richard~M. Schoen.
\newblock Variational theory for the total scalar curvature functional for
  {R}iemannian metrics and related topics.
\newblock In {\em Topics in calculus of variations ({M}ontecatini {T}erme,
  1987)}, volume 1365 of {\em Lecture Notes in Math.}, pages 120--154.
  Springer, Berlin, 1989.

\bibitem[Sch04]{zbMATH02104390}
Thomas Schick.
\newblock Real versus complex {K}-theory using {Kasparov}'s bivariant
  {KK}-theory.
\newblock {\em Algebr. Geom. Topol.}, 4:333--346, 2004.

\bibitem[Sul84]{zbMATH03903608}
Dennis Sullivan.
\newblock Entropy, {Hausdorff} measures old and new, and limit sets of
  geometrically finite {Kleinian} groups.
\newblock {\em Acta Math.}, 153:259--277, 1984.

\bibitem[SY87]{zbMATH04075990}
Richard Schoen and Shing-Tung Yau.
\newblock The structure of manifolds with positive scalar curvature.
\newblock Directions in partial differential equations, {Proc}. {Symp}.,
  {Madison}/{Wis}. 1985, {Publ}. {Math}. {Res}. {Cent}. {Univ}. {Wis}.
  {Madison} 54, 235-242 (1987)., 1987.

\bibitem[SY88]{zbMATH04075988}
Richard Schoen and Shing-Tung Yau.
\newblock Conformally flat manifolds, {Kleinian} groups and scalar curvature.
\newblock {\em Invent. Math.}, 92(1):47--71, 1988.

\bibitem[Tei99]{MR1720873}
M.~Teicher.
\newblock Hirzebruch surfaces: degenerations, related braid monodromy, {G}alois
  covers.
\newblock In {\em Algebraic geometry: {H}irzebruch 70 ({W}arsaw, 1998)}, volume
  241 of {\em Contemp. Math.}, pages 305--325. Amer. Math. Soc., Providence,
  RI, 1999.

\bibitem[Ye93]{zbMATH00427933}
Rugang Ye.
\newblock Ricci flow, {Einstein} metrics and space forms.
\newblock {\em Trans. Am. Math. Soc.}, 338(2):871--896, 1993.

\end{thebibliography}
%\Addresses

\end{document}